\documentclass[a4paper,11pt,reqno]{amsart}

\usepackage[margin=1in,includehead,includefoot]{geometry}
\usepackage[utf8]{inputenc}
\usepackage[T1]{fontenc}
\usepackage{lmodern,microtype}
\usepackage{upref}
\usepackage{mathtools,amssymb,amsthm}
\usepackage[foot]{amsaddr}
\usepackage{enumitem}
\usepackage[disable]{todonotes}
\usepackage{tikz}
\usepackage{graphicx}
\usepackage{wrapfig}
\usepackage{hyperref}
\usepackage[margin=1cm]{caption}
\usepackage{subcaption}
\usepackage{bookmark}
\usepackage{mycommands}

\graphicspath{{./graphics/}}

\makeatletter
\let\old@setaddresses\@setaddresses
\def\@setaddresses{\bigskip{\parindent 0pt\let\scshape\relax\let\ttfamily\relax\old@setaddresses}}
\makeatother

\newtheorem{theorem}{Theorem}

\newtheorem{lemma}[theorem]{Lemma}
\newtheorem{conjecture}[theorem]{Conjecture}

\theoremstyle{remark}


\hypersetup{
  pdftitle={Star transposition Gray codes for multiset permutations}
  pdfauthor={Petr Gregor, Arturo Merino, Torsten M\"utze}
}

\title{Star transposition Gray codes for multiset permutations}

\author{Petr Gregor}
\address[Petr Gregor]{Department of Theoretical Computer Science and Mathematical Logic, Charles University, Prague, Czech Republic}
\email{gregor@ktiml.mff.cuni.cz}

\author{Arturo Merino}
\address[Arturo Merino]{Department of Mathematics, TU Berlin, Germany}
\email{merino@math.tu-berlin.de}

\author{Torsten M\"utze}
\address[Torsten M\"utze]{Department of Computer Science, University of Warwick, United Kingdom \& Department of Theoretical Computer Science and Mathematical Logic, Charles University, Prague, Czech Republic}
\email{torsten.mutze@warwick.ac.uk}

\thanks{An extended abstract of this paper was accepted for presentation at the 39th International Symposium on Theoretical Aspects of Computer Science (STACS 2022).}
\thanks{This work was supported by Czech Science Foundation grant GA~19-08554S, and by German Science Foundation grant~413902284.}

\begin{document}

\begin{abstract}
Given integers $k\geq 2$ and $a_1,\ldots,a_k\geq 1$, let $\ba:=(a_1,\ldots,a_k)$ and $n:=a_1+\cdots+a_k$.
An $\ba$-multiset permutation is a string of length~$n$ that contains exactly $a_i$ symbols~$i$ for each $i=1,\ldots,k$.
In this work we consider the problem of exhaustively generating all $\ba$-multiset permutations by star transpositions, i.e., in each step, the first entry of the string is transposed with any other entry distinct from the first one.
This is a far-ranging generalization of several known results.
For example, it is known that permutations ($a_1=\cdots=a_k=1$) can be generated by star transpositions, while combinations ($k=2$) can be generated by these operations if and only if they are balanced ($a_1=a_2$), with the positive case following from the middle levels theorem.
To understand the problem in general, we introduce a parameter~$\Delta(\ba):=n-2\max\{a_1,\ldots,a_k\}$ that allows us to distinguish three different regimes for this problem.
We show that if $\Delta(\ba)<0$, then a star transposition Gray code for $\ba$-multiset permutations does not exist.
We also construct such Gray codes for the case $\Delta(\ba)>0$, assuming that they exist for the case~$\Delta(\ba)=0$.
For the case~$\Delta(\ba)=0$ we present some partial positive results.
Our proofs establish Hamilton-connectedness or Hamilton-laceability of the underlying flip graphs, and they answer several cases of a recent conjecture of Shen and Williams.
In particular, we prove that the middle levels graph is Hamilton-laceable.
\end{abstract}

\keywords{Gray code, permutation, combination, transposition, Hamilton cycle}

\maketitle

\section{Introduction}

Permutations and combinations are two of the most fundamental classes of combinatorial objects.
Specifically, \emph{$k$-permutations} are all linear orderings of $[k]:=\{1,\ldots,k\}$, and their number is~$k!$.
Moreover, \emph{$(\alpha,\beta)$-combinations} are all $\beta$-element subsets of~$[n]$ where $n:=\alpha+\beta$, and their number is $\binom{n}{\alpha}=\binom{n}{\beta}$.
Permutations and combinations are generalized by so-called multiset permutations, and in this paper we consider the task of listing them such that any two consecutive objects in the list differ by particular transpositions, i.e., by swapping two elements.
Such a listing of objects subject to a `small change' operation is often referred to as \emph{Gray code}~\cite{MR1491049,MR3523863}.
One of the standard references for algorithms that efficiently generate various combinatorial objects, including permutations and combinations, is Knuth's book~\cite{MR3444818} (see also~\cite{MR0396274}).

\subsection{Permutation generation}
\label{sec:perm}

There is a vast number of Gray codes for permutation generation, most prominently the Steinhaus-Johnson-Trotter algorithm~\cite{DBLP:journals/cacm/Trotter62,MR0159764}, which generates all $k$-permutations by adjacent transpositions, i.e., swaps of two neighboring entries of the permutation; see Figure~\ref{fig:perm}~(a).
In this work, we focus on \emph{star transpositions}, i.e., swaps of the first entry of the permutation with any later entry.
An efficient algorithm for generating permutations by star transpositions was found by Ehrlich, and it is described as Algorithm~E in Knuth's book~\cite[Section~7.2.1.2]{MR3444818}; see Figure~\ref{fig:perm}~(b).
For any permutation generation algorithm based on transpositions, we can define the \emph{transposition graph} as the graph with vertex set~$[k]$, and an edge between~$i$ and~$j$ if the algorithm uses transpositions between the $i$th and $j$th entry of the permutation.
Clearly, the transposition graph for adjacent transpositions is a path, whereas the transposition graph for star transpositions is a star (hence the name `star transposition').
In fact, Kompel'maher and Liskovec~\cite{MR0498276}, and independently Slater~\cite{MR504868}, showed that all $k$-permutations can be generated for any transposition tree on~$[k]$.
Transposition Gray codes for permutations with additional restrictions were studied by Compton and Williamson~\cite{MR1308693} and by Shen and Williams~\cite{DBLP:journals/endm/ShenW13}.

\begin{wrapfigure}{r}{0.5\textwidth}
\includegraphics{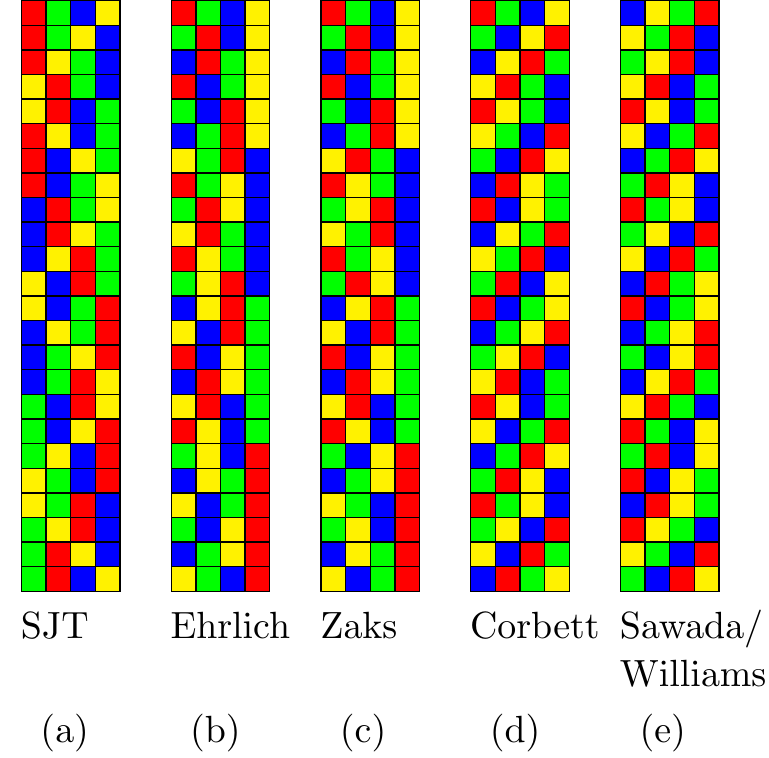}
\caption{Gray codes for 4-permutations (SJT=Steinhaus-Johnson-Trotter).}
\vspace{-3mm}
\label{fig:perm}
\end{wrapfigure}
Several known algorithms for permutation generation use operations other than transpositions.
Specifically, Zaks~\cite{MR753548} presented an algorithm for generating permutations by prefix reversals; see Figure~\ref{fig:perm}~(c).
Moreover, Corbett~\cite{corbett_1992} showed that all $k$-permutations can be generated by cyclic left shifts of any prefix of the permutation by one position; see Figure~\ref{fig:perm}~(d).
Another notable result is Sawada and Williams' recent solution~\cite{MR4060409} of the Sigma-Tau problem, proving that all $k$-permutations can be generated by cyclic left shifts of the entire permutation by one position or transpositions of the first two elements; see Figure~\ref{fig:perm}~(e).

All of the aforementioned results can be seen as explicit constructions of Hamilton paths in the Cayley graph of the symmetric group, generated by different sets of generators (transpositions, reversals, or shifts).
It is an open problem whether the Cayley graph of the symmetric group has a Hamilton path for any set of generators~\cite{MR1201997}.
This is a special case of the well-known open problem whether any connected Cayley graph has a Hamilton path, or even more generally, whether this is the case for any vertex-transitive graph~\cite{MR0263646}.

\subsection{Combination generation and the middle levels conjecture}
\label{sec:comb}

In a computer, $(\alpha,\beta)$-combinations can be conveniently represented by bitstrings of length~$n:=\alpha+\beta$, where the $i$th bit is~1 if the element~$i$ is in the set and~0 otherwise.
For example, the $(5,3)$-combination $\{1,6,7\}$ is represented by the string~$10000110$.

In the 1980s, Buck and Wiedemann~\cite{MR737262} conjectured that all $(\alpha,\alpha)$-combinations can be generated by star transpositions for every $\alpha\geq 1$, i.e., in every step we swap the first bit of the bitstring representation with a later bit.
Figure~\ref{fig:mperm}~(a) shows such a star transposition Gray code for $(4,4)$-combinations.
Buck and Wiedemann's conjecture was raised independently by Havel~\cite{MR737021}, as a question about the existence of a Hamilton cycle through the middle two levels of the $(2\alpha-1)$-dimensional hypercube.
This conjecture became known as \emph{middle levels conjecture}, and it attracted considerable attention in the literature and made its way into popular books~\cite{MR2034896,MR2858033}, until it was answered affirmatively by M\"utze~\cite{MR3483129}; see also~\cite{gregor-muetze-nummenpalo:18}.

Similarly to permutations, there are also many known methods for generating general $(\alpha,\beta)$-combinations that use operations other than star transpositions, see~\cite{MR0349274,MR782221,MR995888,MR1352777,MR2548545}.
In particular, $(\alpha,\beta)$-combinations can be generated by adjacent transpositions if and only if $\alpha=1$, $\beta=1$, or $\alpha$ and~$\beta$ are both odd~\cite{MR737262,MR821383,MR936104}.

\subsection{Multiset permutations}

\begin{figure}
\centerline{
\begin{tabular}{cc}
\includegraphics{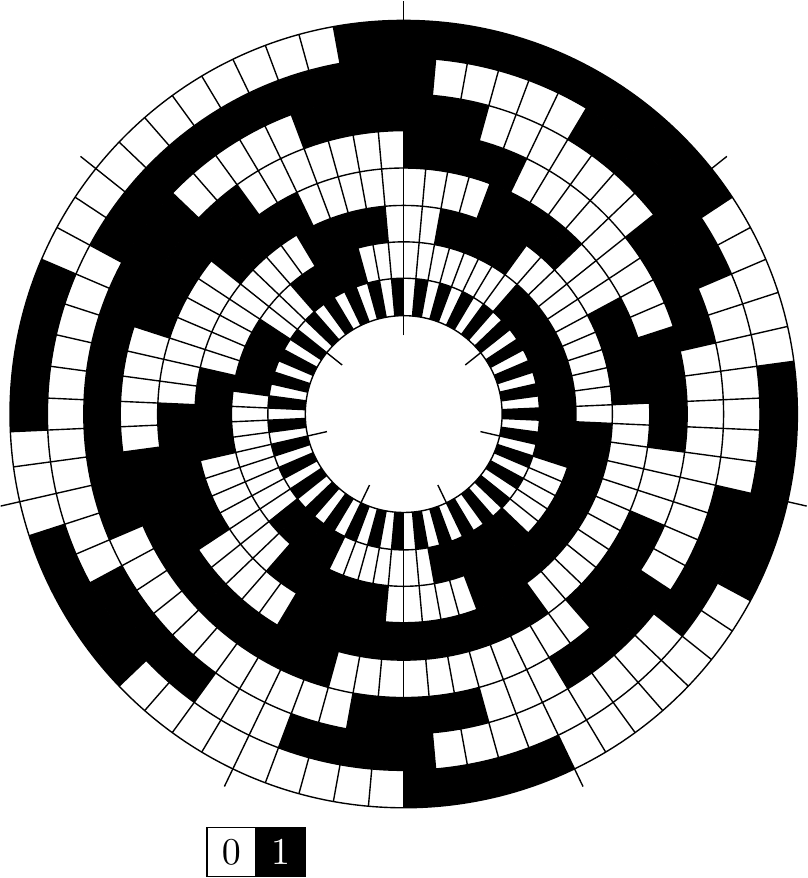} &
\includegraphics{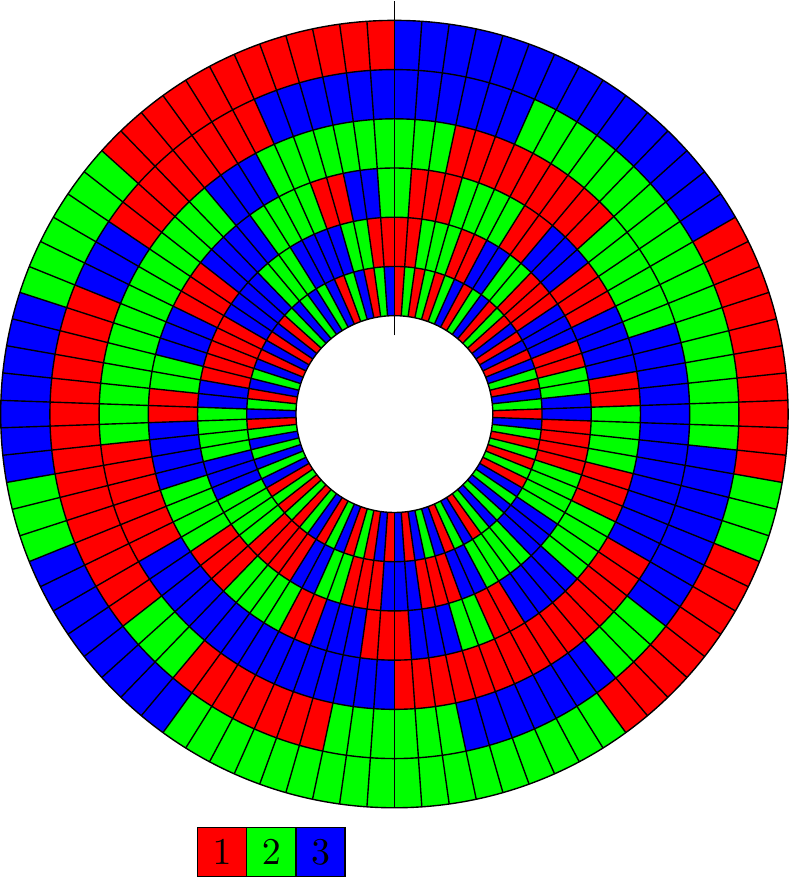} \\
(a) $(a_1,a_2)=(4,4)$ & (b) $(a_1,a_2,a_3)=(2,2,2)$
\end{tabular}
}
\caption{Star transposition Gray codes for (a) $(4,4)$- and (b) $(2,2,2)$-multiset permutations.
The strings are arranged in clockwise order, starting at 12 o'clock, with the first entry on the inner track, and the last entry on the outer track.
As every star transposition changes the first entry, the color on the inner track changes in every step.}
\label{fig:mperm}
\end{figure}

Shen and Williams~\cite{shen_williams_2021} proposed a far-ranging generalization of the middle levels conjecture that connects permutations and combinations.
Their conjecture is about multiset permutations.
For integers $k\geq 2$ and $a_1,\ldots,a_k\geq 1$, an \emph{$(a_1,\ldots,a_k)$-multiset permutation} is a string over the alphabet~$\{1,\ldots,k\}$ that contains exactly~$a_i$ occurrences of the symbol~$i$.
We refer to the sequence $\ba:=(a_1,\ldots,a_k)$ as the \emph{frequency vector}, as it specifies the frequency of each symbol.
The length of a multiset permutation is $n_\ba:=a_1+\cdots+a_k$, and if the context is clear we omit the index and simply write $n=n_\ba$.
If all symbols appear equally often, i.e., $a_1=\cdots=a_k=\alpha$, we use the abbreviation $\alpha^k:=(a_1,\ldots,a_k)$.
For example $123433153$ is a $(2,1,4,1)$-multiset permutation, and $331232142144$ is a $3^4$-multiset permutation.

Clearly, multiset permutations are a generalization of permutations and combinations.
Specifically, $k$-permutations are $1^k$-multiset permutations, and $(\alpha,\beta)$-combinations are $(\alpha,\beta)$-multiset permutations (up to shifting the symbol names $1,2\mapsto 0,1$).
Stachowiak~\cite{MR1157583} showed that $(a_1,\ldots,a_k)$-multiset permutations can be generated by adjacent transpositions if and only if at least two of the $a_i$ are odd.

Shen and Williams~\cite{shen_williams_2021} conjectured that all $\alpha^k$-multiset permutations can be generated by star transpositions, for any $\alpha\geq 1$ and~$k\geq 2$.
We state their conjecture in terms of Hamilton cycles in a suitably defined graph, as follows.
We write $\Pi(\ba)=\Pi(a_1,\ldots,a_k)$ for the set of all $(a_1,\ldots,a_k)$-multiset permutations.
Moreover, we let $G(\ba)=G(a_1,\ldots,a_k)$ denote the graph on the vertex set~$\Pi(\ba)=\Pi(a_1,\ldots,a_k)$ with an edge between any two multiset permutations that differ in a star transposition, i.e., in swapping the first entry of the multiset permutation with any entry at positions~$2,\ldots,n$ that is distinct from the first one.
Figure~\ref{fig:Ga} shows various examples of the graph~$G(\ba)$.
When denoting specific multiset permutations we sometimes omit commas and brackets for brevity, for example $1312214\in\Pi(3,2,1,1)$.

\begin{conjecture}[\cite{shen_williams_2021}]
\label{conj:SW}
For any $\alpha\geq 1$ and~$k\geq 2$, the graph $G(\alpha^k)$ has a Hamilton cycle.
\end{conjecture}

In this and the following statements, the single edge~$G(1,1)$ is also considered a cycle, as it gives a cyclic Gray code.
Note that $G(a_1,\ldots,a_k)$ is vertex-transitive if and only if $a_1=\cdots=a_k=:\alpha$.
In this case, Conjecture~\ref{conj:SW} is an interesting instance of the aforementioned conjecture of Lov{\'a}sz~\cite{MR0263646} on Hamilton paths in vertex-transitive graphs.

Evidence for Conjecture~\ref{conj:SW} comes from the results mentioned in Sections~\ref{sec:perm} and~\ref{sec:comb} on generating permutations by star transpositions and the solution of the middle levels conjecture, respectively, formulated in terms of the graph~$G(\ba)$ below.
These known results settle the boundary cases $\alpha=1$ and $k\geq 2$, and $\alpha\geq 1$ and $k=2$, respectively, of Conjecture~\ref{conj:SW}.

\begin{figure}
\includegraphics[page=1]{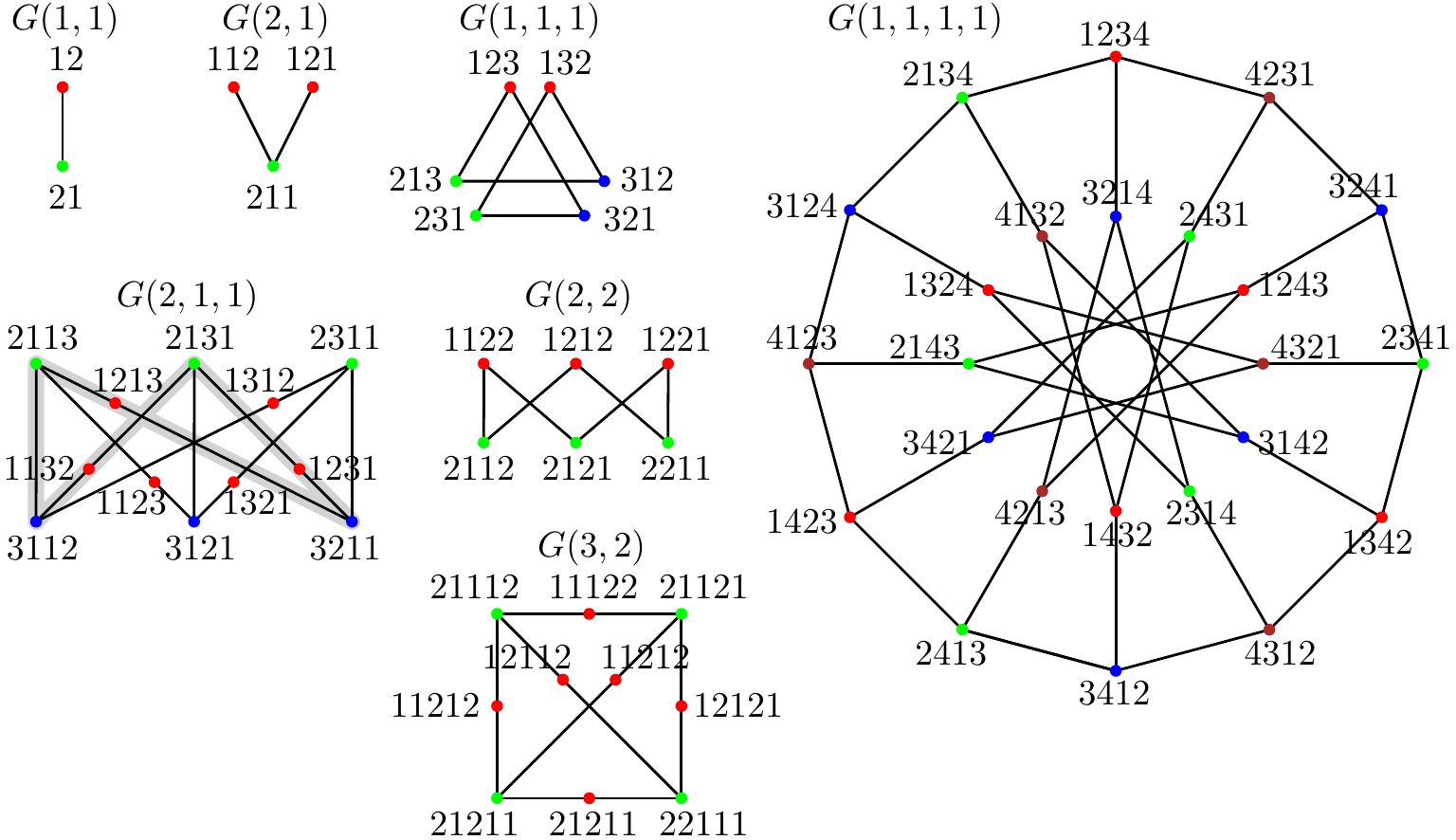}
\caption{Star transposition graphs~$G(\ba)$ for several small multiset permutations~$\ba$.
Vertices are colored according to the first entry of the multiset permutations, and these color classes form independent sets.
In $G(2,1,1)$, an odd cycle is highlighted.
}
\label{fig:Ga}
\end{figure}

\begin{theorem}[Ehrlich; \cite{MR0498276}; \cite{MR504868}]
\label{thm:ehrlich}
For any $k\geq 2$, the graph~$G(1^k)$ has a Hamilton cycle.
\end{theorem}

\begin{theorem}[\cite{MR3483129,gregor-muetze-nummenpalo:18}]
\label{thm:mlc}
For any $\alpha\geq 1$, the graph $G(\alpha,\alpha)$ has a Hamilton cycle.
\end{theorem}

In their paper, Shen and Williams also provided an ad-hoc solution for the first case of their conjecture that is not covered by Theorems~\ref{thm:ehrlich} and~\ref{thm:mlc}, namely a Hamilton cycle in $G(2,2,2)$, which is displayed in Figure~\ref{fig:mperm}~(b).

We approach Conjecture~\ref{conj:SW} by tackling the following even more general question:
For which frequency vectors $\ba=(a_1,\ldots,a_k)$ does the graph~$G(\ba)$ have a Hamilton cycle?
By renaming symbols, we may assume w.l.o.g.\ that the entries of the vector~$\ba$ are non-increasing, i.e.,
\begin{equation}
\label{eq:ai-decreasing}
a_1\geq a_2\geq \cdots\geq a_k.
\end{equation}
We can thus think of the vector~$\ba$ as an integer partition of~$n$.

\subsection{Our results}
\label{sec:results}

For any~$i\in[n]$ and $c\in[k]$, we write $\Pi(\ba)^{i,c}$ for the set of all multiset permutations from~$\Pi(\ba)$ whose $i$th symbol equals~$c$.
Note that every star transposition changes the first entry; see the inner track of each of the two wheels in Figure~\ref{fig:mperm}.
As a consequence, $G(\ba)$ is a $k$-partite graph with partition classes $\Pi(\ba)^{1,1},\ldots,\Pi(\ba)^{1,k}$; see Figure~\ref{fig:Ga}.
Moreover, the partition class~$\Pi(\ba)^{1,1}$ is a largest one because of~\eqref{eq:ai-decreasing}.
This $k$-partition of the graph~$G(\ba)$ is a potential obstacle for the existence of Hamilton cycles and paths.
Specifically, if one partition class is larger than all others combined, then there cannot be a Hamilton cycle, and if the size difference is more than~1, then there cannot be a Hamilton path.

We capture this by defining a parameter~$\Delta(\ba)$ for any integer partition~$\ba=(a_1,\ldots,a_k)$ as
\begin{equation}
\label{eq:Delta}
\Delta(\ba):=n-2a_1=\Big(\sum_{i=2}^k a_i\Big)-a_1.
\end{equation}
We will see that if $\Delta(\ba)<0$, then the partition class~$\Pi(\ba)^{1,1}$ of the graph~$G(\ba)$ is larger than all others combined, excluding the existence of Hamilton cycles.
On the other hand, if $\Delta(\ba)\geq 0$, then every partition class of the graph~$G(\ba)$ is at most as large as all others combined (equality holds if $\Delta(\ba)=0$), which does not exclude the existence of a Hamilton cycle.
The cases with $\Delta(\ba)=0$ lie on the boundary between the two regimes, and they are the hardest in terms of proving Hamiltonicity.
These cases can be seen as generalizations of the middle levels conjecture, namely the case $\ba=(\alpha,\alpha)$ captured by Theorem~\ref{thm:mlc}, which also satisfies~$\Delta(\ba)=0$.

\begin{theorem}
\label{thm:N}
For any integer partition~$\ba=(a_1,\ldots,a_k)$ with $\Delta(\ba)<0$ the graph~$G(\ba)$ does not have a Hamilton cycle, and it does not have a Hamilton path unless $\ba=(2,1)$.
\end{theorem}

For $k=2$ symbols, the condition~$\Delta(\ba)<0$ is equivalent to $a_1>a_2$, i.e., there is no star transposition Gray code for `unbalanced' combinations.

We now discuss the cases~$\Delta(\ba)\geq 0$.
Our first main goal is to reduce all cases with~$\Delta(\ba)>0$ to cases with~$\Delta(\ba)=0$.
For doing so, it is helpful to consider stronger notions of Hamiltonicity.
Specifically, we consider Hamilton-connectedness and Hamilton-laceability, which have been heavily studied (see~\cite{MR527987,MR641233,MR945221,MR1796239,MR1848324,MR2252748}).
A graph is called \emph{Hamilton-connected} if there is a Hamilton path between any two distinct vertices.
A bipartite graph is called \emph{Hamilton-laceable} if there is a Hamilton path between any pair of vertices from the two partition classes.
In general, the graphs~$G(\ba)$ are not bipartite, so we say that~$G(\ba)$ is \emph{1-laceable} if there is a Hamilton path between any vertex in~$\Pi(\ba)^{1,1}$ and any vertex not in~$\Pi(\ba)^{1,1}$, i.e., between any vertex with first symbol~1 and any vertex with first symbol distinct from~1.

This approach is inspired by the following result of Tchuente~\cite{MR683982}, who strengthened Theorem~\ref{thm:ehrlich} considerably.

\begin{theorem}[\cite{MR683982}]
\label{thm:T}
For any $k\geq 4$, the graph~$G(1^k)$ is Hamilton-laceable.
\end{theorem}

%This result holds even more generally for arbitrary transposition trees.
The key insight is that proving a stronger property makes the proof easier and shorter, because the inductive statement is more powerful and flexible; see Section~\ref{sec:ideasP} below.
Encouraged by this, we raise the following conjecture about graphs~$G(\ba)$ with $\Delta(\ba)=0$.
It is another natural and far-ranging generalization of the middle levels conjecture, which we support by extensive computer experiments and by proving some special cases.

\begin{conjecture}
\label{conj:0}
For any integer partition~$\ba=(a_1,\ldots,a_k)$ with $\Delta(\ba)=0$ the graph $G(\ba)$ is Hamilton-1-laceable, unless $\ba=(2,2)$.
\end{conjecture}

The exceptional graph~$G(2,2)$ mentioned in this conjecture is a 6-cycle; see Figure~\ref{fig:Ga}.
Assuming the validity of this conjecture, we settle all cases~$G(\ba)$ with $\Delta(\ba)>0$ in the strongest possible sense.
While being a conditional result, the main purpose of this theorem is to reduce all cases~$\Delta(\ba)\geq 0$ to the boundary cases~$\Delta(\ba)=0$.

\begin{theorem}
\label{thm:P}
Conditional on Conjecture~\ref{conj:0}, for any integer partition~$\ba=(a_1,\ldots,a_k)$ with $\Delta(\ba)>0$ the graph~$G(\ba)$ is Hamilton-connected, unless $\ba=(1,1,1)$ or $\ba=1^k$ for $k\geq 4$, and possibly unless $\ba=(\alpha,\alpha,1)$ for $\alpha\geq 3$.
\end{theorem}

The dependence of Theorem~\ref{thm:P} on Conjecture~\ref{conj:0} can be captured more precisely.
Specifically, $G(\ba)$ with $\Delta(\ba)>0$ is shown to be Hamilton-connected, assuming that $G(\bb)$ with $\Delta(\bb)=0$ is Hamilton-1-laceable for all integer partitions~$\bb$ that are majorized componentwise by~$\ba$.

The three exceptions mentioned in Theorem~\ref{thm:P} are well understood:
Specifically, $G(1,1,1)$ is a 6-cycle; see Figure~\ref{fig:Ga}.
Furthermore, $G(1^k)$ for $k\geq 4$ is Hamilton-laceable by Theorem~\ref{thm:T}.
Lastly, we will show that $G(\alpha,\alpha,1)$ for $\alpha\geq 3$ satisfies a variant of Hamilton-laceability, which also guarantees a Hamilton cycle.
In fact, we believe that $G(\alpha,\alpha,1)$ is Hamilton-connected, but we cannot prove it.

We provide the following evidence for Conjecture~\ref{conj:0}.
First of all, with computer help we verified that $G(\ba)$ is indeed Hamilton-1-laceable for all integer partitions~$\ba\neq (2,2)$ with $\Delta(\ba)=0$ that satisfy $n\leq 8$, i.e., for $\ba\in\{(1,1),(2,1,1),(3,3),(3,2,1),(3,1,1,1),(4,4),(4,3,1),(4,2,2),\allowbreak (4,2,1,1),(4,1,1,1,1)\}$.
Furthermore, we prove the case of $k=2$ symbols unconditionally.
Note that for $k=2$, Hamilton-1-laceability is the same as Hamilton-laceability.
Recall that $G(\alpha,\alpha)$ is isomorphic to the subgraph of the $(2\alpha-1)$-dimensional hypercube induced by the middle two levels, so the following result is a considerable strengthening of Theorem~\ref{thm:mlc}, the middle levels theorem.

\begin{theorem}
\label{thm:M}
For any $\alpha\geq 3$, the graph~$G(\alpha,\alpha)$ is Hamilton-laceable.
\end{theorem}

We also have the following (unconditional) result for $k=3$ symbols.

\begin{theorem}
\label{thm:M'}
For any $\alpha\geq 2$, the graph~$G(\alpha,\alpha-1,1)$ has a Hamilton cycle.
\end{theorem}

Lastly, we consider integer partitions $\ba=(a_1,\ldots,a_k)$, $k\geq 3$, with $\Delta(\ba)\geq 0$ and an upper bound on the part size, i.e., $a_1\leq \alpha$ for some constant~$\alpha$.
By the remarks after Theorem~\ref{thm:P}, the inductive proof of the theorem for such integer partitions only relies on Conjecture~\ref{conj:0} being satisfied for integer partitions with the same upper bound on the part size.
For any fixed bound~$\alpha$, there are only finitely many such partitions with~$\Delta(\ba)=0$ that can be checked by computer.
For example, for $\alpha=4$ these are $\ba\in\{(2,1,1),(3,2,1),(3,1,1,1),(4,3,1),(4,2,2),(4,2,1,1),(4,1,1,1,1)\}$.
This yields the following (unconditional) result.

\begin{theorem}
\label{thm:234}
For $\alpha\in\{2,3,4\}$ and any integer partition~$\ba=(a_1,\ldots,a_k)$ with $\Delta(\ba)>0$ and $a_1=\alpha$, the graph~$G(\ba)$ is Hamilton-connected.
\end{theorem}

In words, Theorem~\ref{thm:234} settles all integer partitions~$\ba$ whose largest part~$a_1$ is at most~4.
In particular, this settles the cases $\alpha\in\{2,3,4\}$ and $k\geq 2$ of Shen and Williams' Conjecture~\ref{conj:SW} in a rather strong sense.

\subsection{Proof ideas}

In this section, we give a high-level overview of the main ideas and techniques used in our proofs.

\subsubsection{The case $\Delta(\ba)<0$}
\label{sec:ideasN}

The main idea for proving Theorem~\ref{thm:N} is that if $\Delta(\ba)<0$, then the partition class~$\Pi(\ba)^{1,1}$ of the graph~$G(\ba)$ is larger than all others combined, which excludes the existence of a Hamilton cycle.
To exclude the existence of a Hamilton path, we show that the size difference is strictly more than~1, unless $\ba=(2,1)$.
Note that the graph~$G(2,1)$ is the path on three vertices, so in this case the size difference is precisely~1.
These arguments are based on straightforward algebraic manipulations involving multinomial coefficients.

\subsubsection{The case $\Delta(\ba)>0$}
\label{sec:ideasP}

To prove Theorem~\ref{thm:P}, it is convenient to think of an integer partition~$\ba=(a_1,\ldots,a_k)$ as an infinite non-increasing sequence $(a_1,a_2,\ldots)$, with only $k$ nonzero entries at the beginning.
Given two such integer partitions $\ba=(a_1,a_2,\ldots)$ and $\bb=(b_1,b_2,\ldots)$, we write $\bb\prec\ba$ if $b_i\leq a_i$ for all $i\geq 1$.
Integer partitions with the partial order~$\prec$ form a lattice, which is the sublattice of the infinite lattice $\mathbb{N}^\mathbb{N}$ cut out by the hyperplanes defined by~\eqref{eq:ai-decreasing}; see Figure~\ref{fig:lattice}.
The cover relations in this lattice are given by decrementing any of the~$a_i$ for which $a_i>a_{i+1}$.
We write $\bb\precdot\ba$ for partitions~$\ba\neq \bb$ if $\bb\prec\ba$ and there is no $\bc\notin\{\ba,\bb\}$ with $\bb\prec\bc\prec\ba$.

\begin{figure}
\includegraphics{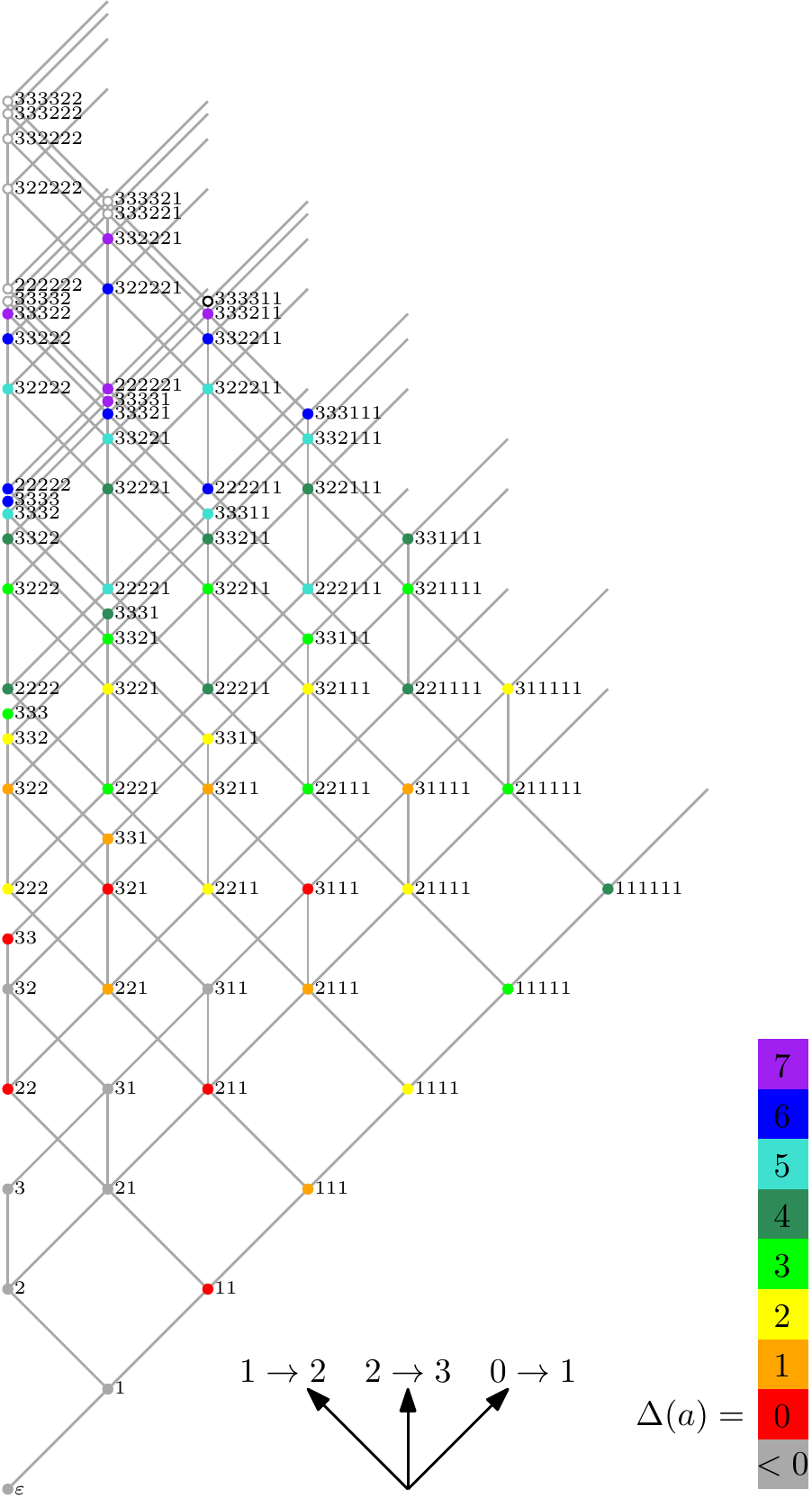}
\caption{The lattice of integer partitions~$\ba=(a_1,\ldots,a_k)$ with largest part~$a_1\leq 3$.
The coordinates are projected into three dimensions depending on which value is increased.}
\label{fig:lattice}
\end{figure}

In this lattice of integer partitions, the hyperplane defined by $\Delta(\ba)=0$ separates the cases where Hamiltonicity is impossible, which lie on the side of the hyperplane where $\Delta(\ba)<0$ (Theorem~\ref{thm:N}), from the cases where Hamiltonicity can be established more easily, which lie on the side of the hyperplane where $\Delta(\ba)>0$ (Theorem~\ref{thm:P}).
The cases $\Delta(\ba)=0$ on the hyperplane are the hardest ones (Conjecture~\ref{conj:0}).

Our proof of Theorem~\ref{thm:P} proceeds by induction in this partition lattice and establishes the Hamiltonicity of~$G(\ba)$ by using the Hamiltonicity of~$G(\bb)$ for all integer partitions~$\bb\precdot\ba$, where Conjecture~\ref{conj:0} serves as the base case of the induction.
This is based on the observation that fixing one of the symbols at positions~$2,\ldots,n$ in~$G(\ba)$ yields subgraphs that are isomorphic to~$G(\bb)$ for~$\bb\precdot\ba$.

Specifically, for any $\ba=(a_1,\ldots,a_k)$, $i=2,\ldots,n$, and $c\in[k]$, the subgraph of $G(\ba)$ induced by the vertex set~$\Pi(\ba)^{i,c}$ is isomorphic to~$G(\bb)$ where $\bb$ is the partition obtained from~$\ba$ by decreasing $a_i$ by~1 (and possibly sorting the resulting numbers non-increasingly); see Figure~\ref{fig:part}.

\begin{wrapfigure}{r}{0.45\textwidth}
\includegraphics[page=2]{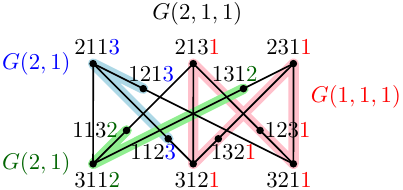}
\caption{Decomposing $G(2,1,1)$ into three subgraphs by fixing the last symbol.
}
\label{fig:part}
\end{wrapfigure}
Moreover, for any $\bb\precdot\ba$ we have $\Delta(\bb)=\Delta(\ba)-1$ or $\Delta(\bb)=\Delta(\ba)+1$.
In particular, if $\Delta(\ba)>0$, then we have $\Delta(\bb)\geq 0$.
For example, the vertex set of~$G(\ba)$ for $\ba=(3,2,2)$ ($\Delta(\ba)=1$) can be partitioned into one copy of~$G(\bb)$ for $\bb=(2,2,2)$ (if the fixed symbol is~$c=1$; $\Delta(\bb)=2$) and two copies of~$G(\bb')$ for $\bb'=(3,2,1)$ (if the fixed symbol is~$c=2$ or~$c=3$; $\Delta(\bb')=0$).
Therefore, we may construct a Hamilton path in~$G(3,2,2)$ by gluing together paths in each of these three subgraphs which exist by induction.

While conceptually simple, implementing this idea incurs considerable technical obstacles, in particular for some of the graphs~$G(\ba)$ with~$\Delta(\ba)=1$, i.e., instances that are very close to the hyperplane $\Delta(\ba)=0$.
The proof is split into several interdependent lemmas, and it is the technically most demanding part of our paper.

Theorem~\ref{thm:234} follows immediately from the inductive proof of Theorem~\ref{thm:P} and by settling finitely many cases with computer help.

To further illustrate the ideas outlined before, we close this section by reproducing Tchuente's proof of Theorem~\ref{thm:T}.
To prove that $G(1^k)$ is Hamilton-laceable, we proceed by induction on~$k$.
The induction basis $k=4$ can be checked by straightforward case analysis.
For the induction step, we assume that $G(1^{k-1})$, $k\geq 5$, is Hamilton-laceable, and we prove that $G(1^k)$ is also Hamilton-laceable.
Note that $1^{k-1}\precdot 1^k$ and $\Delta(1^k)=k-2=\Delta(1^{k-1})+1$.
The following arguments are illustrated in Figure~\ref{fig:T}.
The partition classes of the graph~$G(1^k)$ are given by the parity of the permutations, i.e., by the number of inversions.
Therefore, we consider two distinct permutations~$x$ and~$y$ of~$[k]$ with opposite parity, and we need to show how to connect them by a Hamilton path in~$G(1^k)$.
As $x\neq y$, there is a position~$\ihat>1$ in which $x$ and~$y$ differ, i.e., $x_\ihat\neq y_\ihat$, and this is the position that we will fix to different symbols.
Specifically, we choose a permutation~$\pi$ of~$[k]$ such that $\pi_1=x_\ihat$ and $\pi_k=y_\ihat$.
The permutation~$\pi$ captures the order in which we will fix symbols at position~$\ihat$.
We then choose permutations $u^j,v^j$ of~$[k]$, $j=1,\ldots,k$, satisfying $u^1=x$, $v^k=y$, and such that $u^j$ is obtained from~$v^{j-1}$ by a star transposition of the symbol~$\pi_j$ at position~1 with the symbol~$\pi_{j-1}$ at position~$\ihat$, for all $j=2,\ldots,k$.
Moreover, we choose $u^j$ and~$v^j$ such that the parity of the number of inversions after removing the symbol~$\pi_j$ is opposite.
Specifically, for $u^j$ this parity is the same as for~$u^1=x$ if and only if $\pi_j$ has the same parity as $\pi_1$, and for $v^j$ this parity is the same as for~$v^k=y$ if and only if~$\pi_j$ has the same parity as~$\pi_k$.
For each $j=1,\ldots,k$, we now consider the permutations whose $\ihat$th entry equals~$\pi_j$ (formally, this is the set $\Pi(1^k)^{\ihat,\pi_j}$).
Clearly, the subgraph of~$G(1^k)$ induced by these permutations is isomorphic to~$G(1^{k-1})$.
In other words, by removing the $\ihat$th entry and renaming entries to~$1,\ldots,k-1$, we obtain permutations of~$[k-1]$.
Consequently, by induction there is a path in~$G(1^{k})$ that visits all permutations whose $\ihat$th entry equals~$\pi_j$ and that connects $u^j$ to~$v^j$.
The concatenation of those $k$ paths obtained by induction is the desired Hamilton path in~$G(1^k)$ from~$x$ to~$y$, which completes the induction step.

\begin{figure}
\includegraphics[page=3]{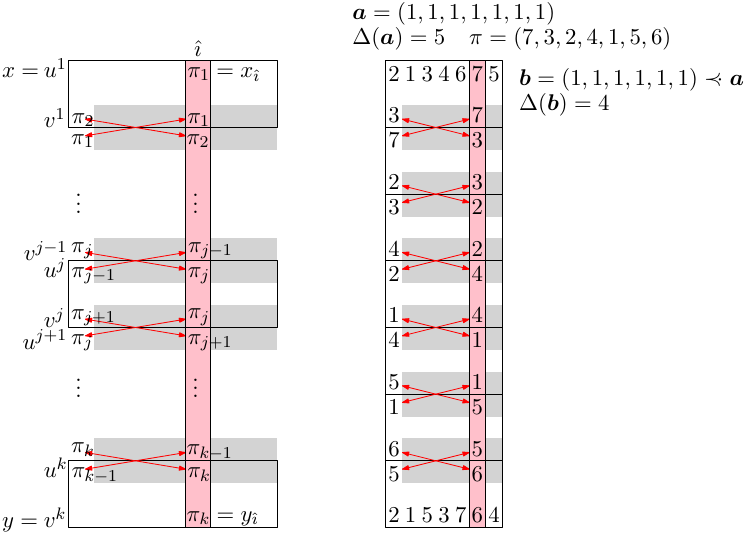}
\caption{Illustration of the proof of Theorem~\ref{thm:T}.
The left hand side shows the general schematic partitioning of the graph~$G(1^k)$ into blocks, each of which is a copy of~$G(1^{k-1})$, by fixing symbols at position~$\ihat$.
The right hand side shows a concrete example for $k=7$.
}
\label{fig:T}
\end{figure}

Note that the constraints imposed on the permutations~$\pi$ and $u^j,v^j$, $j=1,\ldots,k$, in this proof are very mild, and leave a lot of room for modifications to construct many different Hamilton paths, possibly so as to satisfy some additional conditions.

\subsubsection{The case $\Delta(\ba)=0$}
\label{sec:ideas0}

Our proofs of Theorems~\ref{thm:M} and~\ref{thm:M'} build on ideas introduced in the papers~\cite{gregor-muetze-nummenpalo:18,DBLP:conf/soda/MerinoMM21}.

Specifically, the first step in proving Theorem~\ref{thm:M} is to build a cycle factor in the graph~$G(\alpha,\alpha)$, i.e., a collection of disjoint cycles in the graph that together visit all vertices.
We then choose vertices~$x$ and~$y$ from the two partition classes of the graph that we want to connect by a Hamilton path.
In this we can take into account automorphisms of~$G(\alpha,\alpha)$, i.e., for proving laceability only certain pairs of vertices~$x$ and~$y$ in the two partition classes have to be considered.
In the next step, we join a small subset of cycles from the factor, including the ones containing~$x$ and~$y$, to a short path between~$x$ and~$y$.
This is achieved by taking the symmetric difference of the edge set of the cycle factor with a carefully chosen path~$P$ from~$x$ to~$y$ that alternately uses edges on one of the cycles from the factor and edges that go between two such cycles; see Figure~\ref{fig:strategy}~(a)+(b).
In the last step, we join the remaining cycles of the factor to the path between~$x$ and~$y$, until we end with a Hamilton path from~$x$ to~$y$.
Each such joining is achieved by taking the symmetric difference of the cycle factor with a suitably chosen 6-cycle; see  Figure~\ref{fig:strategy}~(b)+(c).

\begin{figure}[h!]
\includegraphics[page=1]{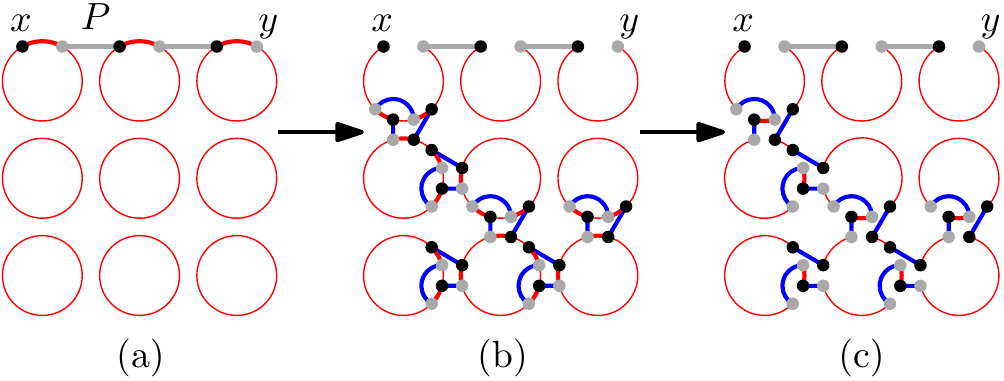}
\caption{Strategy of the proof of Theorem~\ref{thm:M}}
\label{fig:strategy}
\end{figure}

It was shown in~\cite{gregor-muetze-nummenpalo:18} that the cycles of the aforementioned cycle factor in~$G(\alpha,\alpha)$ are bijectively equivalent to plane trees with $\alpha$ vertices, and the joining operations via 6-cycles can be interpreted combinatorially as local change operations between two such plane trees.
To prove Theorem~\ref{thm:M'}, we first generalize the construction of this cycle factor in the graph~$G(\alpha,\alpha)$ to a cycle factor in any graph~$G(\ba)$, $\ba=(a_1,\ldots,a_k)$, with $\Delta(\ba)=0$.
It turns out that the cycles of this generalized factor can be interpreted combinatorially as \emph{vertex-labeled} plane trees, where exactly $a_i$ vertices have the label~$i$ for $i=2,\ldots,k$.
If $\ba=(\alpha,\alpha)$, then all vertex labels are the same, so we can consider the trees as unlabeled.
Moreover, the joining 6-cycles in~$G(\alpha,\alpha)$ generalize nicely to joining 12- or 6-cycles in~$G(\ba)$ with $\Delta(\ba)=0$, and they correspond to local change operations involving sets~$T$ of labeled plane trees with $|T|\in\{2,3,5,6\}$, depending on the location of vertex labels.
For proving Theorem~\ref{thm:M'}, we consider the special case~$\ba=(\alpha,\alpha-1,1)$, i.e., exactly one vertex in the plane trees is labeled differently from all other vertices, and we show that there is a choice of joining cycles so that the symmetric difference with the cycle factor yields a Hamilton cycle in the graph~$G(\ba)$.

\subsection{Outline of this paper}

In Section~\ref{sec:prelim} we collect definitions and observations that will be used throughout this paper.
The proof of Theorem~\ref{thm:N} is presented in Section~\ref{sec:proof-N}.
In Section~\ref{sec:proof-P}, we prove Theorems~\ref{thm:P} and~\ref{thm:234}.
The proofs of Theorems~\ref{thm:M} and~\ref{thm:M'} are presented in Section~\ref{sec:proof-MM'}.
This section can be read independently from the previous sections, but it relies on results from the previous papers~\cite{gregor-muetze-nummenpalo:18,DBLP:conf/soda/MerinoMM21}.
We conclude in Section~\ref{sec:open} with some open questions and possible directions for future work.

\section{Preliminaries}
\label{sec:prelim}

The following definitions and lemmas will be used repeatedly in this paper.

\subsection{Fixing symbols}

For any sequence of non-negative integers~$\ba=(a_1,a_2,\ldots,a_k)$, we let $\varphi(a)$ be the partition obtained by sorting the sequence~$\ba$ non-increasingly.
For instance, we have $\varphi((2,3,1,1,4))=(4,3,2,1,1)$.

As mentioned before, $\Pi(\ba)^{i,c}$ is the set of all multiset permutations from~$\Pi(\ba)$ whose $i$th symbol equals~$c$.
We write $G(\ba)^{i,c}:=G(\ba)[\Pi(\ba)^{i,c}]$ for the subgraph of~$G(\ba)$ induced by the vertex set~$\Pi(\ba)^{i,c}$.
We observed before that~$\Pi(a)^{1,c}$ is an independent set for all~$c\in[k]$, and that $G(\ba)^{i,c}$, $i=2,\ldots,n$, is isomorphic to $G(\bb)$ for $\bb:=\varphi(\ba')\precdot \ba$, $\ba':=(a_1,\ldots,a_{c-1},a_c-1,a_{c+1},\ldots,a_k)$.
We also allow repeating this operation and define $G(\ba)^{(i,j),(c,d)}:=(G(\ba)^{i,c})^{j,d}$ for $i,j\in[n]$ with $i\neq j$ and $c,d\in[k]$ with $c\neq d$ or $a_c=a_d\geq 2$.

Our first lemma follows directly from the definition~\eqref{eq:Delta}.

\begin{lemma}
\label{lem:delta-prec}
If $\bb\precdot\ba$, then we have $\Delta(\bb)=\Delta(\ba)-1$, except if $\ba$ and~$\bb$ differ in the largest part, in which case $\Delta(\bb)=\Delta(\ba)+1$.
\end{lemma}

\subsection{Partition lemmas}

The next lemma quantifies the size of the $k$ partition classes of the graph~$G(\ba)$ discussed in Section~\ref{sec:results}.

\begin{lemma}
\label{lem:kpart}
For any integer partition $\ba=(a_1,\ldots,a_k)$, the graph $G(a_1,\ldots,a_k)$ is $k$-partite, with the partition classes~$\Pi(\ba)^{1,1},\ldots,\Pi(\ba)^{1,k}$ given by fixing the first symbol, and the size of the $i$th partition class is
\begin{equation}
\label{eq:si}
s_i:=|\Pi(\ba)^{1,i}|=\binom{n-1}{a_1,\ldots,a_{i-1},a_i-1,a_{i+1},\ldots,a_k}=\frac{(n-1)!}{a_1!\cdots a_{i-1}!(a_i-1)!a_{i+1}!\cdots a_k!}.
\end{equation}
Moreover, we have $s_1\geq \cdots \geq s_k$, in particular, the first partition class is a largest one.
Lastly, if $\Delta(\ba)<0$, then the first partition class is larger than all others combined, if $\Delta(\ba)=0$ then the first partition class has the same size as all others combined, and if $\Delta(\ba)>0$ then every partition class is smaller than all others combined.
\end{lemma}

\begin{proof}
As every star transposition changes the first symbol, the sets $\Pi(\ba)^{1,1},\ldots,\Pi(\ba)^{1,k}$ are independent sets in the graph~$G(\ba)$.
The multinomial coefficient in~\eqref{eq:si} describes the number of ways of arranging $n-1$ symbols at positions $2,\ldots,n$, with only $a_i-1$ occurrences of the symbol~$i$ remaining (one symbol~$i$ is fixed at the first position).
This proves the first part of the lemma.
The monotonicity $s_1\geq \cdots \geq s_k$ follows immediately from~\eqref{eq:ai-decreasing} and~\eqref{eq:si}.
Using the definition~\eqref{eq:si}, the equation $s_1>\sum_{i=2}^k s_i$ can be rearranged to~$2a_1>n$, which by the definition~\eqref{eq:Delta} is equivalent to~$\Delta(\ba)<0$.
Reversing the direction of inequalities or replacing them by equality completes the proof of the second part of the lemma.
\end{proof}

The following result complements Lemma~\ref{lem:kpart} by showing that some of the graphs~$G(\ba)$ are not only $k$-partite, but also bipartite even for $k\geq 3$ distinct symbols.
An \emph{inversion} in a permutation is a pair of entries where the first entry is larger than the second.

\begin{lemma}
\label{lem:bipart}
For any integer partition $\ba=(a_1,\ldots,a_k)$ with $k\geq 3$ parts, the graph~$G(\ba)$ is not bipartite unless $a_1=\cdots=a_k=1$.
If $a_1=\cdots=a_k=1$, then the two partition classes are given by the parity of the number of inversions of the permutations, and they have the same size.
\end{lemma}

\begin{proof}
Note that if $\bb\prec\ba$, then $G(\bb)$ is a subgraph of~$G(\ba)$.
Specifically, if $b_i<a_i$, then we may fix $a_i-b_i$ occurrences of~$i$ on one of the positions~$2,\ldots,n_\ba$ in the graph~$G(\ba)$.
Consequently, unless $a_1=\cdots=a_k=1$, the graph $G(2,1,1)$ is a subgraph of~$G(\ba)$.
As $G(2,1,1)$ contains a 7-cycle (highlighted in Figure~\ref{fig:Ga}), a graph containing it as a subgraph is not bipartite.

We proceed to show that $G(1^k)$ is bipartite.
This follows directly from the observation that every transposition of two entries of a permutation changes the parity of the number of inversions.
Moreover, permuting the first two entries is a bijection between permutations with an even or odd number of inversions, so both partition classes have the same size.
\end{proof}

The next lemma shows that Hamilton paths in~$G(\ba)$ with $\Delta(\ba)=0$ have a very rigid structure.

\begin{lemma}
\label{lem:D0}
Let $\ba=(a_1,\ldots,a_k)$ be an integer partition with $\Delta(\ba)=0$.
Then on every Hamilton path in~$G(\ba)$ whose end vertices differ in the first entry, every second vertex of the path is from~$\Pi(\ba)^{1,1}$, and every other vertex on the path is not from~$\Pi(\ba)^{1,1}$.
\end{lemma}

In terms of star transpositions, this means that the symbol~1 is swapped in every step, either in or out of the first position.

\begin{proof}
By Lemma~\ref{lem:kpart}, the partition class~$\Pi(\ba)^{1,1}$ of the graph~$G(\ba)$ has the same size as all others combined.
Consequently, one of the end vertices of any Hamilton path must be in this largest class.
By the assumption that the other end vertex is not in~$\Pi(\ba)^{1,1}$, the Hamilton path is forced to visit~$\Pi(\ba)^{1,1}$ in every second step.
\end{proof}

\subsection{Hamiltonicity notions}
\label{sec:notions}

In our proofs we will use Hamilton-connectedness, Hamilton-laceability, and Hamilton-1-laceability heavily, and also other strengthenings of the notion of containing a Hamilton cycle.
It will be convenient to introduce shorthand notations for graphs satisfying these properties.

To this end, we let $\cG$ denote the family of all graphs $G(\ba)$ for any integer partition~$\ba=(a_1,\ldots,a_k)$, where $k\geq 2$ and $a_1\geq \cdots\geq a_k\geq 1$.
We also define the following subsets of~$\cG$:
\begin{itemize}[leftmargin=8mm, labelsep=4mm,noitemsep, topsep=3pt plus 3pt]
\item[\bfseries $\cP$] All graphs in~$\cG$ that have a Hamilton path, i.e., a path that visits every vertex exactly once.
\item[\bfseries $\cC$] All graphs in~$\cG$ that have a Hamilton cycle, i.e., a cycle that visits every vertex exactly once.
\item[\bfseries $\cE$] All graphs in~$\cG$ such that for every edge of the graph, there is a Hamilton cycle containing this edge.
\footnote{This property is called `positively Hamiltonian' in~\cite{MR307952}.}
\item[\bfseries $\cL$] All graphs in~$\cG$ that are Hamilton-laceable, i.e., bipartite graphs such that for any two vertices from the two partition classes, there is a Hamilton path between them.
\item[\bfseries $\cL_1$] All graphs in~$\cG$ that are Hamilton-1-laceable, i.e., for which there is a Hamilton path between any vertex in~$\Pi(\ba)^{1,1}$ and any vertex not in~$\Pi(\ba)^{1,1}$.
Recall that $\Pi(\ba)^{1,1}$ are the multiset permutations whose first symbol is~1.
\item[\bfseries $\cH$] All graphs in~$\cG$ that are Hamilton-connected, i.e., for which there is a Hamilton path between any two distinct vertices.
\end{itemize}

\begin{wrapfigure}{r}{0.45\textwidth}
\vspace{-6mm}
\includegraphics{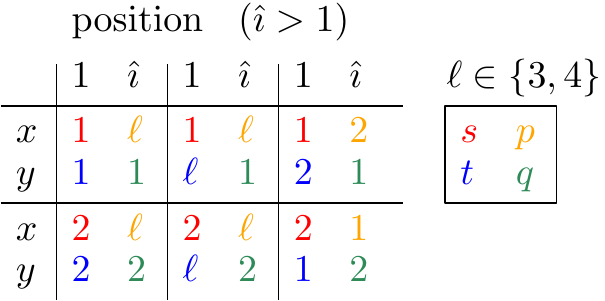}
\vspace{-2mm}
\caption{Definition of~$\cL_{12}$.}
\vspace{-2mm}
\label{fig:pq}
\end{wrapfigure}
To handle the graphs~$G(\ba)$ with $\ba=(\alpha,\alpha,1)$ in our proofs, we need yet another Hamiltonicity notion.
While we conjecture that these graphs~$G(\ba)$ are Hamilton-connected, we are unable to prove this based on Conjecture~\ref{conj:0}.
This is because all integer partitions~$\bb$ with $\bb\precdot\ba$, namely $\bb=(\alpha,\alpha-1,1)$ and $\bb=(\alpha,\alpha)$ satisfy~$\Delta(\ba)=0$, which makes these graphs~$G(\ba)$ particularly difficult to deal with.
We will content ourselves by showing that~$G(\ba)$ satisfies the following variant of Hamilton-laceability, which we denote by~$\cL_{12}$.

Roughly speaking, the set~$\cL_{12}$ contains the graphs from~$\cG$ that admit Hamilton paths between pairs of vertices~$x$ and~$y$ with particular combinations of first symbols~$s$ and~$t$, subject to the constraint that another position~$\ihat>1$ contains a particular combination of values~$p$ and~$q$.
The required combinations of symbols are summarized in Figure~\ref{fig:pq}, and they are formally defined as follows.
For $\ell\in\{3,4\}$ and any $s\in\{1,2\}$ and $t\in\{1,2,\ell\}$ we define
\begin{equation}
\label{eq:pq}
p_\ell(s,t):=\begin{cases}
\ell & \text{ if } t\in\{s,\ell\}, \\
t & \text{ otherwise},
\end{cases}
\quad
\text{ and }
\quad
q_\ell(s,t):=s.
\end{equation}
The set~$\cL_{12}$ contains all graphs in~$\cG$ for which there is a Hamilton path between any vertex~$x$ with first symbol~$s\in\{1,2\}$ and any vertex~$y$ distinct from~$x$ with first symbol~$t\in\{1,2,\ell\}$, $\ell\in\{3,4\}$, for which there is a position~$\ihat>1$ such that $(x_\ihat,y_\ihat)=(p_\ell(s,t),q_\ell(s,t))$.
% 1|3  1|3  1|2
% 1|1  3|1  2|1
%
% 2|3  2|3  2|1
% 2|2  3|2  1|2
%
% 3|- this is the one missing case we cannot do
% 3|-

\begin{wrapfigure}{r}{0.3\textwidth}
\includegraphics{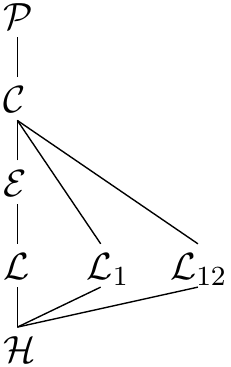}
\caption{Hamiltonicity notions ordered by inclusion.}
\vspace{-5mm}
\label{fig:ham}
\end{wrapfigure}
Even though the vertices of~$G(\alpha,\alpha,1)$ have no symbol~4, we still allow $\ell=4$ in this definition, as this graph may arise for example from~$G(\alpha,\alpha,1,1)$ by fixing the symbol~3, yielding a graph with symbols~$\{1,2,4\}$ that is isomorphic to~$G(\alpha,\alpha,1)$.

The obvious containment relations that follow from these definitions are shown in Figure~\ref{fig:ham}.
It is not clear whether $\cL_1\seq \cE$, because for $k\geq 3$ a graph in~$\cL_1$ has edges between the partition classes where the first symbol is distinct from~1.
It is also not clear whether $\cL\seq \cL_1$, because of the graphs~$G(1^k)$, $k\ge 3$, which by Lemma~\ref{lem:bipart} are bipartite but these partition classes are \emph{not} given by fixing the first symbol.
Recall that we consider a single edge, i.e., the graph~$G(1,1)$, as a cycle (otherwise $\cL\seq \cE$ and $\cL_1\seq\cC$ would be violated).

\section{Proof of Theorem~\ref{thm:N}}
\label{sec:proof-N}

In this section we present the proof of Theorem~\ref{thm:N}, implementing the idea outlined in Section~\ref{sec:ideasN}.

\begin{proof}[Proof of Theorem~\ref{thm:N}]
If $\Delta(\ba)<0$, then by Lemma~\ref{lem:kpart}, the partition class~$\Pi(\ba)^{1,1}$ of the graph~$G(\ba)$ is larger than all others combined, so there cannot be a Hamilton cycle.
This proves the first part of the theorem.

To prove the second part, we first show that if $\Delta(\ba)<0$, then the sizes of the partition classes defined in~\eqref{eq:si} satisfy $s_1-\sum_{i=2}^k s_i>1$, unless $\ba=(2,1)$.
Clearly, if one partition class of~$G(\ba)$ is by strictly more than~1 larger than all others combined, then there cannot be a Hamilton path.
Note also that the graph~$G(2,1)$ is the path on three vertices, which indeed has a Hamilton path.
We have seen that $\Delta(\ba)<0$ implies $s_1-\sum_{i=2}^k s_i>0$, and we now argue that this difference equals~1 if and only if~$\ba=(2,1)$.
We clearly have
\begin{equation}
\label{eq:binom}
\binom{n}{a_1,\ldots,a_k}\geq \binom{n}{a_1}.
\end{equation}
Moreover, from~\eqref{eq:si} we see that
\begin{equation}
\label{eq:s1}
s_1=\frac{a_1}{n}\binom{n}{a_1,\ldots,a_k}.
\end{equation}
Combining these observations we get
\begin{align*}
1 &= s_1 - \sum_{i=2}^k s_i
  = 2s_1 - \sum_{i=1}^k s_i
  \eqBy{eq:si} 2s_1 - \binom{n}{a_1,\ldots,a_k}
	\eqBy{eq:s1} \frac{2a_1-n}{n}\binom{n}{a_1,\ldots,a_k} \\
	&\geBy{eq:binom} \frac{2a_1-n}{n}\binom{n}{a_1}
	= (2a_1-n)\cdot\frac{n-1}{a_1}\cdot\frac{n-2}{a_1-1}\cdots \frac{n-a_1+1}{2}.
\end{align*}
The resulting product of fractions is non-decreasing, i.e., we have $\frac{n-1}{a_1}\leq \frac{n-2}{a_1-1}\leq \cdots\leq \frac{n-a_1+1}{2}$ because $n\geq a_1+1$.
We thus obtain the inequality
\begin{equation}
\label{eq:one-ineq}
1\geq (2a_1-n)\Big(\frac{n-1}{a_1}\Big)^{a_1-1}
\end{equation}
As $2a_1-n=-\Delta(\ba)$ and $\Delta(\ba)<0$ we have $2a_1-n>0$, and as this is an integer we have $2a_1-n\geq 1$.
As $n\geq a_1+1$ we also have $\frac{n-1}{a_1}\geq 1$.
From~\eqref{eq:one-ineq} we conclude that $2a_1-n=1$ and $\frac{n-1}{a_1}=1$, which implies $a_1=2$ and $n=3$, i.e., $\ba=(2,1)$.

This completes the proof.
\end{proof}

\section{Proofs of Theorems~\ref{thm:P} and~\ref{thm:234}}
\label{sec:proof-P}

In this section we prove Theorems~\ref{thm:P} and~\ref{thm:234}.
The proofs are split into several auxiliary lemmas, which we present first.
We then describe the computer experiments we did for settling the base cases of our induction proofs.
We then present the proofs of Theorems~\ref{thm:P} and~\ref{thm:234}, using the base cases and the lemmas.
Finally, we present the proofs of all auxiliary lemmas.

\subsection{Auxiliary lemmas}

We will use the following auxiliary lemmas; see Figure~\ref{fig:ind}.
These lemmas establish the Hamiltonicity of~$G(\ba)$ for all integer partition~$\ba$ with~$\Delta(\ba)\geq 1$, conditional on the Hamiltonicity of~$G(\bb)$ for various integer partitions~$\bb\prec\ba$ that satisfy $\Delta(\bb)\geq 0$.
The main technical achievement here is to partition all possible cases of frequency vectors~$\ba$ with $\Delta(\ba)\geq 1$ into disjoint cases, such that the induction proofs of Theorems~\ref{thm:P} and~\ref{thm:234} which apply these lemmas only rely on previously established cases.

Lemma~\ref{lem:H2} covers most graphs~$G(\ba)$ with~$\Delta(\ba)\geq 2$, while three special cases of these graphs are covered by Lemmas~\ref{lem:H2'}, \ref{lem:H2''}, and~\ref{lem:T'}.
Similarly, Lemma~\ref{lem:H1} covers most graphs~$G(\ba)$ with~$\Delta(\ba)=1$, while three special cases of these are covered by Lemmas~\ref{lem:H1'}, \ref{lem:H1''} and~\ref{lem:L12}.

\begin{figure}
\centerline{
\includegraphics{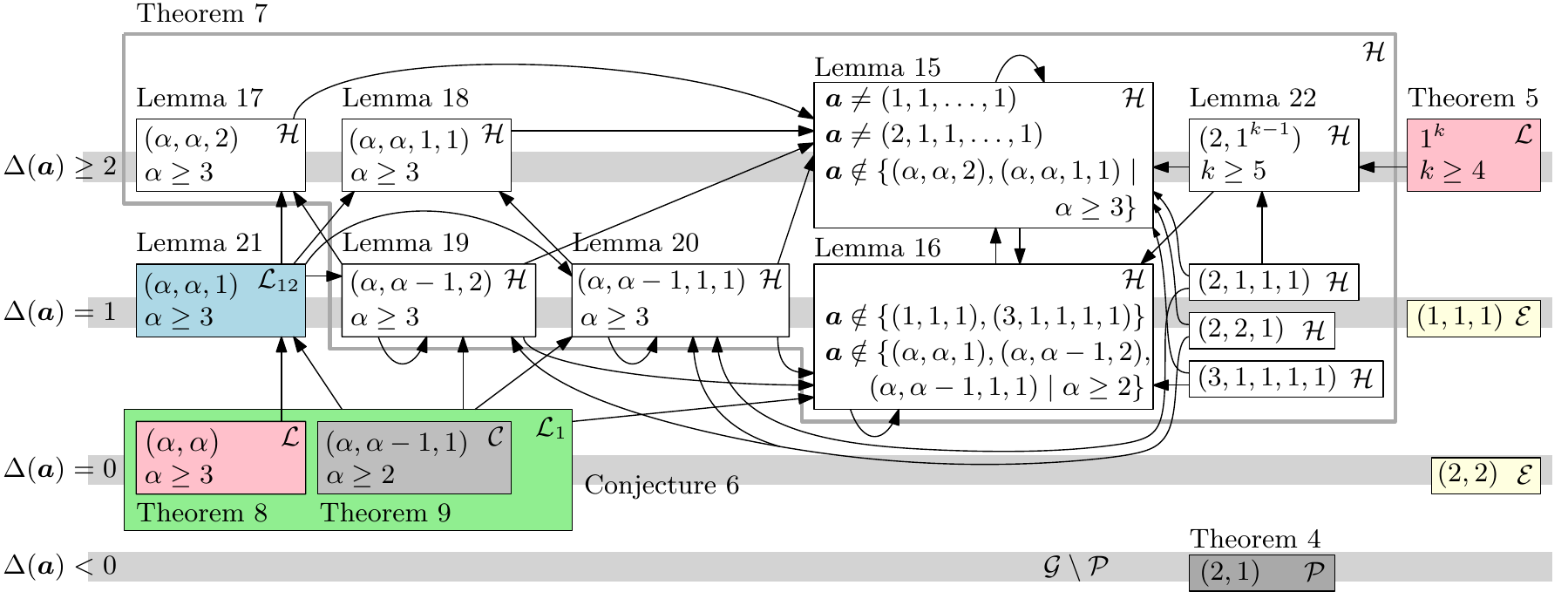}
}
\caption{Relations between our results and auxiliary lemmas, and the different Hamiltonicity notions we establish for them.
Arrows indicate dependencies in the proof of~Theorem~\ref{thm:P}.
}
\label{fig:ind}
\end{figure}

\todo{Make sure theorem/lemma/conjecture numbers are up-to-date.}

\begin{lemma}
\label{lem:H2}
Let $\ba=(a_1,\dots, a_k)$ be an integer partition with $\Delta(\ba)\geq 2$ with $\ba\neq (1,1,\ldots,1)$, $\ba\neq (2,1,1,\ldots,1)$ and~$\ba\notin\{(\alpha,\alpha,2),(\alpha,\alpha,1,1)\mid \alpha\geq 3\}$.
If for all $\bb\precdot\ba$ we have $G(\bb)\in\cH$, then we have $G(\ba)\in\cH$.
\end{lemma}

\begin{lemma}
\label{lem:H1}
Let $\ba=(a_1,\ldots,a_k)$ be an integer partition with $\Delta(\ba)=1$ with $\ba\notin\{(1,1,1),(3,1,1,1,1)\}$ and $\ba\notin\{(\alpha,\alpha,1),(\alpha,\alpha-1,2),(\alpha,\alpha-1,1,1)\mid\alpha\geq 2\}$.
If for all $\bb\precdot\ba$ with $\Delta(\bb)=0$ we have $G(\bb)\in\cL_1$, and for the unique $\bb\precdot\ba$ with $\Delta(\bb)=2$ we have that all $\bc\precdot\bb$ satisfy $G(\bc)\in\cH$, then we have $G(\ba)\in\cH$.
\end{lemma}

\begin{lemma}
\label{lem:H2'}
Consider the integer partition $\ba=(\alpha,\alpha,2)$ for $\alpha\geq 3$, which satisfies $\Delta(\ba)=2$.
If $\bb=(\alpha,\alpha,1)\precdot\ba$ satisfies $G(\bb)\in\cL_{12}$ and $\bb=(\alpha,\alpha-1,2)\precdot\ba$ satisfies $G(\bb)\in\cH$, then we have $G(\ba)\in\cH$.
\end{lemma}

\begin{lemma}
\label{lem:H2''}
Consider the integer partition $\ba=(\alpha,\alpha,1,1)$ for $\alpha\geq 3$, which satisfies $\Delta(\ba)=2$.
If $\bb=(\alpha,\alpha,1)\precdot\ba$ satisfies $G(\bb)\in\cL_{12}$ and $\bb=(\alpha,\alpha-1,1,1)\precdot\ba$ satisfies $G(\bb)\in\cH$, then we have $G(\ba)\in\cH$.
\end{lemma}

\begin{lemma}
\label{lem:H1'}
Consider the integer partition $\ba=(\alpha,\alpha-1,2)$ for $\alpha\geq 3$, which satisfies $\Delta(\ba)=1$.
If $\bb=\varphi((\alpha,\alpha-2,2))\precdot\ba$ and $\bb=(\alpha,\alpha-1,1)\precdot\ba$ satisfy $G(\bb)\in\cL_1$, and if $\bc=(\alpha-1,\alpha-1,1)\prec\ba$ satisfies $G(\bc)\in\cL_{12}$ and $\bc=\varphi((\alpha-1,\alpha-2,2))\prec\ba$ satisfies $G(\bc)\in\cH$, then we have $G(\ba)\in\cH$.
\end{lemma}

\begin{lemma}
\label{lem:H1''}
Consider the integer partition $\ba=(\alpha,\alpha-1,1,1)$ for $\alpha\geq 3$, which satisfies $\Delta(\ba)=1$.
If $\bb=(\alpha,\alpha-2,1,1)\precdot\ba$ and $\bb=(\alpha,\alpha-1,1)\precdot\ba$ satisfy $G(\bb)\in\cL_1$, and if $\bc=(\alpha-1,\alpha-1,1)\prec\ba$ satisfies $G(\bc)\in\cL_{12}$ and $\bc=(\alpha-1,\alpha-2,1,1)\prec\ba$ satisfies $G(\bc)\in\cH$, then we have $G(\ba)\in\cH$.
\end{lemma}

\begin{lemma}
\label{lem:L12}
Consider the integer partition $\ba=(\alpha,\alpha,1)$ for $\alpha\geq 3$, which satisfies $\Delta(\ba)=1$.
If $\bb=(\alpha,\alpha)\precdot\ba$ satisfies $G(\bb)\in\cL$ and $\bb=(\alpha,\alpha-1,1)\precdot\ba$ satisfies $G(\bb)\in\cL_1$, then we have $G(\ba)\in\cL_{12}$.
\end{lemma}

\begin{lemma}
\label{lem:T'}
Consider the integer partition $\ba=(2,1^{k-1})$ for $k\geq 5$, which satisfies $\Delta(\ba)\geq 2$.
Then we have $G(\ba)\in\cH$.
\end{lemma}

\subsection{Base cases}
\label{sec:base}

\begin{table}
\caption{Base cases settled by computer.}
\label{tab:base}
\begin{tabular}{l|l|c|l|l|l|l}
$n$ & Partition~$\ba$ & $\Delta(\ba)$ & Hamiltonicity of~$G(\ba)$ & Vertices & Vertex pairs & Used in proof of \\ \hline
2 & (1,1)         & 0 & $\cH$                 & 2        & 1 & \\ \hline
3 & (1,1,1)       & 1 & $\cE$ but not $\cL$   & 6        & 2 & \\ \hline
4 & (2,2)         & 0 & $\cE$ but not $\cL$   & 6        & 2 & \\
  & (2,1,1)       & 0 & $\cL_1$ but not $\cH$ & 12       & 6 & Thm.~\ref{thm:M'}+\ref{thm:234} \\ \hline
5 & (2,2,1)       & 1 & $\cH$                 & 30       & 23 & Thm.~\ref{thm:P}+\ref{thm:234} \\
  & (2,1,1,1)     & 1 & $\cH$                 & 60       & 31 & Thm.~\ref{thm:P} \\ \hline
6 & (3,3)         & 0 & $\cL_1=\cL$ but not $\cH$ & 20       & 3 & Thm.~\ref{thm:M}+\ref{thm:234} \\
  & (3,2,1)       & 0 & $\cL_1$ but not $\cH$ & 60       & 26 & Thm.~\ref{thm:M'}+\ref{thm:234} \\
  & (3,1,1,1)     & 0 & $\cL_1$ but not $\cH$ & 120      & 16 & Thm.~\ref{thm:234} \\ \hline
7 & (3,3,1)       & 1 & $\cH$                 & 140      & 38 & Thm.~\ref{thm:234} \\
  & (3,2,2)       & 1 & $\cH$                 & 210      & 65 & \\
  & (3,2,1,1)     & 1 & $\cH$                 & 420      & 209 & \\
  & (3,1,1,1,1)   & 1 & $\cH$                 & 840      & 79 & Thm.~\ref{thm:P} \\ \hline
8 & (4,4)         & 0 & $\cL_1=\cL$ but not $\cH$ & 70       & 4  & Thm.~\ref{thm:M}+\ref{thm:234} \\
  & (4,3,1)       & 0 & $\cL_1$ but not $\cH$ & 280      & 40 & Thm.~\ref{thm:234} \\
  & (4,2,2)       & 0 & $\cL_1$ but not $\cH$ & 420      & 32 & Thm.~\ref{thm:234} \\
  & (4,2,1,1)     & 0 & $\cL_1$ but not $\cH$ & 840      & 100 & Thm.~\ref{thm:234} \\
  & (4,1,1,1,1)   & 0 & $\cL_1$ but not $\cH$ & 1680     & 36 & Thm.~\ref{thm:234} \\ \hline
9 & (4,4,1)       & 1 & $\cH$                 & 630      & 53  & Thm.~\ref{thm:234} \\
  & (4,3,2)       & 1 & $\cH$                 & 1260     & 219 & \\
  & (4,3,1,1)     & 1 & $\cH$                 & 2520     & 347 & \\
  & (4,2,2,1)     & 1 & $\cC$ but nothing more tested & 3780     & 565 & \\
  & (4,2,1,1,1)   & 1 & $\cC$ but nothing more tested & 7560     & 685 & \\
  & (4,1,1,1,1,1) & 1 & $\cC$ but nothing more tested & 15120    & 173 & \\
%10& (5,5)   & 0 & $\cL_1=\cL$ but not $\cH$ & 252 & 5 & \\
%  & (5,4,1) & 0 & $\cL_1$ but not $\cH$  & 1260 & 54 & \\ \hline
%11& (5,5,1) & 1 & $\cH$  & 2772 & 68 & \\ \hline
%12& (6,6)   & 0 & $\cL_1=\cL$ but not $\cH$ & 924 & 6 & \\
%  & (6,5,1) & 0 & $\cC$ but nothing more tested & 5544 & 68 & \\ \hline
%13& (6,6,1) & 1 & $\cC$ but nothing more tested & 12012 & 83 & \\
\end{tabular}
\end{table}

We performed computer experiments to settle the base cases for our inductive proofs, and also for collecting evidence for Conjecture~\ref{conj:0}.
The results of these experiments are summarized in Table~\ref{tab:base}.
The second-to-last column contains the number of non-isomorphic pairs of vertices of the graph that need to be checked when testing for~$\cL_1$ or~$\cH$.

We restrict our attention to integer partitions~$\ba$ with~$\Delta(\ba)=0$ and~$\Delta(\ba)=1$.
The results in Table~\ref{tab:base} confirm Conjecture~\ref{conj:0} for all integer partitions~$\ba\neq (2,2)$ with $\Delta(\ba)=0$ that satisfy $n\leq 8$.
Some of these results will be used in our proofs of Theorems~\ref{thm:P}--\ref{thm:234}, as indicated in the last column of the table.
The last three instances in Table~\ref{tab:base} were too large to test for Hamilton-connectedness, but they are Hamilton-connected by Theorem~\ref{thm:234}.

\subsection{Proof of Theorem~\ref{thm:P}}

With the help of Lemmas~\ref{lem:H2}--\ref{lem:T'}, proving Theorem~\ref{thm:P} is relatively straightforward.

\begin{proof}[Proof of Theorem~\ref{thm:P}]
This proof is illustrated in Figure~\ref{fig:ind}.
Conjecture~\ref{conj:0} asserts that for all integer partitions $\ba\neq (2,2)$ with $\Delta(\ba)=0$ we have $G(\ba)\in\cL_1$, and we now assume that this is indeed the case.
We show that under this assumption the following eleven statements hold.
Consider the following statements about integer partitions~$\ba$ with $\Delta(\ba)=1$:
\begin{enumerate}[label=(1\alph*),leftmargin=8mm, topsep=0mm, noitemsep]
\item For $\ba=(1,1,1)$ we have $G(\ba)\in\cE$;
\item For $\ba\in\{(2,1,1,1),(2,2,1),(3,1,1,1,1)\}$ we have $G(\ba)\in\cH$;
\item For any $\ba\in\{(\alpha,\alpha,1)\mid\alpha\geq 3\}$ we have $G(\ba)\in\cL_{12}$;
\item For any $\ba\in\{(\alpha,\alpha-1,2)\mid\alpha\geq 3\}$ we have $G(\ba)\in\cH$;
\item For any $\ba\in\{(\alpha,\alpha-1,1,1)\mid\alpha\geq 3\}$ we have $G(\ba)\in\cH$;
\item For any $\ba$ that satisfies $\ba\notin\{(1,1,1),(3,1,1,1,1)\}$ and $\ba\notin\{(\alpha,\alpha,1),(\alpha,\alpha-1,2),(\alpha,\alpha-1,1,1)\mid \alpha\geq 2\}$ we have $G(\ba)\in\cH$.
\end{enumerate}
Consider the following statements about integer partitions~$\ba$ with $\Delta(\ba)\geq 2$:
\begin{enumerate}[label=(2\alph*),leftmargin=8mm, topsep=0mm, noitemsep]
\item For any $\ba=1^k$, $k\geq 4$, we have $G(\ba)\in\cL$;
\item For any $\ba=(2,1^{k-1})$, $k\geq 5$, we have $G(\ba)\in\cH$;
\item For any $\ba\in\{(\alpha,\alpha,2)\mid\alpha\geq 3\}$ we have $G(\ba)\in\cH$;
\item For any $\ba\in\{(\alpha,\alpha,1,1)\mid\alpha\geq 3\}$ we have $G(\ba)\in\cH$;
\item For any $\ba$ that satisfies $\ba\neq (1,1,\ldots,1)$, $\ba\neq (2,1,1,\ldots,1)$ and $\ba\notin\{(\alpha,\alpha,2),(\alpha,\alpha,1,1)\mid\alpha\geq 3\}$ we have $G(\ba)\in\cH$.
\end{enumerate}
Note that the cases considered in these statements are mutually exclusive and cover all possibilities.

Recall that~$G(1,1,1)$ is a 6-cycle, so (1a) holds trivially and unconditionally.
Also, (1b) holds unconditionally by Table~\ref{tab:base}.
Moreover, (2a) and~(2b) hold unconditionally by Theorem~\ref{thm:T} and Lemma~\ref{lem:T'}, respectively.
Statement~(1c) holds by Lemma~\ref{lem:L12}, using that for $\ba=(\alpha,\alpha,1)$ with $\alpha\geq 3$, the partition $\bb=(\alpha,\alpha)\precdot\ba$ satisfies $G(\bb)\in\cL$ by Theorem~\ref{thm:M}, and the partition $\bb=(\alpha,\alpha-1,1)\precdot\ba$ satisfies $G(\bb)\in\cL_1$ by Conjecture~\ref{conj:0}.

We prove the remaining statements by induction on~$n=\sum_{i=1}^k a_i$.

Statement~(1d) holds by Lemma~\ref{lem:H1'}, using that for $\ba=(\alpha,\alpha-1,2)$ with $\alpha\geq 3$, the partitions $\bb=\varphi((\alpha,\alpha-2,2))\precdot\ba$ and $\bb=(\alpha,\alpha-1,1)\precdot\ba$ satisfy $G(\bb)\in\cL_1$ by Conjecture~\ref{conj:0}, the partition $\bc=(\alpha-1,\alpha-1,1)\prec\ba$ satisfies $G(\bc)\in\cL_{12}$ if $\alpha\geq 4$ by induction and~(1c) and $G(\bc)=G(2,2,1)\in\cH\seq\cL_{12}$ if $\alpha=3$ by Table~\ref{tab:base}, and the partition~$\bc=\varphi((\alpha-1,\alpha-2,2))\prec\ba$ satisfies $G(\bc)\in\cH$ if $\alpha\geq 4$ by induction and~(1d) and $G(\bc)=G(2,2,1)\in\cH$ if $\alpha=3$.

Statement~(1e) holds by Lemma~\ref{lem:H1''}, using that for $\ba=(\alpha,\alpha-1,1,1)$ with $\alpha\geq 3$, the partitions $\bb=(\alpha,\alpha-2,1,1)\precdot\ba$ and $\bb=(\alpha,\alpha-1,1)\precdot\ba$ satisfy $G(\bb)\in\cL_1$ by Conjecture~\ref{conj:0}, the partition $\bc=(\alpha-1,\alpha-1,1)\prec\ba$ satisfies $G(\bc)\in\cL_{12}$ if $\alpha\geq 4$ by induction and~(1c) and $G(\bc)=G(2,2,1)\in\cH$ if $\alpha=3$ by Table~\ref{tab:base}, and the partition~$\bc=(\alpha-1,\alpha-2,1,1)\prec\ba$ satisfies $G(\bc)\in\cH$ if $\alpha\geq 4$ by induction and~(1e) and $G(\bc)=G(2,1,1,1)\in\cH$ if $\alpha=3$.

Statement~(2c) holds by Lemma~\ref{lem:H2'}, using that for $\ba=(\alpha,\alpha,2)$ with $\alpha\geq 3$, the partition $\bb=(\alpha,\alpha,1)\precdot\ba$ satisfies $G(\bb)\in\cL_{12}$ by induction and~(1c), and the partition $\bb=(\alpha,\alpha-1,2)\precdot\ba$ satisfies $G(\bb)\in\cH$ by induction and~(1d).

Statement~(2d) holds by Lemma~\ref{lem:H2''}, using that for $\ba=(\alpha,\alpha,1,1)$ with $\alpha\geq 3$, the partition $\bb=(\alpha,\alpha,1)\precdot\ba$ satisfies $G(\bb)\in\cL_{12}$ by induction and~(1c), and the partition $\bb=(\alpha,\alpha-1,1,1)\precdot\ba$ satisfies $G(\bb)\in\cH$ by induction and~(1e).

To prove statement~(1f), consider a partition~$\ba=(a_1,\ldots,a_k)$ satisfying the conditions of~(1f).
Note that $n=\sum_{i=1}^k a_i\geq 9$, as the only partitions with $n\in\{2,\ldots,8\}$ and $\Delta(\ba)=1$ are $\ba\in\{(1,1,1),(2,2,1),(2,1,1,1),(3,3,1),(3,2,2),(3,2,1,1),(3,1,1,1,1)\}$, which are excluded by the conditions of~(1f).
From~\eqref{eq:Delta} and~$\Delta(\ba)=1$ we obtain $a_1=(n-1)/2\geq 4$ and $a_1+1=\sum_{i=2}^k a_i$.
The latter equation implies $a_2<a_1-1$, as otherwise $a_1+1\geq (a_1-1)+\sum_{i=3}^k a_i$, or equivalently, $\sum_{i=3}^k a_i\leq 2$, i.e., we would have $\ba\in\{(\alpha,\alpha,1),(\alpha,\alpha-1,2),(\alpha,\alpha-1,1,1)\}$ for $\alpha:=a_1\geq 4$, which is impossible by the conditions of~(1f).
We thus obtain $a_2\leq a_1-2$ and $\sum_{i=3}^k a_i\geq 3$, in particular $k\geq 3$.

We now check that the preconditions of Lemma~\ref{lem:H1} are met.
Consider any partition $\bb\precdot\ba$ that has the same first entry as~$\ba$.
By Lemma~\ref{lem:delta-prec} we have $\Delta(\bb)=0$.
From $n_\ba\geq 9$ we obtain that $n_\bb\geq 8$, implying that $\bb\neq (2,2)$.
Consequently, by Conjecture~\ref{conj:0} we have~$G(\bb)\in\cL_1$.

Now consider the partition $\bb=(a_1-1,a_2,\ldots,a_k)$, which satisfies $\Delta(\bb)=\Delta(\ba)+1=2$ by Lemma~\ref{lem:delta-prec}.
We show that all $\bc\precdot\bb$ satisfy $G(\bc)\in\cH$.
From $n_\ba\geq 9$ we obtain that $n_\bb\geq 8$ and $n_\bc\geq 7$.
Clearly, any $\bc\precdot\bb$ satisfies $\Delta(\bc)\in\{1,3\}$ by Lemma~\ref{lem:delta-prec}.
Specifically, the partition $\bc=(a_1-2,a_2,\ldots,a_k)$ satisfies $\Delta(\bc)=3$, but we have $\bc\neq 1^k$ for $k\geq 4$ because of $a_1-2\geq 2$, so by induction and~(2b)+(2e) we have $G(\bc)\in\cH$.
Moreover, any other partition $\bc\precdot\bb$ that has the same first entry as~$\bb$ satisfies~$\Delta(\bc)=1$, but we have $\bc\neq (1,1,1)$ and $\bc\neq (\alpha,\alpha,1)$ for $\alpha\geq 3$ because of $\sum_{i=3}^k a_i\geq 3$, so by induction and~(1b)+(1d)+(1e)+(1f) we have $G(\bc)\in\cH$.

We argued that the preconditions for applying Lemma~\ref{lem:H1} are met, and this lemma therefore yields~$G(\ba)\in\cH$, proving~(1f).

It remains to prove statement~(2e).
Consider a partition~$\ba=(a_1,\ldots,a_k)$ satisfying the conditions of~(2e).
Note that $n=\sum_{i=1}^k a_i\geq 6$, as there are no partitions with $n\in\{2,\ldots,5\}$ and $\Delta(\ba)=2$.
Note that $a_1\geq a_2\geq 2$ by the assumptions $\ba\neq(1,1,\ldots,1)$ and $\ba\neq (2,1,1,\ldots,1)$ in~(2e).
Moreover, if $\ba\notin\{(2,2,2),(2,2,1,1)\}$ then we have $a_1>a_2$ by the assumptions $\ba\notin\{(\alpha,\alpha,2),(\alpha,\alpha,1,1)\mid\alpha\geq 3\}$.
From the equation $a_1+2=\sum_{i=2}^k a_i$ and $a_1\geq a_2$ we obtain $\sum_{i=3}^k a_i\geq 2$, in particular $k\geq 3$.
Moreover, if $\ba\notin\{(2,2,2),(2,2,1,1)\}$ then we can use the strict inequality $a_1>a_2$ instead and obtain $\sum_{i=3}^k a_i\geq 3$.

We now verify that the preconditions of Lemma~\ref{lem:H2} are met, i.e., we consider all partitions~$\bb\precdot\ba$.
From $n_\ba\geq 5$ we obtain that $n_\bb\geq 4$.
If $a_1>a_2$ then consider the partition $\bb=(a_1-1,a_2,\ldots,a_k)$, which satisfies $\Delta(\bb)=\Delta(\ba)+1\geq 3$ by Lemma~\ref{lem:delta-prec}, but $\bb\neq 1^k$ for $k\geq 4$ because of $a_2\geq 2$.
Consequently, by induction and~(2b)+(2e) we have $G(\bb)\in\cH$.
Now suppose that $\ba\notin\{(2,2,2),(2,2,1,1)\}$ and consider any partition $\bb\precdot\ba$ that has the same first entry as~$\ba$, then by Lemma~\ref{lem:delta-prec} we have $\Delta(\bb)=\Delta(\ba)-1\geq 1$, but we have $\bb\neq (1,1,1)$ and $\bb\neq (\alpha,\alpha,1)$ for $\alpha\geq 3$ because of $\sum_{i=3}^k a_i-1\geq 2$, and $\bb\neq 1^k$ for $k\geq 4$ and $\bb\neq (2,1^{k-1})$ for $k\geq 5$ because of $a_2\geq 2$, so by induction and (1b)+(1d)+(1e)+(1f)+(2b)+(2c)+(2d)+(2e) we have $G(\bb)\in\cH$.
Lastly, if $\ba\in\{(2,2,2),(2,2,1,1)\}$ then there two partitions~$\bb\precdot\ba$, namely $\bb\in\{(2,2,1),(2,1,1,1)\}$, and both satisfy $G(\bb)\in\cH$ by~(1b).

As the preconditions of Lemma~\ref{lem:H2} are satisfied, the lemma yields~$G(\ba)\in\cH$, proving~(2e).

This completes the proof of Theorem~\ref{thm:P}.
\end{proof}

\subsection{Proof of Theorem~\ref{thm:234}}

\begin{proof}[Proof of Theorem~\ref{thm:234}]
From Table~\ref{tab:base} we see that Conjecture~\ref{conj:0} holds for all ten integer partitions~$\ba\neq (2,2)$ with $\Delta(\ba)=0$ whose first entry~$a_1$ is at most~4.
Specifically, these are the partitions~$\ba\in\{(1,1),(2,1,1),(3,3),(3,2,1),(3,1,1,1),(4,4),(4,3,1),(4,2,2),(4,2,1,1),(4,1,1,1,1)\}$.
For any integer partition~$\ba$ with $a_1\leq 4$, we have that all integer partitions~$\bb\prec\ba$ also satisfy~$b_1\leq 4$.
Consequently, the statements (1a)--(1f) and (2a)--(2e) in the proof of Theorem~\ref{thm:P} hold unconditionally for all such integer partitions.
As $a_1\in\{2,3,4\}$ we have $\ba\neq (1,1,1)$ and $\ba\neq 1^k$ for $k\geq 4$, and as $G(\ba)\in\cH$ for $\ba\in\{(3,3,1),(4,4,1)\}$ by Table~\ref{tab:base}, we conclude that~$G(\ba)\in\cH$.
\end{proof}

\subsection{Proofs of auxiliary lemmas}

To prove Lemmas~\ref{lem:H2}--\ref{lem:T'}, we follow the ideas outlined in Section~\ref{sec:ideasP}, namely to partition the graph~$G(\ba)$ into subgraphs that are isomorphic to~$G(\bb)$ for integer partitions~$\bb\prec\ba$, by fixing symbols at one or more of the positions~$2,\ldots,n$.
This allows us to construct a Hamilton path in the entire graph~$G(\ba)$ by gluing together the Hamilton paths obtained in the various subgraphs~$G(\bb)$ for $\bb\prec\ba$.
The main difficulty is to ensure that the paths in the subgraphs~$G(\bb)$ under various constraints `fit together', and that the constraints imposed on them at the gluing vertices are not too severe to rule out their existence.
Four of these somewhat repetitive proofs are deferred to the appendix.

\begin{proof}[Proof of Lemma~\ref{lem:H2}]
As $\ba\neq (1,1,\ldots,1)$ and $\ba\neq (2,1,1,\ldots,1)$, we know that $a_1\geq a_2\geq 2$.
Moreover, from \eqref{eq:Delta} we know that $a_1+\Delta(\ba)=\sum_{i=2}^k a_i$.
Using $\Delta(\ba)\geq 2$ and $a_1\geq a_2$ in this equation yields $\sum_{i=3}^k a_i\geq 2$, in particular $k\geq 3$.

Let $x,y$ be two distinct vertices in~$G(\ba)$.
As $x\neq y$ there is a position $\ihat>1$ such that $x_\ihat\neq y_\ihat$.
As $a_1\geq a_2\geq 2$ and $\sum_{i=3}^k a_i\geq 2$, there are two indices $i_1,i_2\in[n]\setminus\{1,\ihat\}$ such that $x_{i_1}\neq x_{i_2}$.
Similarly, there are two indices $i_3,i_4\in[n]\setminus\{1,\ihat\}$ such that $y_{i_3}\neq y_{i_4}$.

We fix any permutation~$\pi$ on~$[k]$ such that $\pi_1=x_\ihat$ and $\pi_k=y_\ihat$, and we choose a sequence of multiset permutations $u^j,v^j\in \Pi(\ba)$, $j=1,\ldots,k$, satisfying the following conditions; see Figure~\ref{fig:H2}:
\begin{itemize}[leftmargin=8mm, noitemsep]
\item[(i)] $u^1=x$, $v^1_{i_1}=x_{i_2}$, and $v^k=y$, $u^k_{i_3}=y_{i_4}$;
\item[(ii)] $u^j_1=v^{j-1}_\ihat=\pi_{j-1}$, $u^j_\ihat=v^{j-1}_1=\pi_j$, and $u^j_i=v^{j-1}_i$ for all $2\leq j\leq k$ and $i\in[n]\setminus\{1,\ihat\}$.
\end{itemize}
As $u^j_1=\pi_{j-1}$ and $v^j_1=\pi_{j+1}$ by~(ii), we have $u^j\neq v^j$ for all $1<j<k$.
Moreover, by~(i) and the choice of $i_1,i_2$ and $i_3,i_4$ we also have $u^1\neq v^1$ and $u^k\neq v^k$, respectively.

\begin{figure}
\includegraphics[page=1]{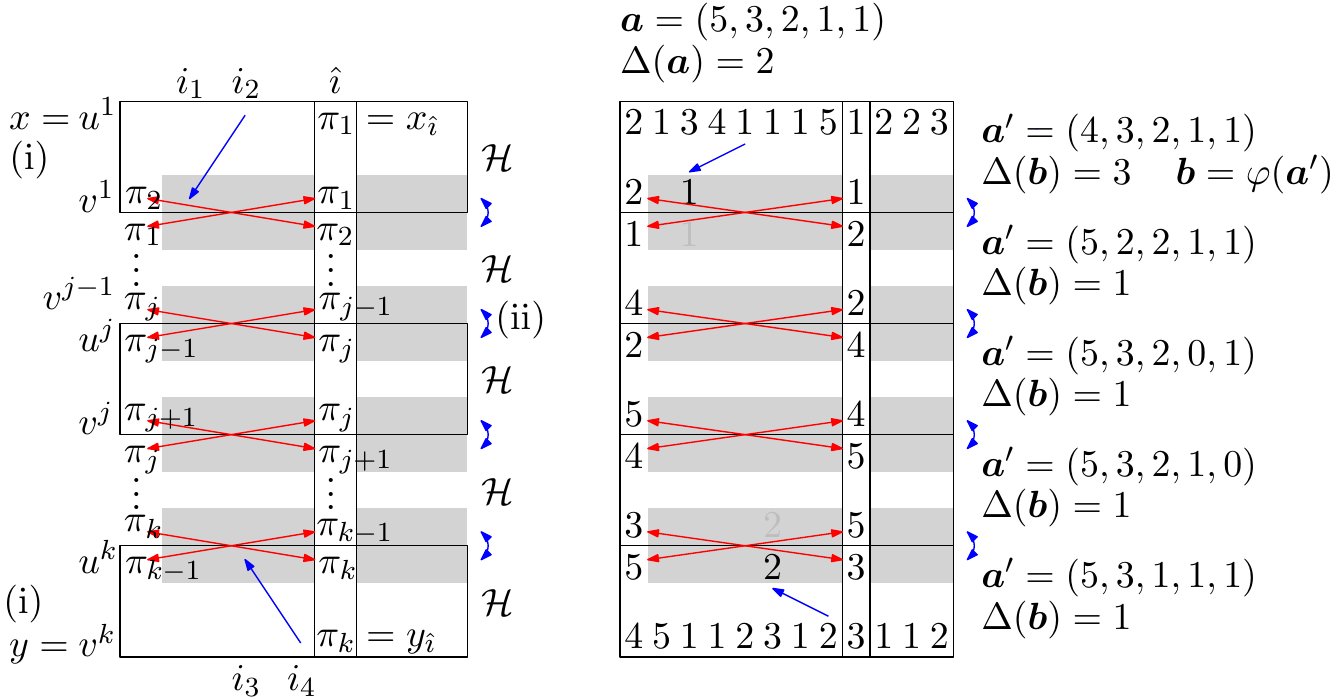}
\caption{Illustration of the proof of Lemma~\ref{lem:H2}.}
\label{fig:H2}
\end{figure}

For $j=1,\ldots,k$ we define $\ba':=(a_1,\ldots,a_{\pi_j}-1,\ldots,a_k)$ and $\bb:=\varphi(\ba')\precdot\ba$, and we consider a Hamilton path~$P_j$ in the graph~$G^{\ihat,\pi_j}(\ba)\simeq G(\bb)$ from~$u^j$ to~$v^j$, which exists by the assumption~$G(\bb)\in\cH$.
By~(ii), for all $2\leq j\leq k$, the permutation $u^j$ differs from~$v^{j-1}$ by a transposition of the entries at positions~1 and~$\ihat$, implying that the concatenation $P_1P_2\cdots P_k$ is a Hamilton path in~$G(\ba)$ from~$x$ to~$y$.
\end{proof}

The strategy for proving Lemma~\ref{lem:H2'} and~\ref{lem:H2''} is the same in the proof of Lemma~\ref{lem:H2}, and these proofs can be found in the appendix.
The crucial difference is that now one or two building blocks for partitions~$\bb\precdot\ba$ are only assumed to satisfy the weaker property $G(\bb)\in\cL_{12}$, and not $G(\bb)\in\cH$, so we have to work around this by imposing extra conditions on those building blocks.

The following lemma will be used in the proof of Lemma~\ref{lem:H1}.
It guarantees the existence of two multiset permutations~$\pi$ and~$\rho$ that are first and last vertices of Hamilton paths in subgraphs of~$G(\ba)$ obtained by fixing symbols, in such a way that~$\pi$ and~$\rho$ satisfy a number of extra conditions.
In the proof of Lemma~\ref{lem:H1} we will encounter six different kinds of conditions, which are captured by the six cases in the lemma.

\begin{lemma}
\label{lem:matching}
For any odd $k\geq 3$ and $\ell:=(k-1)/2$ and any $a,b,c\in[k]$ with $a\neq b$ and $c\neq 1$, there are permutations~$\pi,\rho$ on~$[k]$ satisfying any chosen of the following sets of conditions:
\begin{enumerate}
\item[(oa)] $\pi_1=1$, $\rho_1=a$, $\rho_k=b$, $\pi_{2j}\neq \rho_j$ and $\pi_{2j+1}\neq \rho_{j+1}$ for all $1\leq j\leq \ell$;
\item[(ob1)] $\pi_1=1$, $\pi_k=c$, $\rho_1=1$, $\pi_{2j}\neq \rho_j$ and $\pi_{2j+1}\neq \rho_{j+1}$ for all $1\leq j<\ell$, $\pi_{k-1}\neq \rho_k$;
\item[(ob2)] $\pi_1=1$, $\pi_2=c$, $\rho_1=1$, $\pi_3=\rho_k$, $\pi_{2j}\neq \rho_j$ and $\pi_{2j+1}\neq \rho_{j+1}$ for all $1\leq j\leq \ell$.
\end{enumerate}
For any even $k\geq 4$ and $\ell:=k/2$ and any $a,b,c\in[k]$ with $a\neq b$ and $c\neq 1$, there are permutations~$\pi,\rho$ on~$[k]$ satisfying any chosen of the following sets of conditions:
\begin{enumerate}
\item[(ea)] $\pi_1=1$, $\rho_1=a$, $\rho_k=b$, $\pi_3=\rho_{k-1}$, $\pi_{2j}\neq \rho_j$ and $\pi_{2j+1}\neq \rho_{j+1}$ for all $1\leq j<\ell$, $\pi_k\neq \rho_k$;
\item[(eb1)] $\pi_1=1$, $\pi_k=c$, $\rho_1=1$, $\pi_3=\rho_k$, $\pi_{2j}\neq \rho_j$ and $\pi_{2j+1}\neq \rho_{j+1}$ for all $1\leq j<\ell$;
\item[(eb2)] $\pi_1=1$, $\pi_k=c$, $\rho_1=1$, $\pi_{2j}\neq \rho_j$ and $\pi_{2j+1}\neq \rho_{j+1}$ for all $1\leq j<\ell$, $\pi_k\neq \rho_k$.
\end{enumerate}
\end{lemma}

\begin{figure}
\centerline{
\includegraphics{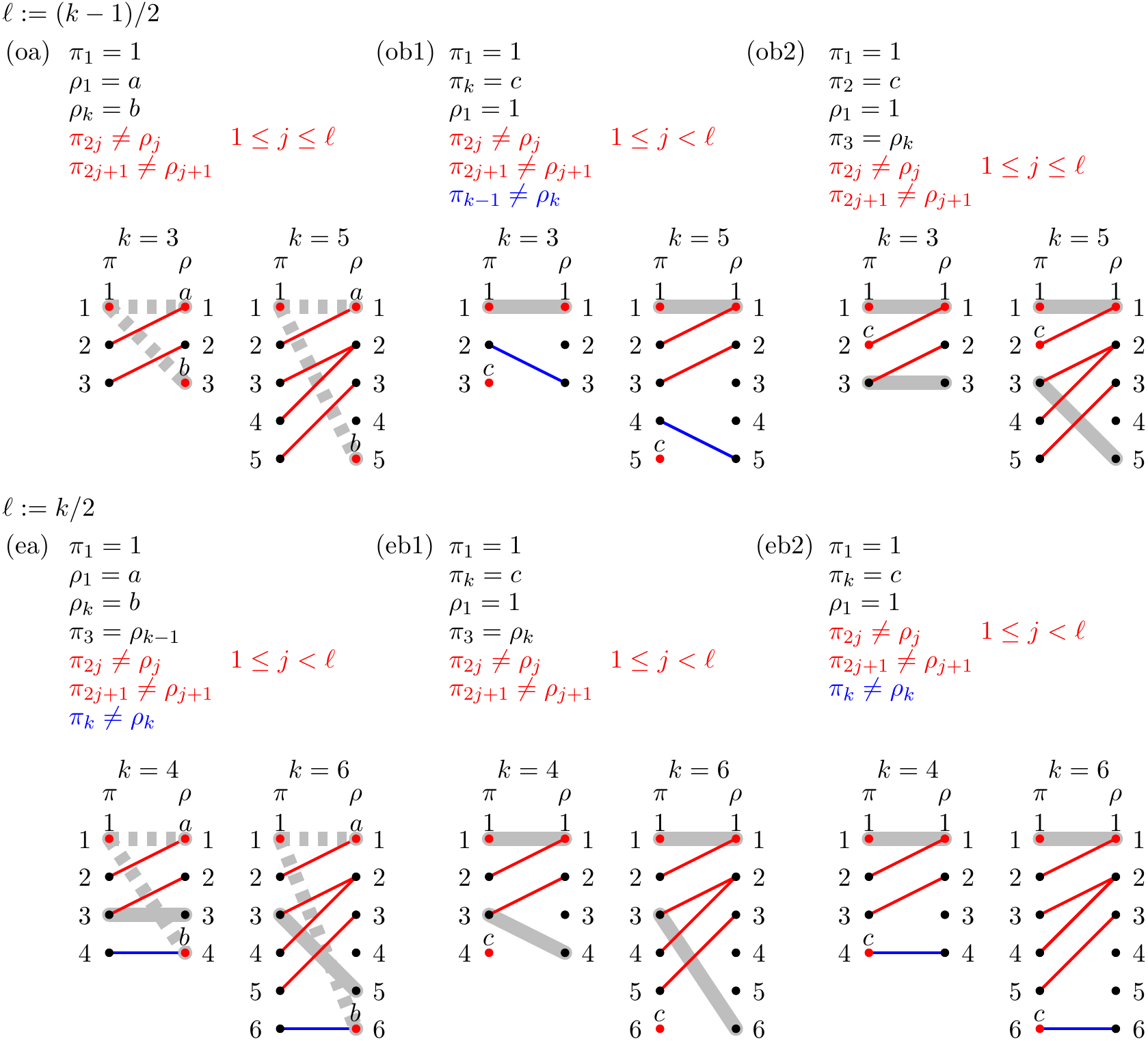}
}
\caption{Illustration of the graph~$H$ defined in the proof of Lemma~\ref{lem:matching}.
The pairs of vertices connected by thick solid lines are removed from the initial complete bipartite graph, and the thin edges are also removed (keeping the end vertices).
From the two pairs of vertices connected by thick dashed lines, one pair of vertices is removed if $a=1$ or $b=1$, and both edges are removed if~$a\neq 1$ and~$b\neq 1$ (keeping the end vertices).
}
\label{fig:matching}
\end{figure}

\begin{proof}
For any of the given sets of conditions, we can define a graph~$H$ that captures the conditions on~$\pi$ and~$\rho$ as follows:
We start with a complete bipartite graph with vertex partitions~$[k]$ and~$[k]$.
For any constraint of the form~$\pi_i=\rho_j$ we remove both end vertices~$i$ and~$j$ from the graph.
Similarly, for any pair of constraints of the form~$\pi_i=x$ and~$\rho_j=x$ we remove both end vertices~$i$ and~$j$ from the graph.
Moreover, for any constraint of the form~$\pi_i\neq \rho_j$ we remove the edge~$(i,j)$ from the graph (keeping the end vertices).
Lastly, for any pair of constraints of the form~$\pi_i=x$ and~$\rho_j=y$ for $x\neq y$ we remove the edge~$(i,j)$ from the graph.

Note that there are permutations~$\pi$ and~$\rho$ satisfying one of the given sets of conditions, if and only if the corresponding graph~$H$ admits a perfect matching.
Specifically, the matched pairs of vertices are the entries from~$\pi$ and~$\rho$ that will be assigned the same value.

The existence of a perfect matching in~$H$ defined before can be verified directly for~$k\leq 5$; see Figure~\ref{fig:matching}.
For $k\geq 6$ we argue via Hall's theorem:
We let~$s$ denote the sizes of the two partition classes of~$H$.
By the given equality constraints, at most two pairs of vertices are removed from the initial complete bipartite graph, so we have~$s\geq k-2\geq 4$.
Moreover, the given inequality constraints involve each~$\pi_i$ and~$\rho_j$ at most twice, so all vertices of~$H$ have degree at least~$s-2$.
Let $A$ be a subset of vertices of one partition class, and let~$A'$ be the set of all of its neighbors in the other partition class of~$H$.
As all vertices of~$H$ have degree at least~$s-2$, we know that~$|A'|\geq s-2$ whenever~$|A|>0$, so it suffices to check Hall's condition for $|A|\geq s-1$.
If $|A'|<|A|$, then there is a vertex on the same side of the partition as~$A'$ that is not connected to any vertex in~$A$, i.e., this vertex has degree at most~$s-|A|\leq s-(s-1)=1<s-2$, a contradiction as $s\geq 4$.
This proves that $|A'|\geq |A|$, so Hall's condition is satisfied and the graph~$H$ has a perfect matching.

This completes the proof of the lemma.
\end{proof}

We now outline our strategy for proving Lemma~\ref{lem:H1}.
In the preceding proof of Lemma~\ref{lem:H2}, we split~$G(\ba)$ into subgraphs~$G(\bb)$ with $\bb\precdot\ba$ by fixing one of the symbols.
If we follow the same approach for an integer partition~$\ba$ as in Lemma~\ref{lem:H1} with $\Delta(\ba)=1$, then by Lemma~\ref{lem:delta-prec} all partitions~$\bb\precdot\ba$ satisfy~$\Delta(\bb)=0$, except the unique partition~$\bb\precdot\ba$ obtained from~$\ba$ by decreasing the first entry, which satisfies~$\Delta(\bb)=2$ (by the assumptions of Lemma~\ref{lem:H1}, the first entry of~$\ba$ is the unique largest one).
All subgraphs~$G(\bb)$ with $\Delta(\bb)=0$ are not Hamilton-connected, but only 1-laceable, so any Hamilton path in them starts with a vertex with first symbol~1 and ends with a vertex with first symbol distinct from~1.
When constructing a Hamilton path in~$G(\ba)$, we process those subgraphs in pairs, so that the path through any pair starts and ends with a vertex with first symbol~1, and every such pair is surrounded by vertices with fixed symbol equal to~1.
As a consequence, we need to further partition the unique subgraph~$G(\bb)$ with~$\Delta(\bb)=2$ by fixing yet another symbol, and split it into subgraphs~$G(\bc)$ with $\bc\precdot\bb$, all of which satisfy~$\Delta(\bc)\geq 1$ and $G(\bc)\in\cH$ by the assumptions of the lemma, and those building blocks are inserted between the aforementioned pairs.
The precise order in which the building blocks for a Hamilton path in~$G(\ba)$ are arranged depends on the parity of the number~$k$ of symbols, and on the start and end vertices of the desired Hamilton path, which results in several cases.

\begin{proof}[Proof of Lemma~\ref{lem:H1}]
As $\ba\notin \{(1,1,1),(2,1,1,1),(2,2,1)\}$ we know that $a_1\geq 3$, and combining this with $a_1+\Delta(\ba)=\sum_{i=2}^k a_i$, $\Delta(\ba)=1$ and $a_1\geq a_2$ yields $\sum_{i=2}^k a_i\geq 4$ and $\sum_{i=3}^k a_i\geq 1$, in particular $k\geq 3$ (recall~\eqref{eq:Delta}).
\todo{In fact, we even have $a_1\geq 4$, but only here and not later anymore in the proof of Lemmas~\ref{lem:H1'} and~\ref{lem:H1''}.}
Moreover, we have $a_1>a_2$ by the assumption that $\ba\neq (\alpha,\alpha,1)$ for $\alpha:=(n-1)/2$.

Let $x,y$ be two distinct vertices in~$G(\ba)$.
If there is an index~$i>1$ with $x_i=1$ and $y_i\neq 1$, then we let $\icheck$ be this index.
Otherwise we fix an arbitrary index~$\icheck>1$ with $x_\icheck=y_\icheck=1$, which is possible as $a_1\geq 3$.

We first consider the case that $k\geq 3$ is odd, and we introduce the abbreviation $\ell:=(k-1)/2\geq 1$.
We now distinguish two cases, depending on the value of~$y_\icheck$.

\textbf{Case~(oa):} $y_\icheck=1$.
As $x\neq y$ there is a position $\ihat\in[n]\setminus\{1,\icheck\}$ such that $x_\ihat\neq y_\ihat$, and by our choice of $\icheck$, we can assume that $x_\ihat,y_\ihat>1$.
\footnote{The assumption $x_\ihat,y_\ihat>1$ will not be used in this proof, but in the proof of Lemma~\ref{lem:H1'}.}

As $k\geq 3$, $a_1\geq 3$, and $\sum_{i=2}^k a_i\geq 4$, there are two indices $i_1,i_2\in[n]\setminus\{1,\ihat,\icheck\}$ such that $x_{i_1}\neq x_{i_2}$.
Similarly, there are two indices $i_3,i_4\in[n]\setminus\{1,\ihat,\icheck\}$ such that $y_{i_3}\neq y_{i_4}$.

By Lemma~\ref{lem:matching}~(oa) there are two permutations $\pi,\rho$ on~$[k]$ satisfying the following conditions:
\begin{enumerate}[label=(\arabic*),leftmargin=8mm, topsep=0mm, noitemsep]
\item $\pi_1=x_\icheck=1$, $\rho_1=x_\ihat$, and $\rho_k=y_\ihat$;
\item $\pi_{2j}\neq \rho_j$ and $\pi_{2j+1}\neq \rho_{j+1}$ for all $1\leq j\leq \ell$.
\end{enumerate}

We choose multiset permutations $u^j,v^j\in \Pi(\ba)$, $j=1,\ldots,k$, and $\uhat^j,\vhat^j,\ucheck^j,\vcheck^j\in\Pi(\ba)$, $j=1,\ldots,\ell$, satisfying the following conditions; see Figure~\ref{fig:H1odd}~(oa):
\begin{itemize}[leftmargin=8mm, noitemsep]
\item[(i)] $u^1=x$, $v^1_{i_1}=x_{i_2}$, and $v^k=y$, $u^k_{i_3}=y_{i_4}$;
\item[(ii)] conditions~$(\phi^1_j)$, $(\phi^2_j)$, and $(\phi^3_j)$ for all $1\leq j\leq \ell$;
\item[(iii)] condition~$(\phi^4_j)$ for all $\ell<j<k$;
\item[(iv)] condition~$(\psi_{\ell+1})$,
\end{itemize}
where we use the following abbreviations:
\begin{itemize}[leftmargin=12mm, noitemsep]
\item[$(\phi^1_j):$] $\uhat^j_1=v^j_\icheck=1$, $\uhat^j_\icheck=v^j_1=\pi_{2j}$, $\uhat^j_\ihat=v^j_\ihat=\rho_j$, and $\uhat^j_i=v^j_i$ for all $i\in[n]\setminus\{1,\icheck,\ihat\}$;
\item[$(\phi^2_j):$] $\ucheck^j_1=\vhat^j_\icheck=\pi_{2j}$, $\ucheck^j_\icheck=\vhat^j_1=\pi_{2j+1}$, and $\ucheck^j_i=\vhat^j_i$ for all $i\in[n]\setminus\{1,\icheck\}$;
\item[$(\phi^3_j):$] $u^{j+1}_1=\vcheck^j_\icheck=\pi_{2j+1}$, $u^{j+1}_\icheck=\vcheck^j_1=1$, $u^{j+1}_\ihat=\vcheck^j_\ihat=\rho_{j+1}$, and $u^{j+1}_i=\vcheck^j_i$ for all $i\in[n]\setminus\{1,\icheck,\ihat\}$;
\item[$(\phi^4_j):$] $u^{j+1}_1=v^j_\ihat=\rho_j$, $u^{j+1}_\ihat=v^j_1=\rho_{j+1}$, $u^{j+1}_\icheck=v^j_\icheck=1$, and $u^{j+1}_i=v^j_i$ for all $i\in[n]\setminus\{1,\icheck,\ihat\}$;
\item[$(\psi_j):$] $u^j_\itilde\neq v^j_\itilde$ for some $\itilde\in[n]\setminus\{1,\icheck,\ihat\}$.
\end{itemize}

\begin{figure}
\centerline{
\includegraphics[page=8]{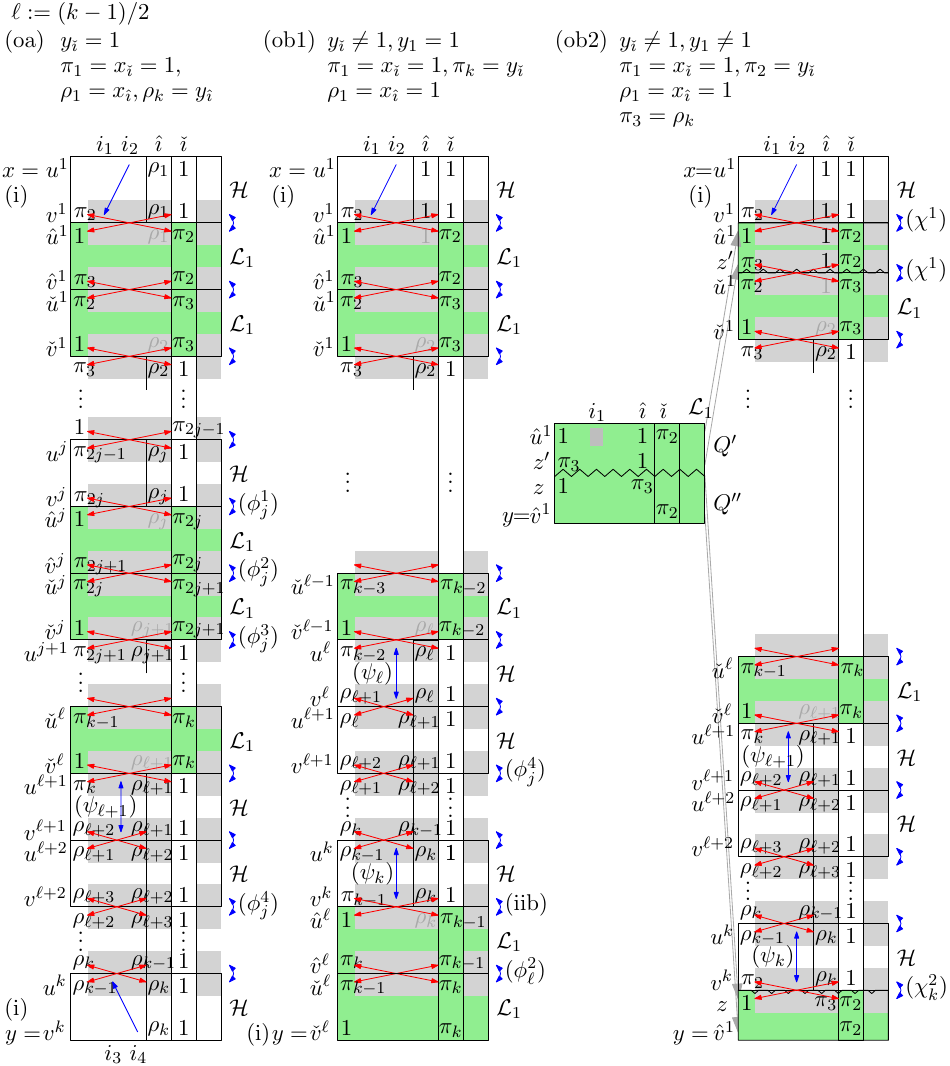}
}
\caption{Illustration of the proof of Lemma~\ref{lem:H1} in the case when $k$ is odd.}
\label{fig:H1odd}
\end{figure}

Observe that condition~$(\phi^1_j)$ requires that both~$v^j$ and~$\uhat^j$ contain the symbols~$\pi_{2j}$ and~$\rho_j$, and condition~$(\phi^3_j)$ requires that both~$\vcheck^j$ and~$u^{j+1}$ contain the symbols~$\pi_{2j+1}$ and~$\rho_{j+1}$, and this is possible for all $1\leq j\leq \ell$ as $\pi$ and~$\rho$ satisfy condition~(2) from before, which ensures that these two symbols are distinct.
Also note that $u^1\neq v^1$ and $u^k\neq v^k$ by~(i) and the choice of $i_1,i_2$ and $i_3,i_4$, respectively, and $u^j_1\neq v^j_1$ for all $1<j<k$ by~(ii)+(iii)+(iv).
Moreover, we have $\uhat^j_1\neq \vhat^j_1$ and $\ucheck^j_1\neq \vcheck^j_1$ for all $1\leq j\leq \ell$ by~(ii).

For $j=1,\ldots,k$ we define $\bb:=(a_1-1,a_2,\ldots,a_k)$, $\bc':=(a_1-1,\ldots,a_{\rho_j}-1,\ldots,a_k)$, and $\bc:=\varphi(\bc')$, which satisfy $\bc\precdot\bb\precdot\ba$ and $\Delta(\bb)=2$ and $\Delta(\bc)=1$, and we consider a Hamilton path~$P_j$ in~$G^{(\icheck,\ihat),(1,\rho_j)}(\ba)\simeq G(\bc)$ from~$u^j$ to~$v^j$, which exists by the assumption~$G(\bc)\in\cH$.

For $j=1,\ldots,\ell$ we define $\ba':=(a_1,\ldots,a_{\pi_{2j}}-1,\ldots,a_k)$ and $\bb:=\varphi(\ba')$, which satisfies $\bb\precdot\ba$ and $\Delta(\bb)=0$, and we consider a Hamilton path~$\Phat_j$ in the graph~$G^{\icheck,\pi_{2j}}(\ba)\simeq G(\bb)$ from~$\uhat^j$ to~$\vhat^j$, which exists by the assumption~$G(\bb)\in\cL_1$, using that~$\uhat^j\in\Pi(\ba)^{1,1}$ and~$\vhat^j\notin\Pi(\ba)^{1,1}$ (recall that $\uhat^j_1=1$ and $\vhat^j_1=\pi_{2j+1}\neq 1$ by~(ii)).

Similarly, for $j=1,\ldots,\ell$ we define $\ba':=(a_1,\ldots,a_{\pi_{2j+1}}-1,\ldots,a_k)$ and $\bb:=\varphi(\ba')$, and we consider a Hamilton path~$\Pcheck_j$ in the graph~$G^{\icheck,\pi_{2j+1}}(\ba)\simeq G(\bb)$ from~$\ucheck^j$ to~$\vcheck^j$, which exists by the assumption~$G(\bb)\in\cL_1$, using that~$\ucheck^j\notin\Pi(\ba)^{1,1}$ and~$\vcheck^j\in\Pi(\ba)^{1,1}$ (recall that $\ucheck^j_1=\pi_{2j}\neq 1$ and $\vcheck^j_1=1$ by~(ii)).

The concatenation $(P_1\Phat_1\Pcheck_1) (P_2\Phat_2\Pcheck_2)\cdots (P_\ell\Phat_\ell\Pcheck_\ell) P_{\ell+1}P_{\ell+2}\cdots P_k$ is a Hamilton path in~$G(\ba)$ from~$x$ to~$y$.

\textbf{Case~(ob):} $y_\icheck\neq 1$.
We choose an index~$\ihat\in[n]\setminus\{1,\icheck\}$ for which $x_\ihat=1$, and two indices $i_1,i_2\in[n]\setminus\{1,\ihat,\icheck\}$ such that $x_{i_1}\neq x_{i_2}$.
We consider two subcases depending on the value of~$y_1$.

\textbf{Case~(ob1):} $y_1=1$.
By Lemma~\ref{lem:matching}~(ob1) there are two permutations $\pi,\rho$ on~$[k]$ satisfying the following conditions:
\begin{enumerate}[label=(\arabic*),leftmargin=8mm, topsep=0mm, noitemsep]
\item $\pi_1=x_\icheck=1$, $\pi_k=y_\icheck$, $\rho_1=x_\ihat=1$;
\item $\pi_{2j}\neq \rho_j$ and $\pi_{2j+1}\neq \rho_{j+1}$ for all $1\leq j<\ell$, and $\pi_{k-1}\neq \rho_k$.
\end{enumerate}

We choose multiset permutations $u^j,v^j\in \Pi(\ba)$, $j=1,\ldots,k$, and $\uhat^j,\vhat^j,\ucheck^j,\vcheck^j\in\Pi(\ba)$, $j=1,\ldots,\ell$, satisfying the following conditions; see Figure~\ref{fig:H1odd}~(ob1):
\begin{itemize}[leftmargin=8mm, noitemsep]
\item[(i)] $u^1=x$, $v^1_{i_1}=x_{i_2}$, and $\vcheck^\ell=y$;
\item[(iia)] conditions~$(\phi^1_j)$, $(\phi^2_j)$, $(\phi^3_j)$ for all $1\leq j<\ell$, and condition~$(\phi^2_\ell)$;
\item[(iib)] $\uhat^\ell_1=v^k_\icheck=1$, $\uhat^\ell_\icheck=v^k_1=\pi_{k-1}$, $\uhat^\ell_\ihat=v^k_\ihat=\rho_k$, and $\uhat^\ell_i=v^k_i$ for all $i\in[n]\setminus\{1,\icheck,\ihat\}$;
\item[(iii)] condition~$(\phi^4_j)$ for all $\ell\leq j<k$;
\item[(iv)] condition~$(\psi_j)$ for $j\in\{\ell,k\}$.
\end{itemize}
The argument that all pairs $(u^j,v^j)$, $j=1,\ldots,k$, and $(\uhat^j,\vhat^j)$ and $(\ucheck^j,\vcheck^j)$, $j=1,\ldots,\ell$, are distinct, is straightforward.

As before in case~(oa), we consider Hamilton paths~$P_j$ from~$u^j$ to~$v^j$ for all $j=1,\ldots,k$, Hamilton paths~$\Phat_j$ and~$\Pcheck_j$ from~$\uhat^j$ to~$\vhat^j$, or from~$\ucheck^j$ to~$\vcheck^j$, respectively, for all $j=1,\ldots,\ell$.
The concatenation $(P_1\Phat_1\Pcheck_1) (P_2\Phat_2\Pcheck_2)\cdots (P_{\ell-1}\Phat_{\ell-1}\Pcheck_{\ell-1}) P_\ell P_{\ell+1}\cdots P_k (\Phat_\ell\Pcheck_\ell)$ is a Hamilton path in~$G(\ba)$ from~$x$ to~$y$.

\textbf{Case~(ob2):} $y_1\neq 1$.
By Lemma~\ref{lem:matching}~(ob2) there are two permutations $\pi,\rho$ on~$[k]$ satisfying the following conditions:
\begin{enumerate}[label=(\arabic*),leftmargin=8mm, topsep=0mm, noitemsep]
\item $\pi_1=x_\icheck=1$, $\pi_2=y_\icheck$, $\rho_1=x_\ihat=1$, and $\pi_3=\rho_k$;
\item $\pi_{2j}\neq \rho_j$ and $\pi_{2j+1}\neq \rho_{j+1}$ for all $1\leq j\leq \ell$.
\end{enumerate}

We choose two multiset permutations $\uhat^1,\vhat^1\in\Pi(\ba)$ such that $\uhat^1_1=1$, $\uhat^1_{i_1}=x_{i_2}$, $\uhat^1_\ihat=1$, $\uhat^1_\icheck=\pi_2$, and $\vhat^1=y$.
We define $\ba':=(a_1,\ldots,a_{\pi_2}-1,\ldots,a_k)$ and $\bb:=\varphi(\ba')$, which satisfies $\bb\precdot\ba$ and $\Delta(\bb)=0$, and we consider a Hamilton path~$Q$ in the graph~$G^{\icheck,\pi_2}(\ba)\simeq G(\bb)$ from~$\uhat^1$ to~$\vhat^1$, which exists by the assumption~$G(\bb)\in\cL_1$, using that~$\uhat^1\in\Pi(\ba)^{1,1}$ and~$\vhat^1\notin\Pi(\ba)^{1,1}$ (recall that $\uhat^1_1=1$ and $\vhat^1_1=y_1\neq 1$).
Let $z$ be the first vertex along this path from~$\uhat^1$ to~$\vhat^1$ for which $z_{\ihat}=\pi_3$, and let $z'$ be the predecessor of~$z$ on the path.
By the definition of~$z$ we have $z'_\ihat\neq\pi_3$, and consequently~$z'_1=\pi_3$.
By Lemma~\ref{lem:D0}, we thus obtain $z'_\ihat=z_1=1$.
Let $Q'$ be the subpath of~$Q$ from~$\uhat^1$ to~$z'$, and let $Q''$ be the subpath of~$Q$ from~$z$ to~$\vhat^1$.

We choose multiset permutations $u^j,v^j\in \Pi(\ba)$, $j=1,\ldots,k$, and $\uhat^j,\vhat^j\in\Pi(\ba)$, $j=2,\ldots,\ell$, and $\ucheck^j,\vcheck^j\in\Pi(\ba)$, $j=1,\ldots,\ell$, satisfying the following conditions; see Figure~\ref{fig:H1odd}~(ob2):
\begin{itemize}[leftmargin=8mm, noitemsep]
\item[(i)] $u^1=x$, and conditions~$(\chi^1)$ and~$(\chi^2_k)$;
\item[(ii)] conditions $(\phi^1_j)$ and~$(\phi^2_j)$ for all $1<j\leq \ell$, and condition~$(\phi^3_j)$ for all $1\leq j\leq \ell$;
\item[(iii)] condition~$(\phi^4_j)$ for all $\ell<j<k$;
\item[(iv)] condition~$(\psi_j)$ for $j\in\{\ell+1,k\}$,
\end{itemize}
where we introduce the abbreviations:
\begin{itemize}[leftmargin=12mm, noitemsep]
\item[$(\chi^1):$] $v^1_1=\uhat^1_\icheck=\pi_2$, $v^1_\icheck=\uhat^1_1=1$, $v^1_i=\uhat^1_i$, and $\ucheck^1_1=z'_\icheck=\pi_2$, $\ucheck^1_\icheck=z'_1=\pi_3$, $\ucheck^1_i=z'_i$ for all $i\in[n]\setminus\{1,\icheck\}$;
\item[$(\chi^2_j)$:] $v^j_1=z_\icheck=\pi_2$, $v^j_\icheck=z_1=1$, $v^j_i=z_i$ for all $i\in[n]\setminus\{1,\icheck\}$.
\end{itemize}

The argument that all pairs $(u^j,v^j)$, $j=1,\ldots,k$, and $(\uhat^j,\vhat^j)$, $j=2,\ldots,\ell$, and $(\ucheck^j,\vcheck^j)$, $j=1,\ldots,\ell$, are distinct, is straightforward.

As before in case~(oa), we consider Hamilton paths~$P_j$ from~$u^j$ to~$v^j$ for all $j=1,\ldots,k$, Hamilton paths~$\Phat_j$ from~$\uhat^j$ to~$\vhat^j$ for all $j=2,\ldots,\ell$, and Hamilton paths~$\Pcheck_j$ from~$\ucheck^j$ to~$\vcheck^j$ for all $j=1,\ldots,\ell$.

The concatenation $(P_1Q'\Pcheck_1)(P_2\Phat_2\Pcheck_2)(P_3\Phat_3\Pcheck_3)\cdots (P_\ell\Phat_\ell\Pcheck_\ell) P_{\ell+1}\cdots P_k Q''$ is a Hamilton path in~$G(\ba)$ from~$x$ to~$y$.

It remains to consider the case that $k\geq 4$ is even, and we introduce the abbreviation $\ell:=k/2\geq 2$.
We now distinguish two cases, depending on the value of~$y_\icheck$.

\textbf{Case~(eb):} $y_\icheck\neq 1$.
We choose an index~$\ihat\in[n]\setminus\{1,\icheck\}$ for which $x_\ihat=1$, and two indices $i_1,i_2\in[n]\setminus\{1,\ihat,\icheck\}$ such that $x_{i_1}\neq x_{i_2}$.
We consider two subcases depending on the value of~$y_1$.

\textbf{Case~(eb2):} $y_1\neq 1$.
By Lemma~\ref{lem:matching}~(eb2) there are two permutations $\pi,\rho$ on~$[k]$ satisfying the following conditions:
\begin{enumerate}[label=(\arabic*),leftmargin=8mm, topsep=0mm, noitemsep]
\item $\pi_1=x_\icheck=1$, $\pi_k=y_\icheck$, and $\rho_1=x_\ihat=1$;
\item $\pi_{2j}\neq \rho_j$ and $\pi_{2j+1}\neq \rho_{j+1}$ for all $1\leq j<\ell$, and $\pi_k\neq \rho_k$.
\end{enumerate}

We choose multiset permutations $u^j,v^j\in \Pi(\ba)$, $j=1,\ldots,k$, and $\uhat^j,\vhat^j\in\Pi(\ba)$, $j=1,\ldots,\ell$, and $\ucheck^j,\vcheck^j\in\Pi(\ba)$, $j=1,\ldots,\ell-1$, satisfying the following conditions; see Figure~\ref{fig:H1even}~(eb2):
\begin{itemize}[leftmargin=8mm, noitemsep]
\item[(i)] $u^1=x$, $v^1_{i_1}=x_{i_2}$, and $\vhat^\ell=y$;
\item[(iia)] conditions~$(\phi^1_j)$, $(\phi^2_j)$, and $(\phi^3_j)$ for all $1\leq j<\ell$;
\item[(iib)] $\uhat^\ell_1=v^k_\icheck=1$, $\uhat^\ell_\icheck=v^k_1=\pi_k$, $\uhat^\ell_\ihat=v^k_\ihat=\rho_k$, and $\uhat^\ell_i=v^k_i$ for all $i\in[n]\setminus\{1,\icheck,\ihat\}$;
\item[(iii)] condition~$(\phi^4_j)$ for all $\ell\leq j<k$;
\item[(iv)] condition~$(\psi_j)$ for $j\in\{\ell,k\}$.
\end{itemize}

\begin{figure}
\centerline{
\includegraphics[page=10]{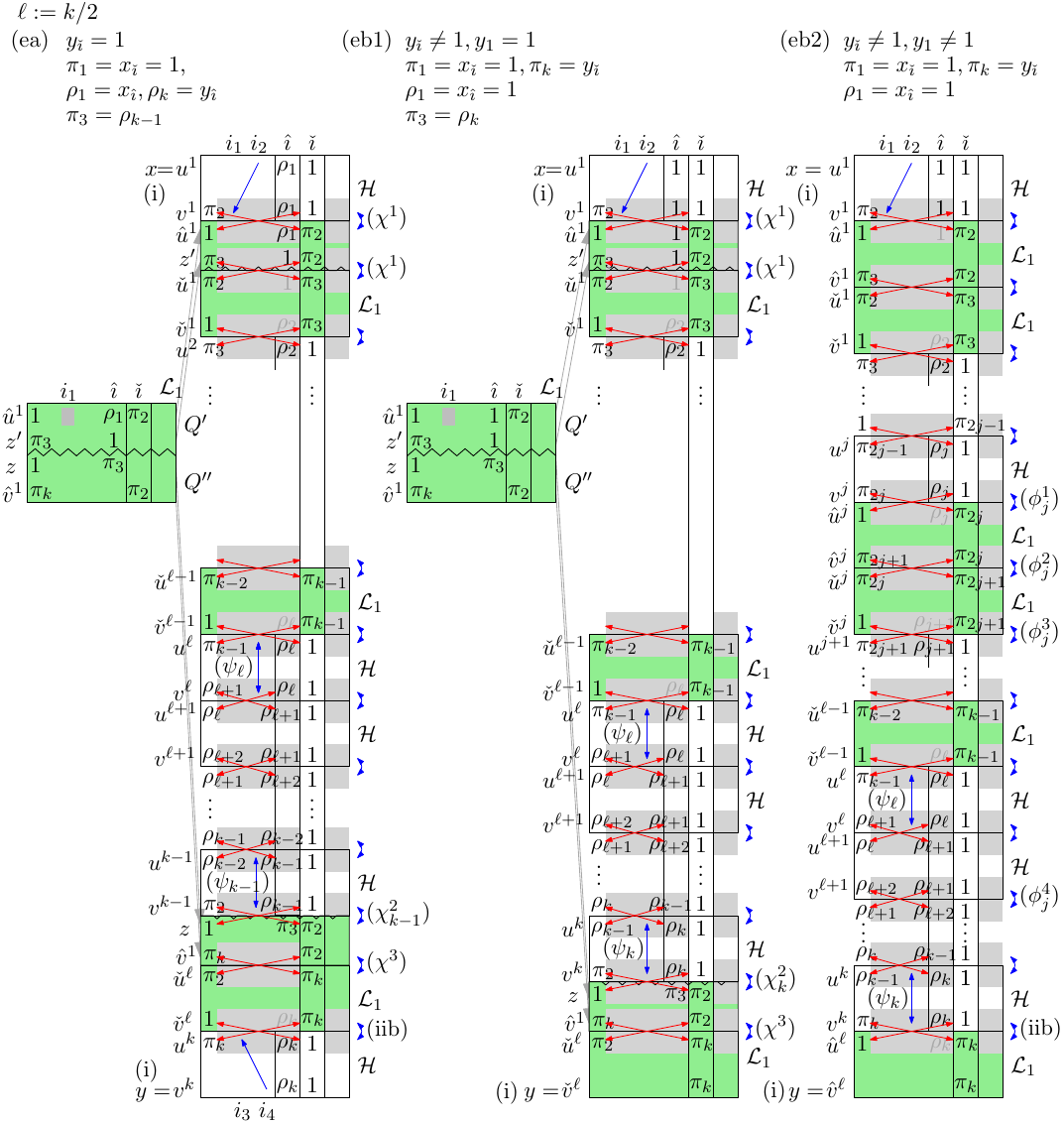}
}
\caption{Illustration of the proof of Lemma~\ref{lem:H1} in the case when $k$ is even.}
\label{fig:H1even}
\end{figure}

The argument that all pairs $(u^j,v^j)$, $j=1,\ldots,k$, and $(\uhat^j,\vhat^j)$, $j=1,\ldots,\ell$, and $(\ucheck^j,\vcheck^j)$, $j=1,\ldots,\ell-1$, are distinct, is straightforward.

For $j=1,\ldots,k$ we define $\bb:=(a_1-1,a_2,\ldots,a_k)$, $\bc':=(a_1-1,\ldots,a_{\rho_j}-1,\ldots,a_k)$, and $\bc:=\varphi(\bc')$, which satisfy $\bc\precdot\bb\precdot\ba$ and $\Delta(\bb)=2$ and $\Delta(\bc)=1$, and we consider a Hamilton path~$P_j$ in~$G^{(\icheck,\ihat),(1,\rho_j)}(\ba)\simeq G(\bc)$ from~$u^j$ to~$v^j$, which exists by the assumption~$G(\bc)\in\cH$.

For $j=1,\ldots,\ell$ we define $\ba':=(a_1,\ldots,a_{\pi_{2j}}-1,\ldots,a_k)$ and $\bb:=\varphi(\ba')$, which satisfies $\bb\precdot\ba$ and $\Delta(\bb)=0$, and we consider a Hamilton path~$\Phat_j$ in the graph~$G^{\icheck,\pi_{2j}}(\ba)\simeq G(\bb)$ from~$\uhat^j$ to~$\vhat^j$, which exists by the assumption~$G(\bb)\in\cL_1$, using that~$\uhat^j\in\Pi(\ba)^{1,1}$ and~$\vhat^j\notin\Pi(\ba)^{1,1}$ (recall that $\uhat^j_1=1$ and $\vhat^j_1=\pi_{2j+1}\neq 1$ for $j=1,\ldots,\ell-1$ by~(ii) and $\uhat^\ell_1=1$ and $\vhat^\ell_1=y_1\neq 1$ by~(iib) and~(i), respectively).

Similarly, for $j=1,\ldots,\ell-1$ we define $\ba':=(a_1,\ldots,a_{\pi_{2j+1}}-1,\ldots,a_k)$ and $\bb:=\varphi(\ba')$, and we consider a Hamilton path~$\Pcheck_j$ in the graph~$G^{\icheck,\pi_{2j+1}}(\ba)\simeq G(\bb)$ from~$\ucheck^j$ to~$\vcheck^j$, which exists by the assumption~$G(\bb)\in\cL_1$, using that~$\ucheck^j\notin\Pi(\ba)^{1,1}$ and~$\vcheck^j\in\Pi(\ba)^{1,1}$ (recall that $\ucheck^j_1=\pi_{2j}\neq 1$ and $\vcheck^j_1=1$ by~(iia)).

The concatenation $(P_1\Phat_1\Pcheck_1) (P_2\Phat_2\Pcheck_2)\cdots (P_{\ell-1}\Phat_{\ell-1}\Pcheck_{\ell-1}) P_\ell P_{\ell+1}\cdots P_{k-1}P_k \Phat_\ell$ is a Hamilton path in~$G(\ba)$ from~$x$ to~$y$.

\textbf{Case~(eb1):} $y_1=1$.
By Lemma~\ref{lem:matching}~(eb1) there are two permutations $\pi,\rho$ on~$[k]$ satisfying the following conditions:
\begin{enumerate}[label=(\arabic*),leftmargin=8mm, topsep=0mm, noitemsep]
\item $\pi_1=x_\icheck=1$, $\pi_k=y_\icheck$, $\rho_1=x_\ihat=1$, and $\pi_3=\rho_k$;
\item $\pi_{2j}\neq \rho_j$ and $\pi_{2j+1}\neq \rho_{j+1}$ for all $1\leq j<\ell$.
\end{enumerate}

We choose two multiset permutations $\uhat^1,\vhat^1\in\Pi(\ba)$ such that $\uhat^1_1=1$, $\uhat^1_{i_1}=x_{i_2}$, $\uhat^1_\ihat=\rho_1=1$, $\uhat^1_\icheck=\pi_2$, $\vhat^1_1=\pi_k$, and $\vhat^1_\icheck=\pi_2$.
We define $\ba':=(a_1,\ldots,a_{\pi_2}-1,\ldots,a_k)$ and $\bb:=\varphi(\ba')$, which satisfies $\bb\precdot\ba$ and $\Delta(\bb)=0$, and we consider a Hamilton path~$Q$ in the graph~$G^{\icheck,\pi_2}(\ba)\simeq G(\bb)$ from~$\uhat^1$ to~$\vhat^1$, which exists by the assumption~$G(\bb)\in\cL_1$, using that~$\uhat^1\in\Pi(\ba)^{1,1}$ and~$\vhat^1\notin\Pi(\ba)^{1,1}$ (recall that $\uhat^1_1=1$ and $\vhat^1_1=\pi_k\neq 1$).
Let $z$ be the first vertex along this path from~$\uhat^1$ to~$\vhat^1$ for which $z_{\ihat}=\pi_3$, and let $z'$ be the predecessor of~$z$ on the path.
By the definition of~$z$ we have $z'_\ihat\neq\pi_3$, and consequently~$z'_1=\pi_3$.
By Lemma~\ref{lem:D0}, we thus obtain $z'_\ihat=z_1=1$.
Let $Q'$ be the subpath of~$Q$ from~$\uhat^1$ to~$z'$, and let $Q''$ be the subpath of~$Q$ from~$z$ to~$\vhat^1$.

We choose multiset permutations $u^j,v^j\in \Pi(\ba)$, $j=1,\ldots,k$, and $\uhat^j,\vhat^j\in\Pi(\ba)$, $j=2,\ldots,\ell-1$, and $\ucheck^j,\vcheck^j\in\Pi(\ba)$, $j=1,\ldots,\ell$, satisfying the following conditions; see Figure~\ref{fig:H1even}~(eb1):
\begin{itemize}[leftmargin=8mm, noitemsep]
\item[(i)] $u^1=x$, $\vcheck^\ell=y$, and conditions~$(\chi^1)$, $(\chi^2_k)$, and $(\chi^3)$;
\item[(ii)] conditions $(\phi^1_j)$ and~$(\phi^2_j)$ for all $1<j<\ell$, and condition~$(\phi^3_j)$ for all $1\leq j<\ell$;
\item[(iii)] condition~$(\phi^4_j)$ for all $\ell\leq j<k$;
\item[(iv)] condition~$(\psi_j)$ for $j\in\{\ell,k\}$,
\end{itemize}
where we introduce the abbreviation:
\begin{itemize}[leftmargin=12mm, noitemsep]
\item[$(\chi^3):$] $\ucheck^\ell_1=\vhat^1_\icheck=\pi_2$, $\ucheck^\ell_\icheck=\vhat^1_1=\pi_k$, $\ucheck^\ell_i=\vhat^1_i$ for all $i\in[n]\setminus\{1,\icheck\}$.
\end{itemize}
The argument that all pairs $(u^j,v^j)$, $j=1,\ldots,k$, and $(\uhat^j,\vhat^j)$, $j=2,\ldots,\ell-1$, and $(\ucheck^j,\vcheck^j)$, $j=1,\ldots,\ell$, are distinct, is straightforward.

As before in case~(eb2), we consider Hamilton paths~$P_j$ from~$u^j$ to~$v^j$ for all $j=1,\ldots,k$, Hamilton paths~$\Phat_j$ from~$\uhat^j$ to~$\vhat^j$ for all $j=2,\ldots,\ell-1$, and Hamilton paths~$\Pcheck_j$ from~$\ucheck^j$ to~$\vcheck^j$ for all $j=1,\ldots,\ell-1$.
In addition, we consider a Hamilton path~$\Pcheck_\ell$ from~$\ucheck^\ell$ to~$\vcheck^\ell$, which exists by the assumption~$G(\bb)\in\cL_1$, using that~$\ucheck^\ell\notin\Pi(\ba)^{1,1}$ and~$\vcheck^\ell\in\Pi(\ba)^{1,1}$ (recall that $\ucheck^\ell_1=\pi_2\neq 1$ and $\vcheck^\ell_1=y_1=1$ by~(i)).

The concatenation $(P_1Q'\Pcheck_1)(P_2\Phat_2\Pcheck_2)(P_3\Phat_3\Pcheck_3)\cdots (P_{\ell-1}\Phat_{\ell-1}\Pcheck_{\ell-1}) P_\ell P_{\ell+1}\cdots P_{k-1}P_k Q''\Pcheck_\ell$ is a Hamilton path in~$G(\ba)$ from~$x$ to~$y$.

\textbf{Case~(ea):} $y_\icheck=1$.
As $x\neq y$ there is a position $\ihat\in[n]\setminus\{1,\icheck\}$ such that $x_\ihat\neq y_\ihat$, and by our choice of $\icheck$, we can assume that $x_\ihat,y_\ihat>1$.
\footnote{The assumption $x_\ihat,y_\ihat>1$ will not be used in this proof, but in the proof of Lemma~\ref{lem:H1''}.}

We fix two indices $i_1,i_2\in[n]\setminus\{1,\ihat,\icheck\}$ such that $x_{i_1}\neq x_{i_2}$, and two indices $i_3,i_4\in[n]\setminus\{1,\ihat,\icheck\}$ such that $y_{i_3}\neq y_{i_4}$.

By Lemma~\ref{lem:matching}~(ea) there are two permutations $\pi,\rho$ on~$[k]$ satisfying the following conditions:
\begin{enumerate}[label=(\arabic*),leftmargin=8mm, topsep=0mm, noitemsep]
\item $\pi_1=x_\icheck=1$, $\rho_1=x_\ihat$, $\rho_k=y_\ihat$, and $\pi_3=\rho_{k-1}$;
\item $\pi_{2j}\neq \rho_j$ and $\pi_{2j+1}\neq \rho_{j+1}$ for all $1\leq j<\ell$, and $\pi_k\neq \rho_k$.
\end{enumerate}

We choose two multiset permutations $\uhat^1,\vhat^1\in\Pi(\ba)$ as before in case~(eb1), and we construct a Hamilton path~$Q$ from~$\uhat^1$ to~$\vhat^1$ and its subpaths~$Q'$ from~$\uhat^1$ to~$z'$ and~$Q''$ from~$z$ to~$\vhat^1$, as before.

We choose multiset permutations $u^j,v^j\in \Pi(\ba)$, $j=1,\ldots,k$, and $\uhat^j,\vhat^j\in\Pi(\ba)$, $j=2,\ldots,\ell-1$, and $\ucheck^j,\vcheck^j\in\Pi(\ba)$, $j=1,\ldots,\ell$, satisfying the following conditions; see Figure~\ref{fig:H1even}~(ea):
\begin{itemize}[leftmargin=8mm, noitemsep]
\item[(i)] $u^1=x$, $v^k=y$, $u^k_{i_3}=x_{i_4}$, and conditions~$(\chi^1)$, $(\chi^2_{k-1})$, and $(\chi^3)$;
\item[(iia)] conditions $(\phi^1_j)$ and~$(\phi^2_j)$ for all $1<j<\ell$, and condition~$(\phi^3_j)$ for all $1\leq j<\ell$;
\item[(iib)] $u^k_1=\vcheck^\ell_\icheck=\pi_k$, $u^k_\icheck=\vcheck^\ell_1=1$, $u^k_\ihat=\vcheck^\ell_\ihat=\rho_k$, $u^k_i=\vcheck^\ell_i$ for all $i\in[n]\setminus\{1,\icheck,\ihat\}$;
\item[(iii)] condition~$(\phi^4_j)$ for all $\ell\leq j<k-1$;
\item[(iv)] condition~$(\psi_j)$ for $j\in\{\ell,k-1\}$.
\end{itemize}
The argument that all pairs $(u^j,v^j)$, $j=1,\ldots,k$, and $(\uhat^j,\vhat^j)$, $j=2,\ldots,\ell-1$, and $(\ucheck^j,\vcheck^j)$, $j=1,\ldots,\ell$, are distinct, is straightforward.

As before in case~(eb1), we consider Hamilton paths~$P_j$ from~$u^j$ to~$v^j$ for all $j=1,\ldots,k$, Hamilton paths~$\Phat_j$ from~$\uhat^j$ to~$\vhat^j$ for all $j=2,\ldots,\ell-1$, and Hamilton paths~$\Pcheck_j$ from~$\ucheck^j$ to~$\vcheck^j$ for all $j=1,\ldots,\ell$.

The concatenation $(P_1Q'\Pcheck_1)(P_2\Phat_2\Pcheck_2)(P_3\Phat_3\Pcheck_3)\cdots (P_{\ell-1}\Phat_{\ell-1}\Pcheck_{\ell-1}) P_\ell P_{\ell+1}\cdots P_{k-2}P_{k-1} Q''\Pcheck_\ell P_k$ is a Hamilton path in~$G(\ba)$ from~$x$ to~$y$.

This completes the proof of Lemma~\ref{lem:H1}.
\end{proof}

The strategy for proving Lemma~\ref{lem:H1'} and~\ref{lem:H1''} is the same as for the proof of Lemma~\ref{lem:H1}, and these proofs can be found in the appendix.
The crucial difference is that now one or two building blocks $\bc\precdot\bb$ for the unique~$\bb\precdot\ba$ with $\Delta(\bb)=2$ are only assumed to satisfy the weaker property $G(\bc)\in\cL_{12}$, and not $G(\bc)\in\cH$, so we have to work around this by imposing extra conditions on those building blocks.

\begin{proof}[Proof of Lemma~\ref{lem:L12}]
As the symbols~1 and~2 both appear with the same frequency~$\alpha$, it suffices to prove that~$G(\ba)$ has property~$\cL_{12}$ for $s=1$.
Let $x,y$ be two distinct vertices in~$G(\ba)$ such that $x_1=1$ and $y_1=t$ for which there is a position~$\ihat>1$ with $(x_\ihat,y_\ihat)=(p_3(1,t),q_3(1,t))$.

We first consider the case~$t\in\{1,3\}$.
In this case we have $(x_\ihat,y_\ihat)=(p_3(1,t),q_3(1,t))=(3,1)$ by~\eqref{eq:pq}.
We choose multiset permutations $u^j,v^j\in \Pi(\ba)$, $j=1,2,3$, satisfying the following conditions; see Figure~\ref{fig:L1}~(a):
\begin{itemize}[leftmargin=8mm, noitemsep]
\item[(i)] $u^1=x$ and $v^3=y$;
\item[(ii)] $u^2_1=v^1_\ihat=3$, $u^2_\ihat=v^1_1=2$, $u^3_1=v^2_\ihat=2$, $u^3_\ihat=v^2_1=1$, and $u^j_i=v^{j-1}_i$ for $j=2,3$ and all $i\in[n]\setminus\{1,\ihat\}$.
\end{itemize}
Note that $u^1_1\neq v^1_1$, $u^2_1\neq v^2_1$, and $u^3_1\neq v^3_1$ by~(i)+(ii).
Consequently, we have $u^j\neq v^j$ for $j=1,2,3$.

\begin{figure}[b!]
\includegraphics[page=4]{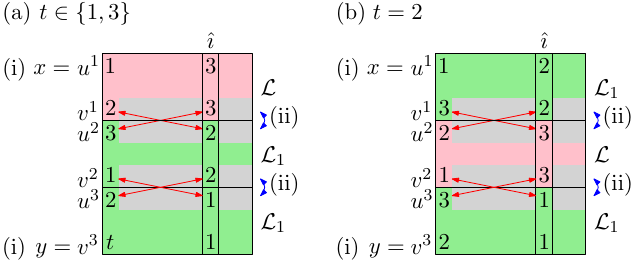}
\caption{Illustration of the proof of Lemma~\ref{lem:L12}.}
\label{fig:L1}
\end{figure}

We define $\bb:=(\alpha,\alpha)\precdot\ba$, and we consider a Hamilton path~$P_1$ in the graph~$G^{\ihat,3}(\ba)\simeq G(\bb)$ from~$u^1$ to~$v^1$, which exists by the assumption $G(\bb)\in\cL$, using that $u^1$ and~$v^1$ belong to the two distinct partition classes of the graph~$G^{\ihat,3}(\ba)$ by Lemma~\ref{lem:kpart} (recall that $(u^1_1,v^1_1)=(1,2)$).
We define $\bb:=(\alpha,\alpha-1,1)\precdot\ba$, and we consider a Hamilton path~$P_2$ in the graph~$G^{\ihat,2}(\ba)\simeq G(\bb)$ from~$u^2$ to~$v^2$, which exists by the assumption~$G(\bb)\in\cL_1$, using that $u^2\notin\Pi(\ba)^{1,1}$ and~$v^2\in\Pi(\ba)^{1,1}$ (recall that $(u^2_1,v^2_1)=(3,1)$).
We define $\ba':=(\alpha-1,\alpha,1)$ and $\bb:=\varphi(\ba')=(\alpha,\alpha-1,1)\precdot\ba$, and we consider a Hamilton path~$P_3$ in the graph~$G^{\ihat,1}(\ba)\simeq G(\bb)$ from~$u^3$ to~$v^3$, which exists by the assumption~$G(\bb)\in\cL_1$, using that $u^3\in\Pi(\ba)^{1,2}$ and~$v^3\notin\Pi(\ba)^{1,2}$ (recall that $(u^3_1,v^3_1)=(2,t)$ and $t\in\{1,3\}$).
Note here that in $G^{\ihat,1}(\ba)$, the symbol~2 is the most frequent one and takes the role of~1 in~$G(\bb)$.

The concatenation $P_1P_2P_3$ is a Hamilton path in~$G(\ba)$ from~$x$ to~$y$.

It remains to consider the case~$t=2$.
In this case we have $(x_\ihat,y_\ihat)=(p_3(1,t),q_3(1,t))=(2,1)$ by~\eqref{eq:pq}.
We choose multiset permutations $u^j,v^j\in \Pi(\ba)$, $j=1,2,3$, satisfying the following conditions; see Figure~\ref{fig:L1}~(b):
\begin{itemize}[leftmargin=8mm, noitemsep]
\item[(i)] $u^1=x$ and $v^3=y$;
\item[(ii)] $u^2_1=v^1_\ihat=2$, $u^2_\ihat=v^1_1=3$, $u^3_1=v^2_\ihat=3$, $u^3_\ihat=v^2_1=1$, and $u^j_i=v^{j-1}_i$ for $j=2,3$ and all $i\in[n]\setminus\{1,\ihat\}$.
\end{itemize}
Note that $u^1_1\neq v^1_1$, $u^2_1\neq v^2_1$, and $u^3_1\neq v^3_1$ by~(i)+(ii).
Consequently, we have $u^j\neq v^j$ for $j=1,2,3$.

We consider a Hamilton path~$P_1$ in~$G^{\ihat,2}(\ba)$ from~$u^1$ to~$v^1$, which exists by the assumption~$G(\alpha,\alpha-1,1)\in\cL_1$, using that~$u^1\in\Pi(\ba)^{1,1}$ and~$v^1\notin\Pi(\ba)^{1,1}$ (recall that $(u^1_1,v^1_1)=(1,3)$).
We also consider a Hamilton path~$P_2$ in~$G^{\ihat,3}(\ba)$ from~$u^2$ to~$v^2$, which exists by the assumption $G(\alpha,\alpha)\in\cL$, using that $u^2$ and~$v^2$ belong to the two distinct partition classes of the graph~$G^{\ihat,3}(\ba)$ by Lemma~\ref{lem:kpart} (recall that $(u^2_1,v^2_1)=(2,1)$).
Finally, we consider a Hamilton path~$P_3$ in~$G^{\ihat,1}(\ba)$ from~$u^3$ to~$v^3$, which exists by the assumption~$G(\alpha,\alpha-1,1)\in\cL_1$, using that~$u^3\notin\Pi(\ba)^{1,2}$ and~$v^3\in\Pi(\ba)^{1,2}$ (recall that $(u^3_1,v^3_1)=(3,2)$).

The concatenation $P_1P_2P_3$ is a Hamilton path in~$G(\ba)$ from~$x$ to~$y$.
\end{proof}

The proof of Lemma~\ref{lem:T'} follows the same strategy as the proof of Theorem~\ref{thm:T} presented in Section~\ref{sec:ideasP}.

\begin{proof}[Proof of Lemma~\ref{lem:T'}]
\begin{figure}
\includegraphics[page=2]{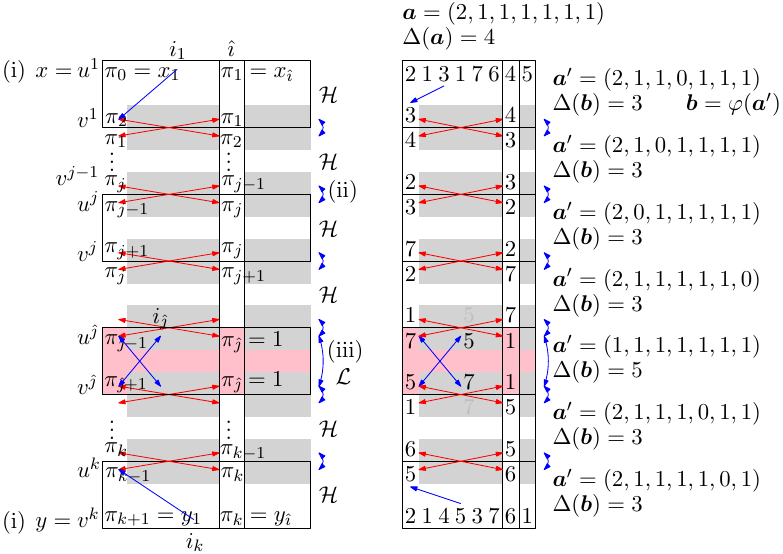}
\caption{Illustration of the proof of Lemma~\ref{lem:T'}.}
\label{fig:T'}
\end{figure}

We argue by induction on~$k\geq 5$, using as a base case that $G(2,1,1,1)\in\cH$ by Table~\ref{tab:base}.
Let $x,y$ be two distinct vertices in~$G(\ba)$.
As $x\neq y$ there is a position $\ihat>1$ such that $x_\ihat\neq y_\ihat$.
As $k\geq 5$, there is an index~$i_1\in[n]\setminus\{1,\ihat\}$ such that $x_{i_1}\notin\{x_1,x_\ihat,y_\ihat\}$.
Similarly, there is an index~$i_k\in[n]\setminus\{1,\ihat\}$ such that $y_{i_k}\notin\{y_1,x_\ihat,y_\ihat,x_{i_1}\}$.
We fix any permutation~$\pi$ on~$[k]$ such that $\pi_1=x_\ihat$, $\pi_2=x_{i_1}$, $\pi_{k-1}=y_{i_k}$, and $\pi_k=y_\ihat$.
For convenience, we also define $\pi_0:=x_1$ and $\pi_{k+1}:=y_1$.
We let $\jhat\in[k]$ be such that $\pi_\jhat=1$.
Then we choose a sequence of multiset permutations $u^j,v^j\in \Pi(\ba)$, $j=1,\ldots,k$, satisfying the following conditions; see Figure~\ref{fig:T'}:
\begin{itemize}[leftmargin=8mm, noitemsep]
\item[(i)] $u^1=x$ and $v^k=y$;
\item[(ii)] $u^j_1=v^{j-1}_\ihat=\pi_{j-1}$, $u^j_\ihat=v^{j-1}_1=\pi_j$, and $u^j_i=v^{j-1}_i$ for all $2\leq j\leq k$ and $i\in[n]\setminus\{1,\ihat\}$.
\item[(iii)] $u^\jhat_i=v^\jhat_i$ for all $i\in[n]\setminus\{1,\ihat,i_\jhat\}$, where $i_\jhat\in[n]\setminus\{1,\ihat\}$ is chosen such that $u^\jhat_{i_\jhat}=\pi_{\jhat+1}$ and $v^\jhat_{i_\jhat}=\pi_{\jhat-1}$ (in particular, if $\jhat=1$ then we have $i_\jhat=i_1$, and if $\jhat=k$ then we have $i_\jhat=i_k$).
\end{itemize}
Recall that $\pi_0\neq \pi_2$ by the choice of~$i_1$, and $\pi_{k+1}\neq \pi_{k-1}$ by the choice of~$i_k$.
Consequently, using that $u^j_1=\pi_{j-1}$ and $v^j_1=\pi_{j+1}$ by~(ii) we obtain that $u^j\neq v^j$ for all $1\leq j\leq k$.

For $j=1,\ldots,k$ we define $\ba':=(a_1,\ldots,a_{\pi_j}-1,\ldots,a_k)$ and $\bb:=\varphi(\ba')$.
Note that $\bb=(2,1^{k-2})$ for all $j\in[k]\setminus\{\jhat\}$, whereas $\bb=1^k$ for $j=\jhat$.
In the first case, there is a Hamilton path~$P_j$ in the graph~$G^{\ihat,\pi_j}(\ba)\simeq G(\bb)=G((2,1^{k-2}))$ from~$u^j$ to~$v^j$ by induction.
In the second case, there is a Hamilton path~$P_j$ in the graph~$G^{\ihat,1}(\ba)\simeq G(\bb)=G(1^k)$ from~$u^j$ to~$v^j$ by Theorem~\ref{thm:T}, using that $u^j$ and~$v^j$ differ only in one transposition by~(ii)+(iii), i.e., they belong to two distinct partition classes of the graph~$G^{\ihat,1}(\ba)$ by Lemma~\ref{lem:bipart}.
The concatenation $P_1P_2\cdots P_k$ is a Hamilton path in~$G(\ba)$ from~$x$ to~$y$.
\end{proof}

\section{Proof of Theorems~\ref{thm:M} and~\ref{thm:M'}}
\label{sec:proof-MM'}

In this section we prove Theorems~\ref{thm:M} and~\ref{thm:M'}.
For this, we focus entirely on integer partitions~$\ba=(a_1,\ldots,a_k)$ that satisfy~$\Delta(\ba)=0$.
By~\eqref{eq:Delta}, this is equivalent to $a_1=\sum_{i=2}^k a_i$, i.e., the first symbol appears equally often as all other symbols combined.
Based on this observation, we will first reparametrize the problem for convenience, changing the meaning of the variables~$n,k,\ba$.

\subsection{Reparametrizing the problem}

The following notations are illustrated in Figure~\ref{fig:para}.

\begin{wrapfigure}{r}{0.45\textwidth}
\includegraphics[page=3]{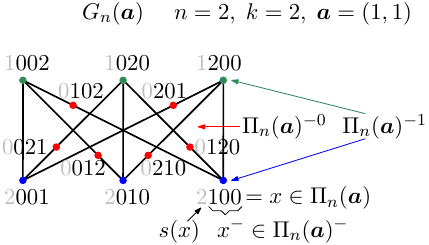}
\caption{The graph from Figure~\ref{fig:part} in the reparametrized notation.
}
\label{fig:para}
\end{wrapfigure}
For any $n\geq 1$ and $k\geq 1$, we consider multiset permutations with $n$ symbols~0, and with $n$ non-zero symbols from the set~$\{1,\ldots,k\}$.
Specifically, the number of each non-zero symbol~$i\in[k]$ is given by~$a_i\geq 1$, where $n=a_1+\cdots+a_k$ and $a_1\geq a_2\geq \ldots\geq a_k$.
We write $\Pi_n(\ba)$, $\ba:=(a_1,\ldots,a_k)$, to denote those multiset permutations.
For any $x\in\Pi_n(\ba)$, we write $x^-$ for the string obtained from~$x$ by omitting the first entry, and we define $\Pi_n(\ba)^-:=\{x^-\mid x\in\Pi_n(\ba)\}$.
Given any~$x\in\Pi_n(\ba)^-$, we can uniquely infer the omitted symbol, and we refer to it as the \emph{suppressed symbol}, denoted $s(x)\in\{0,\ldots,k\}$.
Clearly, $x,y\in\Pi_n(\ba)$ differ in a star transposition if and only if $x^-,y^-\in\Pi_n(\ba)^-$ differ in exactly one position.
We also write $\Pi_n(\ba)^{-0}$ and $\Pi_n(\ba)^{-1}$ for all multiset permutations from~$\Pi_n(\ba)^-$ with suppressed symbol~0 or suppressed symbol distinct from~0, respectively.
We write~$G_n(\ba)$ for the graph with vertex set~$\Pi_n(\ba)^-$, with an edge between any two multiset permutations that differ in exactly one position.
By Lemma~\ref{lem:D0}, any Hamilton cycle in~$G_n(\ba)$ will alternately visit the vertex sets~$\Pi_n(\ba)^{-0}$ and~$\Pi_n(\ba)^{-1}$.

\subsection{Proof of Theorem~\ref{thm:M}}

We now consider the graph~$G_n:=G_n(n)\simeq G(n,n)$ for $n\geq 3$.
The fact that it is Hamilton-laceable for~$n=3$ and~$n=4$ follows from Table~\ref{tab:base}.
For the rest of this section, we will therefore assume that~$n\geq 5$.
We follow the proof strategy outlined in Section~\ref{sec:ideas0}.

\subsubsection{Translation into the hypercube}

We introduce the abbreviations~$B_n:=\Pi_n(n)^{-1}$ and~$B_n':=\Pi_n(n)^{-0}$ for the sets of bitstrings of length~$2n-1$ with exactly $n-1$ or $n$ many 1s, respectively.
Note that the graph~$G_n$ is the subgraph of the $(2n-1)$-dimensional hypercube induced by the `middle levels' $B_n$ and~$B_n'$.
In the following we will show that~$G_n$, $n\geq 5$, is Hamilton-laceable.

We first reduce the number of different pairs of vertices $x\in B_n$, $y\in B_n'$ that we need to connect by a Hamilton path to only to $n$ cases, one for each possible Hamming distance $d:=d(x,y)$, where $1\leq d \leq 2n-1$ is odd.

\begin{lemma}
\label{lem:auto}
For any two pairs of vertices $(x_1,y_1),(x_2,y_2)\in B_n\times B_n'$ with $d(x_1,y_1)=d(x_2,y_2)$ there is an automorphism~$\varphi$ of $G_n$ such that $\varphi(x_1)=x_2$ and $\varphi(y_1)=y_2$.
\end{lemma}

\begin{proof}
Take an arbitrary permutation~$\varphi_1$ of coordinates such that $\varphi_1(x_1)=x_2$ and define $z:=\varphi_1(y_1)$.
Since $d(x_2,y_2)=d(x_2,z)=:d$, both $y_2$ and~$z$ share with~$x_2$ exactly $(2n-1-d)/2$ coordinates set to~1.
Moreover, both $y_2$ and~$z$ have exactly $(d+1)/2$ coordinates set to~1 where $x_2$ has a~0.
Clearly, there is a permutation~$\varphi_2$ that maps these two sets of coordinates of~$z$ to the corresponding sets of~$y_2$ and preserves the other coordinates.
As $\varphi_2(z)=y_2$ and $\varphi_2(x_2)=x_2$, the desired automorphism is the composition of~$\varphi_1$ and~$\varphi_2$.
\end{proof}

\subsubsection{Cycle factor construction}
\label{sec:factor}

We now recall the construction of a cycle factor~$\cC_n$ in the graph~$G_n$ described in~\cite{DBLP:conf/soda/MerinoMM21}.
For any bitstring~$x$ and any integer~$\ell$, we write $\sigma^\ell(x)$ for the bitstring obtained from~$x$ by cyclic left rotation by~$\ell$ steps.
If $\ell$ is negative, then this is a right rotation by~$|\ell|$ steps.
We write~$D_n$ for the set of all \emph{Dyck words}, i.e., binary strings of length~$2n-2$ with $n-1$ many 0s and 1s, such that in every prefix, there are at least as many 1s as~0s.
We also define~$D:=\bigcup_{n\geq 1} D_n$.
Any Dyck word~$x$ can be decomposed uniquely as~$x=1\,u\,0\,v$ for $u,v\in D$, and as~$x=u\,1\,v\,0$ for $u,v\in D$.
Moreover, any Dyck word~$x=1\,u\,0\,v\in D_n$ can be interpreted as an (ordered) rooted tree~$T$ with $n$ vertices as follows; see Figure~\ref{fig:trees}~(a):
The leftmost child of the root of~$T$ leads to the subtree defined by~$u$, and the remaining children of the root of~$T$ form the subtree defined by~$v$.
A \emph{tree rotation} operation transforms the tree~$x=1\,u\,0\,v$ into $\rho(x):=u\,1\,v\,0$; see Figure~\ref{fig:trees}~(a).
For any rooted tree~$x\in D$ we write $[x]:=\{\rho^i(x)\mid i\geq 0\}$ for the set of all trees obtained from~$x$ by tree rotation.
The set~$[x]$ can be interpreted as the underlying plane tree obtained from~$x$ by `forgetting' the root.
We write $T_n$ for the set of all plane trees with $n$ vertices.

\begin{figure}
\includegraphics[page=1]{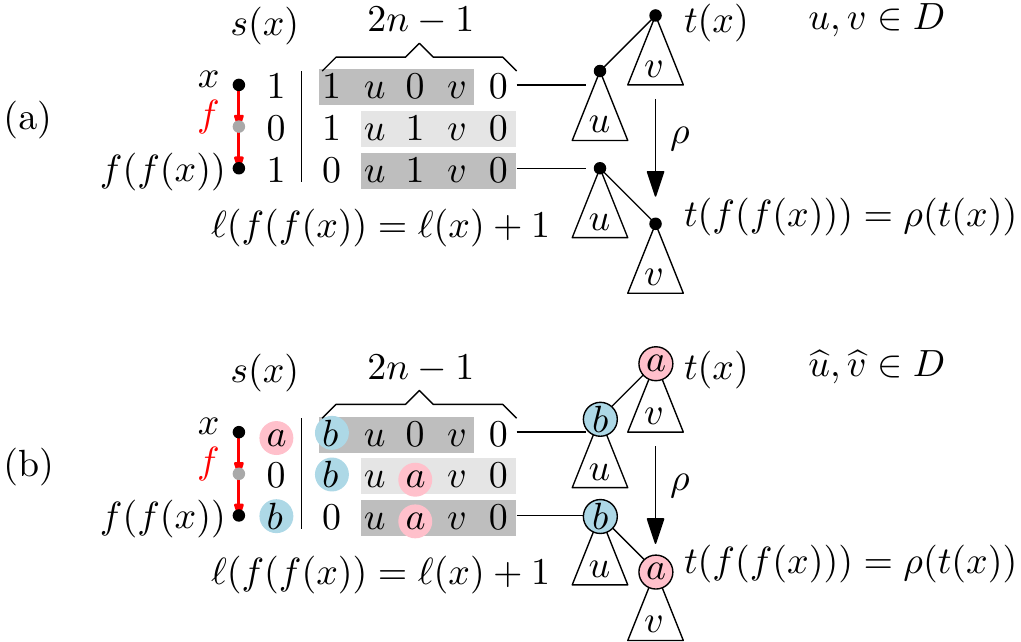}
\caption{(a) Correspondence between cycles in~$\cC_n$ and tree rotation.
The gray boxes highlight the Dyck substrings.
Part~(b) of the figure is explained in Section~\ref{sec:factor-labeled}.
}
\label{fig:trees}
\end{figure}

For any bitstring~$x\in B_n$ there is a unique integer $\ell:=\ell(x)$, $0\leq \ell\leq 2n-2$, such that the first $2n-2$ bits of~$\sigma^\ell(x)$ are a Dyck word.
Similarly, for any bitstring~$y\in B_n'$ there is a unique integer $\ell:=\ell(y)$, $0\leq \ell\leq 2n-2$, such that the last $2n-2$ bits of~$\sigma^\ell(y)$ are a Dyck word.
We refer to~$\ell(x)$ as the \emph{shift} of~$x$ or~$y$, respectively.
Moreover, we write $t(x)\in D_n$ for the first $2n-2$ bits of~$\sigma^\ell(x)$, and $t(y)$ for the last $2n-2$ bits of~$\sigma^\ell(y)$, and we interpret them as a rooted trees with $n$ vertices.

For any vertex $x\in B_n$ we define vertices $f(x),g(x)\in B_n'$ by considering the unique decomposition~$\sigma^\ell(x)=t(x)\,0=1\,u\,0\,v\,0$, where $0\leq \ell\leq 2n-2$ and $u,v\in D$, and by setting
\begin{subequations}
\label{eq:fg}
\begin{equation}
\label{eq:fgx}
f(x):=\sigma^{-\ell}(1\,u\,1\,v\,0)\in B_n',\quad g(x):=\sigma^{-\ell}(1\,u\,0\,v\,1)\in B_n'.
\end{equation}
Clearly, $f(x)$ and~$g(x)$ are two distinct neighbors of~$x$ in the graph~$G_n$.
For any vertex $y\in B_n'$ we define vertices $f(y),g(y)\in B_n$ by considering the unique decomposition~$\sigma^\ell(y)=1\,t(y)=1\,u\,1\,v\,0$, where $0\leq \ell\leq 2n-2$ and $u,v\in D$, and by setting
\begin{equation}
\label{eq:fgy}
f(y):=\sigma^{-\ell}(0\,u\,1\,v\,0)\in B_n,\quad g(y):=\sigma^{-\ell}(1\,u\,0\,v\,0)\in B_n.
\end{equation}
\end{subequations}
Clearly, $f(y)$ and~$g(y)$ are two distinct neighbors of~$x$ in the graph~$G_n$.
The definition of these mappings is illustrated in Figures~\ref{fig:trees}~(a) and~\ref{fig:cycles}~(a).
For any $x\in B_n\cup B_n'$ we define~$C(x):=(x,f(x),f^2(x),\ldots)$.

The following lemma was shown in~\cite{DBLP:conf/soda/MerinoMM21}.

\begin{lemma}[\cite{DBLP:conf/soda/MerinoMM21}]
\label{lem:factor}
Let $n\geq 1$.
\begin{enumerate}[label=(\roman*),leftmargin=8mm, topsep=0mm, noitemsep]
\item For any $x\in B_n\cup B_n'$ we have $g(f(x))=x$ and~$f(g(x))=x$, i.e., $f$ and~$g$ are inverse mappings.
Consequently, $M:=\{(x,f(x))\mid x\in B_n\}$ and $N:=\{(g(x),x)\mid x\in B_n\}$ are two disjoint perfect matchings in the graph~$G_n$ and $\cC_n:=M\cup N=\{C(x)\mid x\in B_n\cup B_n'\}$ is a cycle factor of~$G_n$.
\item For any $x\in B_n$ we have $t(f(f(x)))=\rho(t(x))$ and $\ell(f(f(x)))=\ell(x)+1 \bmod (2n-1)$.
In words, the next vertex from~$B_n$ that follows after~$x$ on the same cycle of~$\cC_n$ is obtained by rotating the tree~$t(x)$ and incrementing the shift.
\item For any $x\in B_n$ we have $C(x)\cap B_n=\{\sigma^i(x' 0) \mid 0\leq i \leq 2n-2\wedge x'\in [t(x)]\}$.
Consequently, the cycles of~$\cC_n$ are in bijection to the set~$T_n$ of plane trees with $n$ vertices.
\end{enumerate}
\end{lemma}

Properties~(ii) and~(iii) are illustrated in Figures~\ref{fig:trees}~(a) and~\ref{fig:cycles}~(a).
One can check that the perfect matchings~$M$ and~$N$ defined in part~(i) of the lemma are in fact the $(n-2)$-lexical and $(n-1)$-lexical matchings defined in~\cite{MR962224} and used in~\cite{gregor-muetze-nummenpalo:18}.
We will think of the cycles of~$\cC_n$ as being oriented in the direction of applying~$f$, oppositely to~$g=f^{-1}$.

The following property of the cycle factor~$\cC_n$ was established in~\cite{gregor-muetze-nummenpalo:18}.

\begin{lemma}[\cite{gregor-muetze-nummenpalo:18}]
\label{lem:auto-Cn}
The mapping $h:B_n\cup B_n'\rightarrow B_n\cup B_n'$ defined by $(x_1,\ldots,x_{2n-1})\mapsto (\ol{x_{2n-2}},\ol{x_{2n-3}},\ldots,\ol{x_1},\ol{x_{2n-1}})$, where overline denotes complementation, is an automorphism of~$G_n$ that maps~$\cC_n$ onto itself.
\end{lemma}

\subsubsection{Gluing cycles}
\label{sec:gluing}

We now describe how to join the cycles of the factor~$\cC_n$ to a single Hamilton cycle, by taking the symmetric difference with suitably chosen 6-cycles.

We consider two Dyck words $x,y\in D_n$ of the form
\begin{equation}
\label{eq:gluing-01}
x=1\,1\,0\,u\,0\,v, \quad y=1\,0\,1\,u\,0\,v, \quad \text{ with } u,v\in D.
\end{equation}

We refer to such a pair~$(x,y)$ as a \emph{gluing pair}, and we write $\cG_n$ for the set of all gluing pairs.
The rooted trees corresponding to~$x$ and~$y$ are shown in Figure~\ref{fig:pull}.
Note that the tree~$y$ is obtained from~$x$ by removing the leftmost edge that leads from the leftmost child of the root of~$x$ to a leaf, and reattaching this edge as the leftmost child of the root.

For any gluing pair~$(x,y)\in\cG_n$ we define $x^i:=f^i(x\,0)$ and $y^i:=f^i(y\,0)$ for $i\geq 0$.
Using the definition~\eqref{eq:fg}, a straightforward calculation yields the vertices shown in Figure~\ref{fig:xyi}.

\begin{figure}
\begin{minipage}{0.46\textwidth}
\centering
\includegraphics{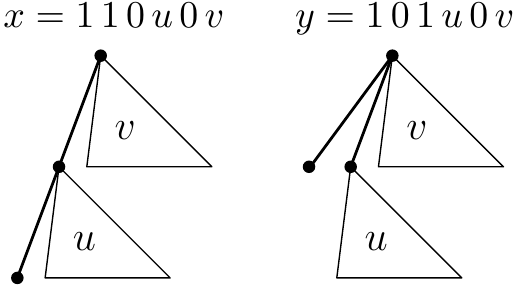}
\captionof{figure}{A gluing pair~$(x,y)$ with corresponding rooted trees.}
\label{fig:pull}
\end{minipage}%
\begin{minipage}{0.54\textwidth}
\centering
\includegraphics{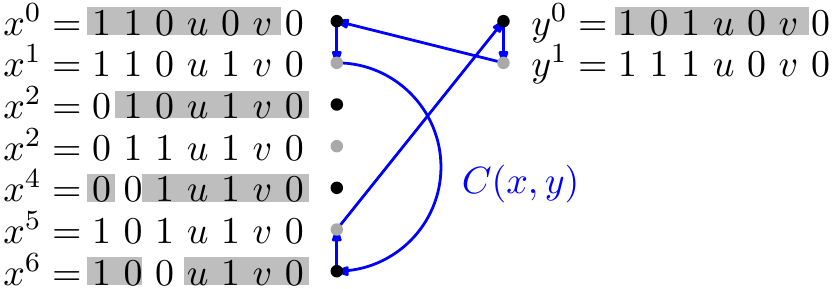}
\captionof{figure}{The vertices $x^i$, $y^i$ and the cycle $C(x,y)$.}
\label{fig:xyi}
\end{minipage}
\end{figure}

Note that
\begin{equation}
\label{eq:C6}
C(x,y):=(x^0,x^1,x^6,x^5,y^0,y^1)
\end{equation}
is a 6-cycle in the graph~$G_n$; see Figure~\ref{fig:xyi}.
This 6-cycle has the edges $(x^0,x^1)$ and $(x^6,x^5)$ in common with the cycle~$C(x^0)$, and the edge~$(y^0,y^1)$ in common with the cycle~$C(y^0)$.
One can check that the symmetric difference of the edge sets~$(C(x^0)\cup C(y^0))\bigtriangleup C(x,y)$ is a single cycle on the vertex set of~$C(x^0)\cup C(y^0)$.
In other words, the cycle $C(x,y)$ glues together two cycles from the cycle factor to a single cycle.
For any set of gluing pairs~$\cF\seq \cG_n$ we write $C(\cF):=\{C(x,y)\mid (x,y)\in\cF\}$.

The following results about the cycle factor~$\cC_n$ and the set of gluing pairs~$\cG_n$ were proved in~\cite{gregor-muetze-nummenpalo:18}; see this paper for illustrations.

\begin{lemma}[\cite{gregor-muetze-nummenpalo:18}]
\label{lem:6cycles}
For any two gluing pairs~$(x,y),(x',y')\in\cG_n$, the 6-cyles $C(x,y)$ and~$C(x',y')$ are edge-disjoint.
For any two gluing pairs~$(x,y),(x,y')\in\cG_n$, the two pairs of edges that the two 6-cycles $C(x,y)$ and~$C(x,y')$ have in common with the cycle~$C(x^0)$ are not interleaved.
\end{lemma}

\begin{lemma}[\cite{gregor-muetze-nummenpalo:18}]
\label{lem:sptree}
For any $n\geq 4$, there is a set~$\cT_n\seq\cG_n$ of gluing pairs of cardinality $|\cT_n|=|\cC_n|-1=|T_n|-1$, such that $\{([x],[y])\mid (x,y)\in\cT_n\}$ is a spanning tree on the set of plane trees~$T_n$, implying that the symmetric difference~$\cC_n\bigtriangleup C(\cT_n)$ is a Hamilton cycle in~$G_n$.
\end{lemma}

\subsubsection{Alternating path}

In the following, we describe how to join two vertices $x\in B_n$, $y\in B_n'$ with Hamming distance $d(x,y)=d$ in the graph~$G_n$ via some cycles from~$\cC_n$ to a short path between~$x$ and~$y$.

\begin{figure}
\includegraphics[page=2]{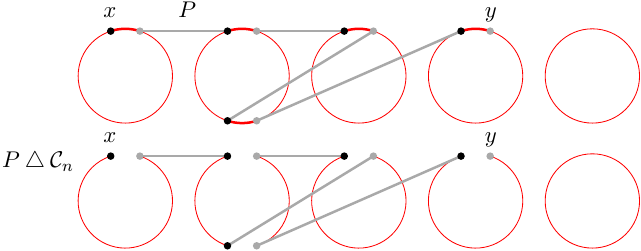}
\caption{Example of a $\cC_n$-alternating path between~$x$ and~$y$.}
\label{fig:alternating}
\end{figure}

We say that a path~$P$ between two vertices~$x$ and~$y$ of~$G_n$ is \emph{$\cC_n$-alternating} if it alternately uses edges and non-edges of~$\cC_n$, starting and ending with an edge of~$\cC_n$, and the symmetric difference $P\bigtriangleup\cC_n$ yields a path between~$x$ and~$y$ that contains all vertices of the intersected cycles; see Figure~\ref{fig:alternating}.
Note that the resulting path will then also contain all vertices of~$P$.
In other words, the path~$P$ glues together all intersected cycles to a single path between~$x$ and~$y$, while the non-intersected cycles are unchanged.
Note that $P$ is allowed to intersect a single cycle of $\cC_n$ multiple times.
The $\cC_n$-alternating paths we use in our proof are obtained as subpaths of one fixed $\cC_n$-alternating path between two complementary vertices with maximum possible Hamming distance~$2n-1$.

\begin{lemma}
\label{lem:alt-path}
For $n\geq 5$, let $P$ be the path defined in Figure~\ref{fig:alt-path} between the two vertices~$x_1\in B_n$ and $y_n\in B_n'$.
For every $i=1,\ldots,n$, the subpath of~$P$ between~$x_1\in B_n$ and $y_i\in B_n'$ is a $\cC_n$-alternating path and $d(x_1,y_i)=2i-1$.
\end{lemma}

\begin{figure}[h!]
\includegraphics[page=3]{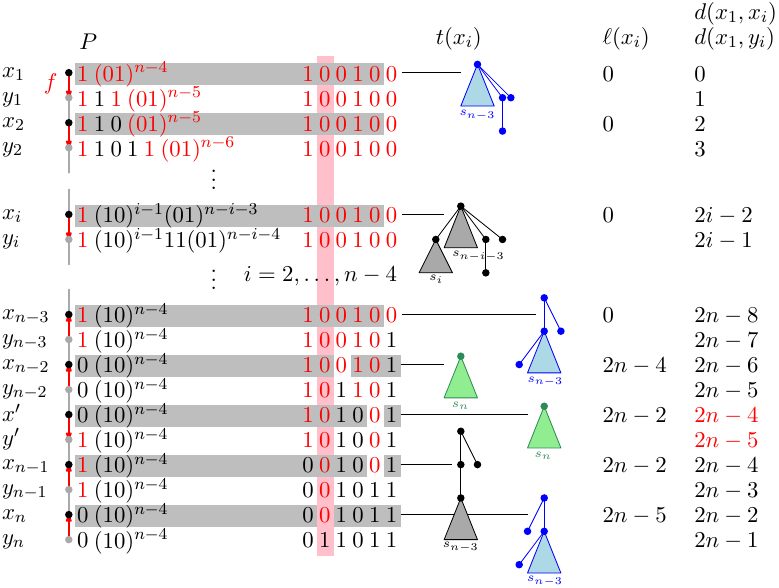}
\caption{Definition of the $\cC_n$-alternating path~$P$ in Lemma~\ref{lem:alt-path}.
The gray boxes highlight the Dyck substrings of the path vertices from~$B_n$.
The corresponding rooted trees are shown to the right, as well as the shifts and the Hamming distance from the first vertex~$x_1$.
The subtrees~$s_i$ are stars with $i$ vertices rooted at the center vertex.
}
\label{fig:alt-path}
\end{figure}

\begin{proof}
By counting the positions in which~$x_i$ and~$y_i$ differ from~$x_1$, it is straightforward to check that $d(x_1,x_i)=2i-2$ and $d(x_1,y_i)=2i-1$ for all $i=1,\ldots,n$.
Note that the Hamming distance from~$x_1$ increases monotonically along~$P$, with the only exception being the step being from~$x'$ to~$y'$, which satisfy $d(x_1,x')=2n-4$ and $d(x_1,y')=2n-5$.
Next, we observe that every edge $(x_i,y_i)$ for $i=1,\ldots,n$ and the edge~$(x',y')$  belong to a cycle from~$\cC_n$.
Specifically, by~\eqref{eq:fgx} we have $y_i=f(x_i)$ for all $i=1,\ldots,n-4$ and $y_i=g(x_i)=f^{-1}(x_i)$ for $i=n-3,\ldots,n$, and $y'=f(x')$.
Furthermore, all other edges of~$P$ are not on a cycle from~$\cC_n$, which can be verified by~\eqref{eq:fgx}.

\begin{figure}[b!]
\includegraphics[page=4]{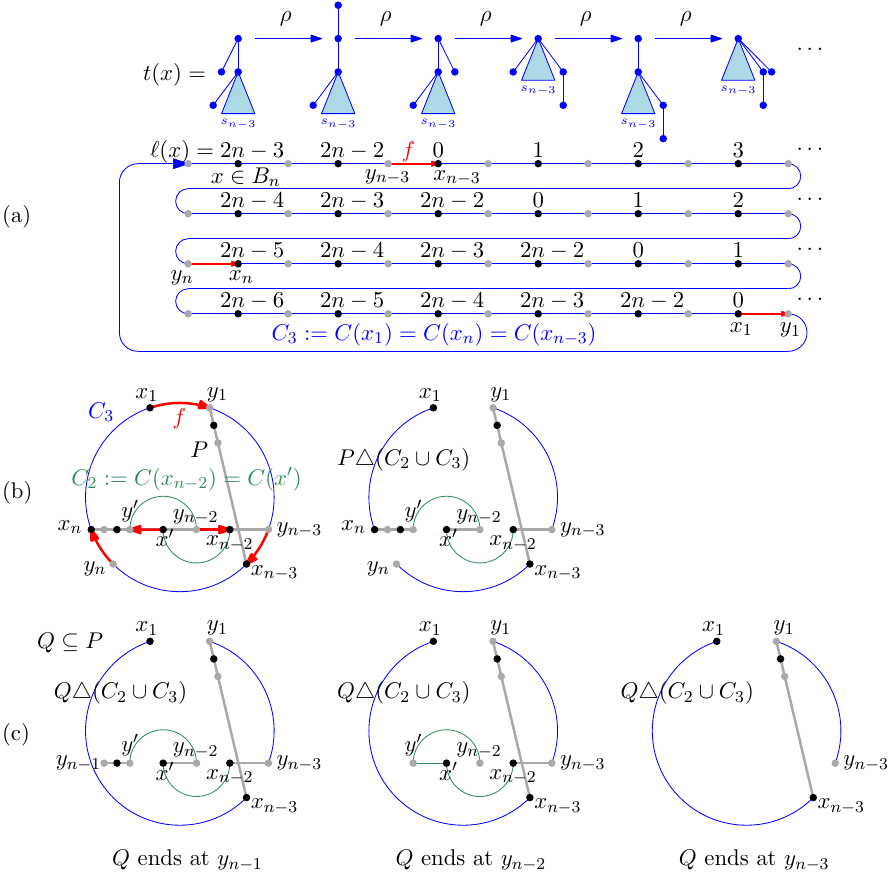}
\caption{The intersection of the path~$P$ and subpaths~$Q\subseteq P$ with the cycles~$C_2$ and~$C_3$.}
\label{fig:intersection}
\end{figure}

We now determine which cycles of~$\cC_n$ the path~$P$ intersects.
For this we consider the rooted trees~$t(x_i)$ for $i=1,\ldots,n$ and~$t(x')$; see Figure~\ref{fig:alt-path}.
For any $i\geq 1$ we define the rooted tree $s_i:=(10)^{i-1}\in D_i$, which is the star with $i$~vertices rooted at the center vertex.
With this abbreviation we obtain
\begin{equation}
\label{eq:txi}
\begin{alignedat}{3}
t(x_1)&=s_{n-3}\,110010,
\quad && t(x_{n-3})=110\,s_{n-3}\,010,
\quad && t(x_n)=10110\,s_{n-3}\,0, \\
t(x_i)&=1\,s_i\,0\,s_{n-i-3}\,110010 \quad && \text{for } i=2,\ldots,n-4, && t(x_{n-1})=1100\,s_{n-3}\,10, \\
t(x_{n-2})&=t(x')=s_n. && &&
\end{alignedat}
\end{equation}
Note that $[t(x_1)]=[t(x_{n-3})]=[t(x_n)]$ is the same underlying plane tree, so the vertices~$x_1$, $x_{n-3}$ and $x_n$ lie on the same cycle of~$\cC_n$ by Lemma~\ref{lem:factor}~(iii).
Similarly, $t(x_{n-2})=t(x')=s_n$, so the vertices~$x_{n-2}$ and~$x'$ lie on the same cycle of~$\cC_n$.
From~\eqref{eq:txi} we see that the tree~$t(x_1)$ has diameter~3 and~$t(x_{n-2})$ has diameter~2, so these two cycles are distinct.
We claim that the remaining vertices~$x_i$ for $i=2,\ldots,n-4$ and~$x_{n-1}$ all lie on their own cycle of~$\cC_n$ that is also distinct from the previous two cycles.
To see this, we use again Lemma~\ref{lem:factor}~(iii), noting that $t(x_i)$ and~$t(x_{n-1})$ have diameter~4, $[t(x_i)]\neq [t(x_j)]$ for any two distinct indices $i,j\in\{2,\ldots,n-4\}$, and moreover $t(x_{n-1})$ has two adjacent vertices of degree~2 unlike~$t(x_i)$ for $i=2,\ldots,n-4$.

Consequently, the entire path~$P$ intersects the cycle $C_3:=C(x_1)=C(x_{n-3})=C(x_n)$ three times, the cycle~$C_2:=C(x_{n-2})=C(x')$ twice, and every other (intersected) cycle only once.

We now check that the symmetric differences~$P\bigtriangleup C_2$, $P\bigtriangleup C_3$, and $P\bigtriangleup(C_2\cup C_3)$ yield a single path containing all vertices of~$C_2$, $C_3$, or $C_2\cup C_3$, respectively.
For $P\bigtriangleup C_2$ this can be seen easily, by noting that the first edge~$(x_{n-2},y_{n-2})$ of~$P$ intersecting~$C_2$ is a backward edge of~$C_2$ (i.e., $y_{n-2}=g(x_{n-2})=f^{-1}(x_{n-2})$), whereas the second edge~$(x',y')$ of~$P$ intersecting~$C_2$ is a forward edge of~$C_2$ (i.e., $y'=f(x')$).
The intersection pattern between~$P$ and~$C_3$ is shown in Figure~\ref{fig:intersection}~(a).
This figure shows the cycle~$C_3$ with all relevant vertices~$x\in B_n$ that lie on this cycle, the corresponding rooted trees~$t(x)$ and the corresponding shifts~$\ell(x)$, as given by Lemma~\ref{lem:factor}~(ii).
The shifts~$\ell(x_1)=0$, $\ell(x_{n-3})=0$ and~$\ell(x_n)=2n-5$ are given by the definition of~$P$, and result in the shown location of the three edges of~$P$ on~$C_3$.
One can check that $P\bigtriangleup C_3$ is indeed a single path that contains all vertices of~$C_3$.
Moreover, the two intersections of~$P$ with~$C_2$ (in the edges $(x_{n-2},y_{n-2})$ and~$(x',y')$) are not separated by intersections with~$C_3$, implying that~$P\bigtriangleup(C_2\cup C_3)$ is also a single path containing all vertices of~$C_2\cup C_3$; see Figure~\ref{fig:intersection}~(b).

So far we have shown that the path~$P$ is $\cC_n$-alternating.
However, the statement of the lemma is stronger, and asserts that every subpath~$Q\seq P$ between~$x_1$ and~$y_i$ for $i=1,\ldots,n$ is also~$\cC_n$-alternating.
The only non-trivial cases to consider are when~$Q$ intersects~$C_2$ or~$C_3$ twice, which happens precisely when~$Q$ ends at~$y_{n-1}$, $y_{n-2}$ or~$y_{n-3}$, and those can be checked to work in Figure~\ref{fig:intersection}~(c).

This completes the proof of the lemma.
\end{proof}

\subsubsection{Proof of Theorem~\ref{thm:M}}

We are now in position to present the proof of Theorem~\ref{thm:M}.

\begin{proof}[Proof of Theorem~\ref{thm:M}]
We need to show that~$G_n=G_n(n)\simeq G(n,n)$, $n\geq 3$, is Hamilton-laceable.
The cases $n=3$ and $n=4$ are covered by Table~\ref{tab:base}, so we will assume that $n\geq 5$.
Let $d$ be an odd integer with $1\leq d\leq 2n-2$.
Moreover, let $\cC_n$ be the cycle factor in~$G_n$ defined in Lemma~\ref{lem:factor}~(i), and let $Q$ be the path given by Lemma~\ref{lem:alt-path} between vertices~$x\in B_n$ and~$y\in B_n'$ with $d(x,y)=d$.
Now consider the cyclically shifted path~$Q':=\sigma^{-4}(Q)$, connecting $x':=\sigma^{-4}(x)$ and $y':=\sigma^{-4}(y)$, which also satisfy $d(x',y')=d$.
By Lemma~\ref{lem:auto}, the theorem is proved by exhibiting a Hamilton path from~$x'$ to~$y'$, which we will do in the following.
All except possibly the last vertex of~$Q$ have a 0-bit at position~$2n-5$; see the highlighted column in Figure~\ref{fig:alt-path}.
It follows that all except possibly the last vertex of~$Q'$ have a 0-bit at their last position.
As $Q$ is~$\cC_n$-alternating by Lemma~\ref{lem:alt-path} and the definition of~$f$ in~\eqref{eq:fg} is invariant under cyclic rotation of bitstrings, the path~$Q'$ is also~$\cC_n$-alternating.
Let $r$ denote the number of cycles of~$\cC_n$ intersected by the path~$Q'$.
Then the symmetric difference $Q'\bigtriangleup\cC_n$ is a path between~$x'$ and~$y'$ plus a set of~$|\cC_n|-r$ cycles, and together they visit all vertices of~$G_n$.

Now consider a set~$\cT_n\seq\cG_n$ of gluing pairs with the properties guaranteed by Lemma~\ref{lem:sptree} and the corresponding 6-cycles~$\cS:=C(\cT_n)=\{C(x,y)\mid (x,y)\in\cT_n\}$.
We consider the image~$\cS':=h(\cS)$ of those 6-cycles under the automorphism~$h$ stated in Lemma~\ref{lem:auto-Cn}.
As this automorphism maps~$\cC_n$ onto itself, we have that~$\cC_n\bigtriangleup~\cS'$ is a Hamilton cycle in~$G_n$.
From Figure~\ref{fig:xyi} we see that for every 6-cycle $C(x,y)\in\cS$, the last bit of all vertices $x^0,\ldots,x^6,y^0,y^1$ is~0, and by the definition of~$h$, the last bit of the corresponding vertices in~$h(\cS)=\cS'$ is~1.
Consequently, all edges between these vertices are disjoint from the path~$Q'$.
As $r$ cycles of~$\cC_n$ are already joined by~$Q'$, we can discard $r-1$ of the 6-cycles from the set~$\cS'$ to obtain a set~$\cS''$ of 6-cycles of cardinality~$|\cS''|=|\cS_n'|-(r-1)=|\cC_n|-r$ such that~$(Q'\bigtriangleup\cC_n)\bigtriangleup\cS''$ is a Hamilton path in~$G_n$ between~$x'$ and~$y'$.
This completes the proof of the theorem.
\end{proof}

\subsection{Proof of Theorem~\ref{thm:M'}}
\label{sec:proof-M'}

The basic strategy for proving Theorem~\ref{thm:M'} is very similar to the proof of Theorem~\ref{thm:M}.
We first construct a cycle factor in the graph, and we then join its cycles to a single Hamilton cycle.
The construction of the cycle factor and of the gluing cycles is achieved by carefully generalizing the constructions described in Sections~\ref{sec:factor} and~\ref{sec:gluing}.
These two steps work for any graph~$G_n(\ba)$, and only the last step of constructing a Hamilton cycle is done specifically for the case~$\ba=(n-1,1)$ required for Theorem~\ref{thm:M'}.

\subsubsection{A cycle factor in~$G_n(\ba)$}
\label{sec:factor-labeled}

For any $n\geq 1$ and any integer partition~$\ba$ of~$n$, we now describe how to construct a cycle factor~$\cC_n(\ba)$ in~$G_n(\ba)$.

Given any string~$x$ over the alphabet~$\{0,1,\ldots,k\}$, we write~$\wh{x}$ for the bitstring obtained from~$x$ by replacing all non-zero symbols by~1.
The mappings~$\ell$, $t$ and~$f$ defined in Section~\ref{sec:factor} on bitstrings can be generalized to operate on multiset permutations in the natural way.
Specifically, for any $x\in\Pi_n(\ba)^-$, we define~$\ell(x):=\ell(\wh{x})$.
Moreover, $t(x)$ is the substring of~$x$ of length~$2n-2$ starting at position~$\ell(x)+1$ (modulo the length~$2n-1$).
Similarly, $f(x)\in\Pi_n(\ba)^-$ is obtained by considering the position~$i$ in which~$\wh{x}$ and~$f(\wh{x})$ differ, and by replacing the symbol at position~$i$ in~$x$ by the suppressed symbol~$s(x)$.

For example, for $x=00230\in\Pi_3(1,1,1)^-$ we have $s(x)=1$, $\wh{x}=00110$, $\ell(\wh{x})=2$ and $t(\wh{x})=1100$.
Consequently, we obtain $\ell(x)=2$ and $t(x)=2300$.
Moreover, $f(\wh{x})=10110$ differs in the first bit from~$\wh{x}$, and therefore $f(x)=s(x)0230=10230$.

For any~$x\in\Pi_n(\ba)^-$ we define the cycle~$C(x):=(x,f(x),f^2(x),\ldots)$, and we define the cycle factor~$\cC_n(\ba):=\{C(x)\mid x\in\Pi_n(\ba)^-\}$.

We also define $D_n(\ba):=\{x\mid x\,0\in\Pi_n(\ba)^{-1}\wedge\wh{x}\in D_n\}$.
We can interpret any~$x\in D_n(\ba)$ as a vertex-labeled rooted tree with $n$ vertices, which has precisely $a_i$~vertices labeled~$i$ for all~$i=1,\ldots,k$ (recall that $n=\sum_{i=1}^k a_i$); see Figures~\ref{fig:trees} and~\ref{fig:cycles}.
For this recall the interpretation of the Dyck word~$\wh{x}\in D_n$ as a rooted tree with $n$ vertices, which can be expressed iteratively as follows:
We start by adding a root vertex.
We then read~$\wh{x}$ from left to right, and for every 1-bit encountered we add a new rightmost child below the current vertex and we move to this vertex, and for every 0-bit encountered we move from the current vertex to its parent, without adding any new vertices.
Now the tree corresponding to~$x$ is obtained as follows:
We start by adding a root vertex with label~$s(x\,0)$.
We then read the string~$x$ from left to right, and for every non-zero symbol~$i>0$ encountered we add a new rightmost child with label~$i$ below the current vertex and we move to this vertex, and for every symbol~0 encountered we move from the current vertex to its parent, without adding any new vertices.

Tree rotations also generalize straightforwardly to the labeled setting.
Specifically, for $x=b\,u\,0\,v\in D_n(\ba)$ with $\wh{u},\wh{v}\in D$ and $a,b\in[k]$ such that $a=s(x\,0)$ we define $\rho(x):=u\,a\,v\,0$, which corresponds to rotating the tree together with its vertex labels (note that $s(\rho(x)\,0)=b$).

\begin{figure}
\centerline{
\includegraphics[page=2]{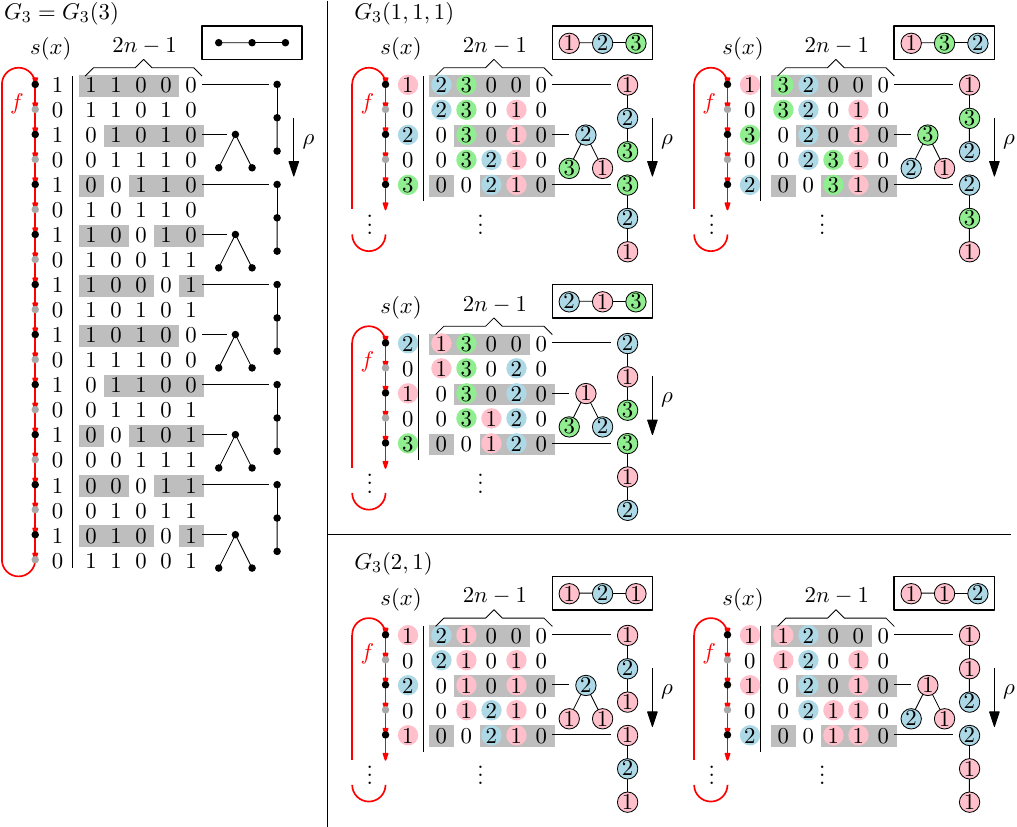}
}
\caption{Construction of the cycle factor~$\cC_n(\ba)$ in the graph~$G_n(\ba)$ (right hand side) from the cycle factor~$\cC_n$ in the graph~$G_n$ (left hand side).
}
\label{fig:cycles}
\end{figure}

With these definitions the properties asserted in Lemma~\ref{lem:factor} about the factor~$\cC_n$ generalize straightforwardly to the labeled version~$\cC_n(\ba)$.
Specifically, if $x=\Pi_n(\ba)^{-1}$, then the tree~$t(f(f(x)))$ is obtained from~$t(x)$ by a labeled tree rotation~$\rho$; see Figure~\ref{fig:trees}~(b) (and the shift is incremented, as before).
Consequently, the cycles of~$\cC_n(\ba)$ are in bijective correspondence with vertex-labeled plane trees with $n$ vertices, exactly $a_i$ of which have the label~$i$ for all $i=1,\ldots,k$.
For example, the cycle factor~$\cC_3(1,1,1)$ shown in Figure~\ref{fig:cycles} has three cycles which correspond to all plane trees with three vertices with vertex labels~1,2,3, namely the path~$P_3$ on three vertices with either 1,2, or 3 as the center vertex.
The cycle factor~$\cC_3(2,1)$ on the other hand, has only two cycles, corresponding to $P_3$ with the label~2 either at the center vertex or at a leaf.
For comparison, the cycle factor~$\cC_3(3)$ has only a single cycle corresponding to~$P_3$ with all vertices labeled~1.

\subsubsection{Labeled gluing cycles}
\label{sec:gluing-labeled}

\begin{figure}[b!]
\centerline{
\includegraphics[page=3]{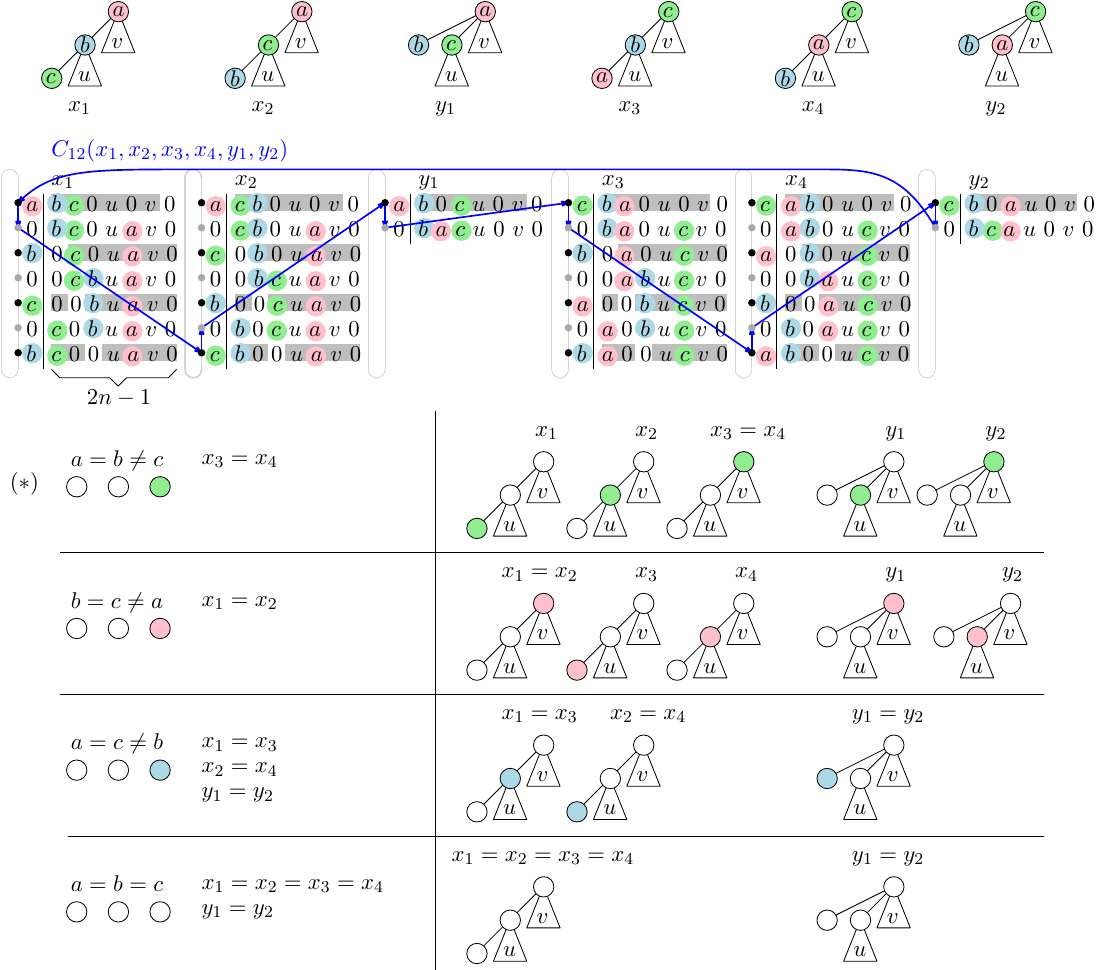}
}
\caption{Definition of the cycle~$C(x_1,x_2,y_1,x_3,x_4,y_2)$.}
\label{fig:gluing}
\end{figure}

We generalize the construction of gluing cycles described in Section~\ref{sec:gluing} to the labeled case.
Specifically, we consider six Dyck words $x_1,x_2,x_3,x_4,y_1,y_2\in D_n(\ba)$ of the form
\begin{equation}
\label{eq:gluing}
\begin{alignedat}{2}
x_1&=b\,c\,0\,u\,0\,v, \quad y_1&&=b\,0\,c\,u\,0\,v, \\
x_2&=c\,b\,0\,u\,0\,v, \quad y_2&&=b\,0\,a\,u\,0\,v, \\
x_3&=b\,a\,0\,u\,0\,v, \quad && \\
x_4&=a\,b\,0\,u\,0\,v, \quad && \text{ with } \wh{u},\wh{v}\in D \text{ and } a,b,c\in [k].
\end{alignedat}
\end{equation}
Note that $s(x_1)=s(x_2)=s(y_1)=a$ and $s(x_3)=s(x_4)=s(y_2)=c$.
We refer to such a 6-tuple $(x_1,x_2,x_3,x_4,y_1,y_2)$ as a \emph{gluing tuple}, and we write~$\cG_n(\ba)$ for the set of all gluing tuples.
Note that by this definition we have $s(x_1)=s(x_2)=s(y_1)=a$ and $s(x_3)=s(x_4)=s(y_2)=c$. 
Also note that $\wh{x_1}=\wh{x_2}=\wh{x_3}=\wh{x_4}=:x$ and $\wh{y_1}=\wh{y_2}=:y$, and that $(x,y)$ is a gluing pair as defined in~\eqref{eq:gluing-01}, i.e., the rooted trees corresponding to $x_1,x_2,x_3,x_4,y_1,y_2$, which are shown in Figure~\ref{fig:gluing}, are vertex-labeled versions of the two trees shown in Figure~\ref{fig:pull}.
Observe also that some of the values~$a,b,c$ may coincide, which leads to different possible coincidences between the~$x_j$ or~$y_j$; see Figure~\ref{fig:gluing}.

For a gluing tuple~$(x_1,x_2,x_3,x_4,y_1,y_2)\in\cG_n(\ba)$ we define $x_j^i:=f^i(x_j\,0)$ for $j=1,2,3,4$ and $i\geq 0$, and $y_j^i:=f^i(y_j\,0)$ for $j=1,2$ and $i\geq 0$.
It can be verified from~Figure~\ref{fig:gluing} that the sequence of vertices
\begin{equation}
C(x_1,x_2,x_3,x_4,y_1,y_2):=(x_1^0,x_1^1,x_2^6,x_2^5,y_1^0,y_1^1,x_3^0,x_3^1,x_4^6,x_4^5,y_2^0,y_2^1)
\end{equation}
is cyclic, and any two consecutive vertices differ in one position.
This sequence is obtained by applying the flip sequence of the 6-cycle in~\eqref{eq:C6} two times.
For any set of gluing tuples~$\cF\seq\cG_n(\ba)$ we write $C(\cF):=\{C(x_1,x_2,x_3,x_4,y_1,y_2)\mid (x_1,x_2,x_3,x_4,y_1,y_2)\in\cF\}$.
Note that $C(x_1,x_2,x_3,x_4,y_1,y_2)$ is a 12-cycle if $a\neq c$ and a 6-cycle if $a=c$.
Moreover, this cycle can be used to join either 2, 3, 5 or 6 different cycles from~$\cC_n(\ba)$, depending on the coincidences between the values~$a,b,c$, as shown in Figure~\ref{fig:gluing}.

By Lemma~\ref{lem:6cycles}, any two distinct cycles from~$C(\cG_n(\ba))$ are edge-disjoint, and if two of them intersect one cycle of~$\cC_n(\ba)$ twice each, then the two pairs of edges on this cycle are not interleaved.
Note that the cycles $C(x_1,x_2,x_3,x_4,y_1,y_2)$ and $C(x_3,x_4,x_1,x_2,y_2,y_1)$ have the exact same set of vertices and edges, so they are the same cycle.

\subsubsection{Proof of Theorem~\ref{thm:M'}}

The following lemma, proved in~\cite{DBLP:conf/soda/MerinoMM21}, strengthens Lemma~\ref{lem:sptree} from before.
To state the lemma, we need a few more definitions.
As before, we write $s_n:=(10)^{n-1}\in D_n$ for the star with $n$ vertices, rooted at the center vertex.
Given a (rooted or plane) tree~$x$, the \emph{potential} of a vertex of~$x$ is the sum of distances from that vertex to all other vertices of~$x$.
Moreover, the \emph{potential} of~$x$, denoted by $\varphi(x)$, is the minimum potential of all vertices of~$x$.
For example, $\varphi(s_n)=n-1$.

\begin{figure}
\centerline{
\includegraphics[page=1]{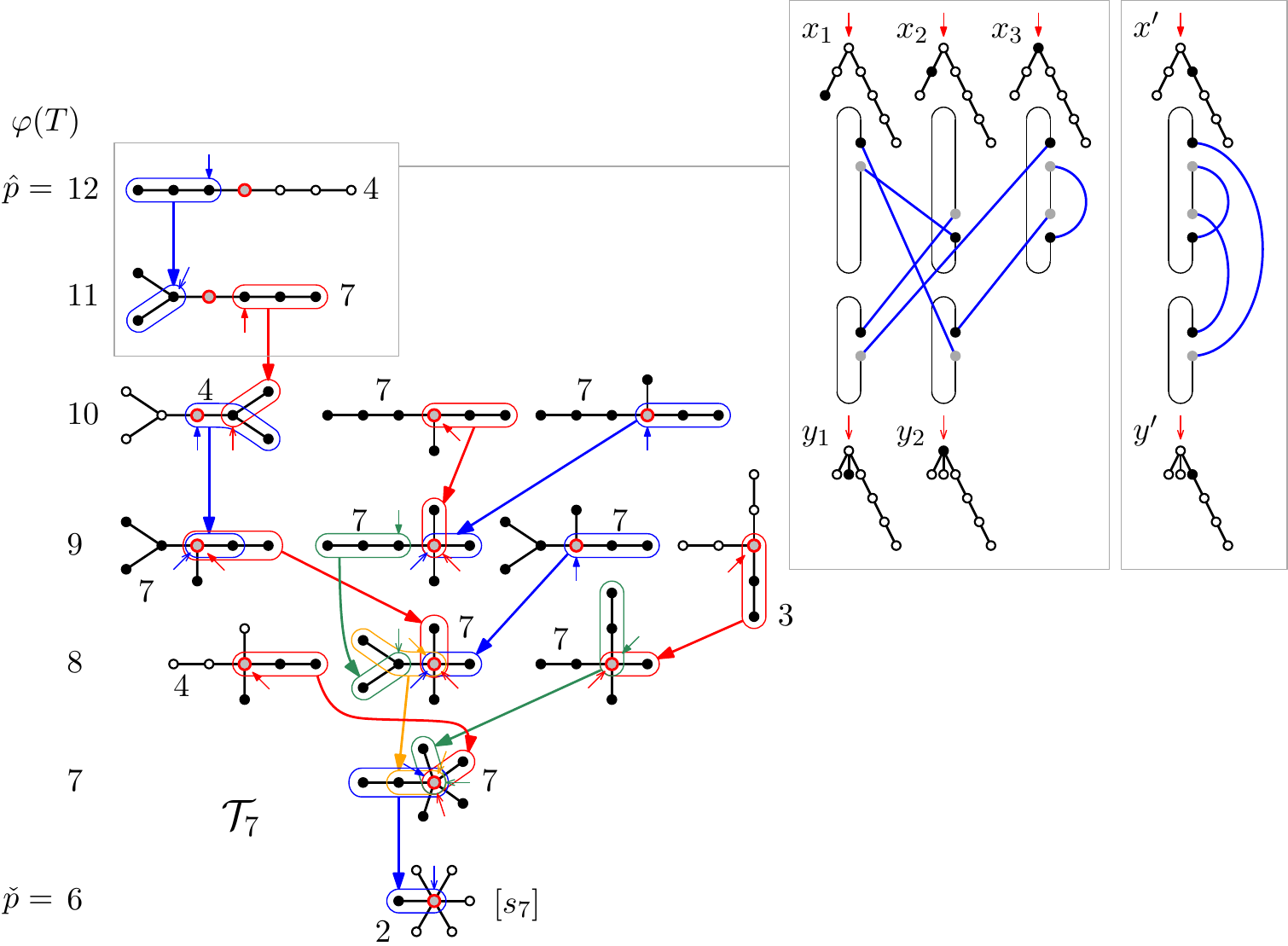}
}
\caption{Illustration of Lemma~\ref{lem:sptree-pot} and the proof of Theorem~\ref{thm:M'}.
The figure shows all plane trees with 7~vertices, arranged in levels according to their potential.
The vertices with minimum potential are highlighted.
The number next to each tree specifies the number of distinct plane trees obtained by marking a single vertex (in each orbit of symmetric vertices, one is drawn black and the others white).
Each gluing pair~$(x,y)\in\cT_n\seq \cG_n$ is visualized by indicating the placement of roots at the plane trees~$[x]$ and~$[y]$ by a small filled or non-filled arrow, respectively.
The bubbles indicate gluing tuples from~$\cT_n'\seq\cG_n(n-1,1)$ that join the 5 marked plane trees with the mark placed at one vertex within the bubbles in each tree.
An example of two gluing tuples that are used for the four marked paths with potential~12 is depicted in the box. 
}
\label{fig:sptree}
\end{figure}

\begin{lemma}[\cite{DBLP:conf/soda/MerinoMM21}]
\label{lem:sptree-pot}
For any $n\geq 4$, there is a set~$\cT_n\seq\cG_n$ of gluing pairs of cardinality $|\cT_n|=|\cC_n|-1=|T_n|-1$, such that $\{([x],[y])\mid (x,y)\in\cT_n\}$ is a spanning tree on the set of plane trees~$T_n$, implying that the symmetric difference~$\cC_n\bigtriangleup C(\cT_n)$ is a Hamilton cycle in~$G_n$.
Moreover, for every gluing pair~$(x,y)\in\cT_n$ we have $\varphi(y)=\varphi(x)-1$, and for every plane tree~$T\in T_n\setminus\{[s_n]\}$ there is exactly one gluing pair~$(x,y)\in\cT_n$ with $T=[x]$ and $\varphi(y)=\varphi(x)-1$.
\end{lemma}

The spanning tree described in this lemma is illustrated in Figure~\ref{fig:sptree} for $n=7$.

We are now ready to present the proof of Theorem~\ref{thm:M'}.

\begin{proof}[Proof of Theorem~\ref{thm:M'}]
The graphs~$G(2,1,1)$ and~$G(3,2,1)$ have a Hamilton cycle by Table~\ref{tab:base}.
In the remainder of the proof we construct a Hamilton cycle in~$G_n(n-1,1)\simeq G(n,n-1,1)$ for $n\geq 4$.
We define $\ba:=(n-1,1)$.

We start with the cycle factor~$\cC_n(\ba)$ constructed in Section~\ref{sec:factor-labeled}.
Its cycles are in bijective correspondence to plane trees that have a single vertex labeled~2, and all other vertices labeled~1.
We say that this label-2 vertex is \emph{marked}, and we refer to these trees as \emph{marked plane trees}.
We construct a set of gluing tuples~$\cT_n'\seq\cG_n(\ba)$ such that~$\cC_n(\ba)\bigtriangleup C(\cT_n')$ is a Hamilton cycle in~$G_n(\ba)$.

We let $\cT_n\seq\cG_n$ be the set of gluing pairs given by Lemma~\ref{lem:sptree-pot}.
The set~$\cT_n'\seq\cG_n(\ba)$ is constructed from~$\cT_n$ by considering the potential of the trees of these gluing pairs.
Note that the path on $n$~vertices has the highest potential~$\hat{p}:=\lfloor n^2/4\rfloor$, whereas the star~$[s_n]$ has the lowest potential~$\check{p}:=n-1$.
Roughly speaking, we consider all plane trees with decreasing potential $p=\hat{p},\hat{p}-1,\ldots,\check{p}$, and we add gluing tuples from~$\cG_n(\ba)$ derived from gluing pairs of~$\cT_n$ step by step, such that all cycles of~$\cC_n(\ba)$ corresponding to marked plane trees with potential at least~$p$ are joined together.

Formally, for $p=\hat{p},\hat{p}-1,\ldots,\check{p}$, we inductively construct a set~$\cT_{n,p}'\seq\cG_n(\ba)$ of gluing tuples such that~$\cC_n(\ba)\bigtriangleup C(\cT_{n,p}')$ is a cycle factor in~$G_n(\ba)$ with the property that every vertex is contained in a cycle together with all vertices~$\{\tilde{x}\,0\mid \tilde{x}\in[x]\}$ for some $x\in D_n(\ba)$ with $\varphi(x)\leq p$.
Note that there are precisely two marked plane trees with minimal potential~$p=\check{p}=n-1$, namely the star~$[s_n]$ with either the center vertex or a leaf marked, i.e., these are $[s_n']$ and~$[s_n'']$ where $s_n':=(10)^{n-1}\in D_n(\ba)$ and $s_n'':=(20)(10)^{n-2}\in D_n(\ba)$.
We will show that $s_n'\,0,s_n''\,0\in\Pi_n(\ba)^{-1}$ lie on the same cycle, and these conditions imply that~$\cC_n(\ba)\bigtriangleup C(\cT_{n,\check{p}}')$ is a Hamilton cycle in~$G_n(\ba)$.

The induction basis is~$\cT_{n,\hat{p}}':=\emptyset$.
For the induction step~$p\rightarrow p-1$, suppose that $\cT_{n,p}'$, $\hat{p}\geq p>\check{p}$, with the properties as stated before is given.
Then~$\cT_{n,p-1}'$ is obtained by adding gluing tuples from~$\cG_n(\ba)$ to the set~$\cT_{n,p}'$ as follows.

We consider each plane tree~$T$ with $n$ vertices and potential~$p$, and we consider the unique gluing pair~$(x,y)\in\cT_n$ with $T=[x]$ and $\varphi(y)=\varphi(x)-1$.
We then consider all rooted trees obtained from~$x$ by marking one vertex, i.e., the set $X':=\{x'\in D_n(\ba)\mid \wh{x'}=x\}$.
Two marked rooted trees $x',x''\in X'$ are equivalent, if $x'\,0$ and $x''\,0$ lie on the same cycle in the factor~$\cC_n(\ba)\bigtriangleup C(\cT_{n,p}')$, and we partition the set~$X'$ into equivalence classes $X'=X_1'\cup \cdots\cup X_r'$ accordingly.
Note that $x',x''$ can be equivalent either because these trees have the same underlying marked plane tree, i.e., $[x']=[x'']$, or because the cycles~$C(x'\,0)$ and~$C(x''\,0)$ have been joined previously.

\begin{figure}[b!]
\centerline{
\includegraphics[page=2]{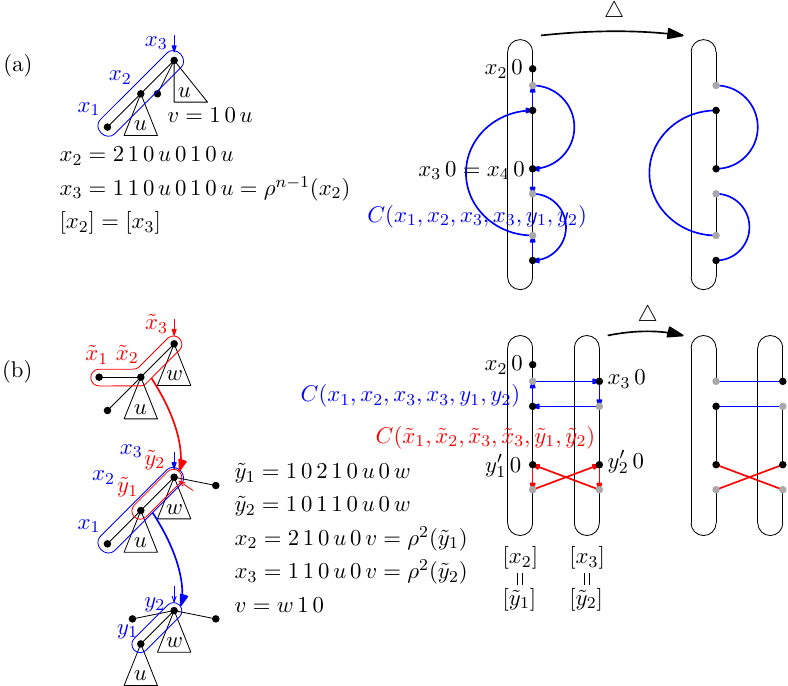}
}
\caption{Verification of multiple intersections in the proof of Theorem~\ref{thm:M'}.
}
\label{fig:check}
\end{figure}

By~\eqref{eq:gluing-01} we have $x=1\,1\,0\,u\,0\,v$ and $y=1\,0\,1\,u\,0\,v$ for some~$u,v\in D$.
We consider the three trees~$x_3,x_2,x_1\in D_n(\ba)$ obtained from~$x$ by marking the root of~$x$, or the leftmost descendants in distance~1 or~2 from the root, respectively, and the two trees~$y_2,y_1\in D_n(\ba)$ obtained from~$y$ by marking the root of~$y$, or its second child from left to right, respectively (these are the five trees marked by~(*) in Figure~\ref{fig:gluing}).
By these definitions we have
\begin{equation*}
\begin{alignedat}{2}
x_1&=1\,2\,0\,u\,0\,v, \quad y_1&&=1\,0\,2\,u\,0\,v, \\
x_2&=2\,1\,0\,u\,0\,v, \quad y_2&&=1\,0\,1\,u\,0\,v, \\
x_3&=1\,1\,0\,u\,0\,v.
\end{alignedat}
\end{equation*}
We then add the gluing tuple~$(x_1,x_2,x_3,x_3,y_1,y_2)$ to the set~$\cT_{n,p-1}'$.
The corresponding 12-cycle~$C(x_1,x_2,x_3,x_3,y_1,y_2)$ intersects five cycles from the factor~$\cC_n(\ba)$, corresponding to the marked plane trees~$[x_1]$, $[x_2]$, $[x_3]$, $[y_1]$, and~$[y_2]$; see Figure~\ref{fig:sptree}.
There is one special case, namely if $v=1\,0\,u$, in which case~$[x_2]=[x_3]$, and then the 12-cycle intersects only four cycles from the factor~$\cC_n(\ba)$, but one can check that the symmetric difference is still a single cycle; see Figure~\ref{fig:check}~(a).

In addition to the aforementioned gluing tuple, which is added in every case, we add further gluing tuples to the set~$\cT_{n,p-1}'$, depending on~$X'$.
Specifically, for every equivalence class~$X_i'$, $i\in[r]$, that contains none of~$x_1,x_2,x_3$, we select one tree~$x'\in X_i'$, and we let $y'\in D_n(\ba)$ be the string obtained from~$x'$ by swapping the second and third entry, i.e., $x'=110u0v$, $y'=101u0v$ with $\wh{u},\wh{v}\in D$ such that the concatenation~$uv$ contains a single occurrence of~2.
We then add the gluing tuple~$(x',x',x',x',y',y')$ to the set~$\cT_{n,p}'$.
The corresponding 6-cycle~$C(x',x',x',x',y',y')$ intersects two cycles from the factor~$\cC_n(\ba)$, corresponding to the marked plane trees~$[x']$ and~$[y']$.

By the construction of~$\cT_{n,p-1}'$, each equivalence class~$X_i'$, $i\in[r]$, contains at most two of the~$x_i$, and if this happens then $X_i'\cap\{x_1,x_2,x_3\}=\{x_2,x_3\}$, and $[x_2]=[\tilde{y}_1]$ and~$[x_3]=[\tilde{y}_2]$ for some gluing pair~$(\tilde{x}_1,\tilde{x}_2,\tilde{x}_3,\tilde{x}_3,\tilde{y}_1,\tilde{y}_2)\in\cT_{n,p-1}'$; see Figure~\ref{fig:check}~(b).
One can check in the figure that the double intersection between the two gluing cycles and the two cycles from the factor still yields a single cycle in the symmetric difference.

By the above steps and the induction hypothesis for $p$, we have achieved that in~$\cC_n(\ba)\bigtriangleup C(\cT_{n,p-1}')$, every vertex is contained in a cycle together with all vertices~$\{\tilde{x}\,0\mid \tilde{x}\in[x]\}$ for some $x\in D_n(\ba)$ with $\varphi(x)\leq p-1$, i.e., the induction hypothesis holds for $p-1$.
In the last step of the construction for $p=\check{p}+1$ we added a gluing tuple $(x_1,x_2,x_3,x_3,y_1,y_2)$ where $y_1=s_n''$, $y_2=s_n'$, and thus we glued both cycles $[s_n'']$ and $[s_n']$ into one.

This completes the proof of the theorem.
\end{proof}

\section{Open questions}
\label{sec:open}

We conclude this paper with the following open questions.

\begin{itemize}[leftmargin=8mm, labelsep=4mm,noitemsep, topsep=3pt plus 3pt]
\item
We believe that the techniques used in Section~\ref{sec:proof-M'} to prove that~$G(\alpha,\alpha-1,1)$ has a Hamilton cycle are in principle suitable to prove that~$G(\ba)$ has a Hamilton cycle for all $\ba$ with~$\Delta(\ba)=0$, which would be an important first step towards a proof of Conjecture~\ref{conj:0}.
In particular, the construction of the cycle factor~$\cC_n(\ba)$ and gluing tuples~$\cG_n(\ba)$ described in Sections~\ref{sec:factor-labeled} and~\ref{sec:gluing-labeled}, respectively, are fully general.
The main difficulty in combining these ingredients lies in the fact that some gluing cycles join more than two cycles from the factor (namely 3, 5, or 6 cycles) to a single cycle, and in this case the resulting interactions between different gluing cycles seem to be hard to control (recall Figure~\ref{fig:check}).

\item
We conjecture that~$G(\alpha,\alpha,1)$ for $\alpha\geq 2$ is Hamilton-connected, just as all other graphs~$G(\ba)$ with~$\Delta(\ba)>0$ covered by Theorem~\ref{thm:P}.
However, we are unable to prove this based on Conjecture~\ref{conj:0}; recall the discussion in Section~\ref{sec:notions}.
We confirmed this conjecture with computer help for $\alpha=2,3,4$; see Table~\ref{tab:base}.
Proving the conjecture would streamline our proof of Theorem~\ref{thm:P} considerably, as it would make Lemmas~\ref{lem:H2'}--\ref{lem:H1''} redundant, which build on top of Lemma~\ref{lem:L12} in the induction proof.
To prove that~$G(\alpha,\alpha,1)$ is Hamilton-connected, it may help to establish a Hamiltonicity property for graphs~$G(\ba)$ with~$\Delta(\ba)=0$ and $k\geq 3$ that is stronger than~$\cL_1$.
Specifically, in addition to a Hamilton path between any vertex in~$\Pi(\ba)^{1,1}$ and any vertex not in~$\Pi(\ba)^{1,1}$, we may also ask for a Hamilton path between any two distinct vertices in~$\Pi(\ba)^{1,1}$.
We checked by computer whether~$G(\ba)$ has this stronger property for $\ba\in\{(2,1,1),(3,2,1),(3,1,1,1),(4,3,1),(4,2,2),(4,2,1,1),(4,1,1,1,1)\}$, and it was satisfied in all cases except for~$\ba=(2,1,1)$.

\item
While the proofs presented in this paper are constructive, they are far from yielding efficient algorithms for computing the corresponding Gray codes.
Ideally, one would like algorithms whose running time is polynomial in~$n$ per generated multiset permutation of length~$n$.
Such algorithms are known for the Hamilton cycles mentioned in Theorems~\ref{thm:ehrlich} and~\ref{thm:mlc}, see \cite[Section~7.2.1.2]{MR3444818} and~\cite{MR4075363}, respectively.
An interesting direction could be to explore greedy algorithms for generating multiset permutations by star transpositions, which may yield much simpler constructions to start with, cf.~\cite{DBLP:conf/wads/Williams13,MR3513761,cameron-sawada-williams:21}.

\item
Knuth raised the question whether there are star transposition Gray codes for $(\alpha,\alpha)$-combinations whose flip sequence can be partitioned into $2\alpha-1$ blocks, such that each block is obtained from the previous one by adding~$+1$ modulo~$2\alpha-1$.
This problem is a strengthening of the middle levels conjecture, and it was answered affirmatively in~\cite{DBLP:conf/soda/MerinoMM21}.
The Gray code for $(4,4)$-combinations shown in Figure~\ref{fig:mperm}~(a) has such a 7-fold cyclic symmetry.
We can ask more generally: Are there star transposition Gray codes for multiset permutations whose flip sequence can be partitioned into~$n-1$ blocks, such that each block is obtained from the previous one by adding~$+1$ modulo~$n-1$?
Figure~\ref{fig:222_symm} shows an ad-hoc solution for $(2,2,2)$-multiset permutations with 5-fold cyclic symmetry.

\item
A more general version of the problem considered in this paper is the following:
We consider an alphabet $\{1,\ldots,k\}$ of size $k\geq 2$, and frequencies $a_1,\ldots,a_k\geq 1$ that specify that symbol~$i$ appears exactly $a_i$ times for all $i=1,\ldots,k$.
Moreover, there is an additional integer parameter~$s$ with $1\leq s\leq k-1$ that has the following significance.
The objects to be generated are all pairs~$(S,x)$, where $S$ is a set or string of $s$ distinct symbols, and $x$ is a string of the remaining $n-s$ symbols, where $n:=a_1+\cdots+a_k$.
A star transposition swaps one symbol from~$S$ with one symbol from~$x$ that is currently not in~$S$, and the question is whether there is a star transposition Gray code for all those objects.

Note that multiset permutations considered in this paper are the special case when~$s=1$.
Andrea Sportiello suggested this problem with $a_1=\cdots=a_k=\alpha$ (uniform frequency) and $S$ being a set as a generalization of the middle levels conjecture ($s=1$, $k=2$, $a_1=a_2=\alpha$).
Moreover, Ajit A. Diwan suggested this problem with $a_1=\cdots=a_k=\alpha$ (uniform frequency) and a set~$S$ of size $s=k-1$.
Note that the uniform frequency case is particularly interesting, as the underlying flip graph for this problem is vertex-transitive if and only if $a_1=\cdots=a_k$ (recall Lov{\'a}sz' conjecture~\cite{MR0263646}).
\end{itemize}

\begin{figure}
\includegraphics{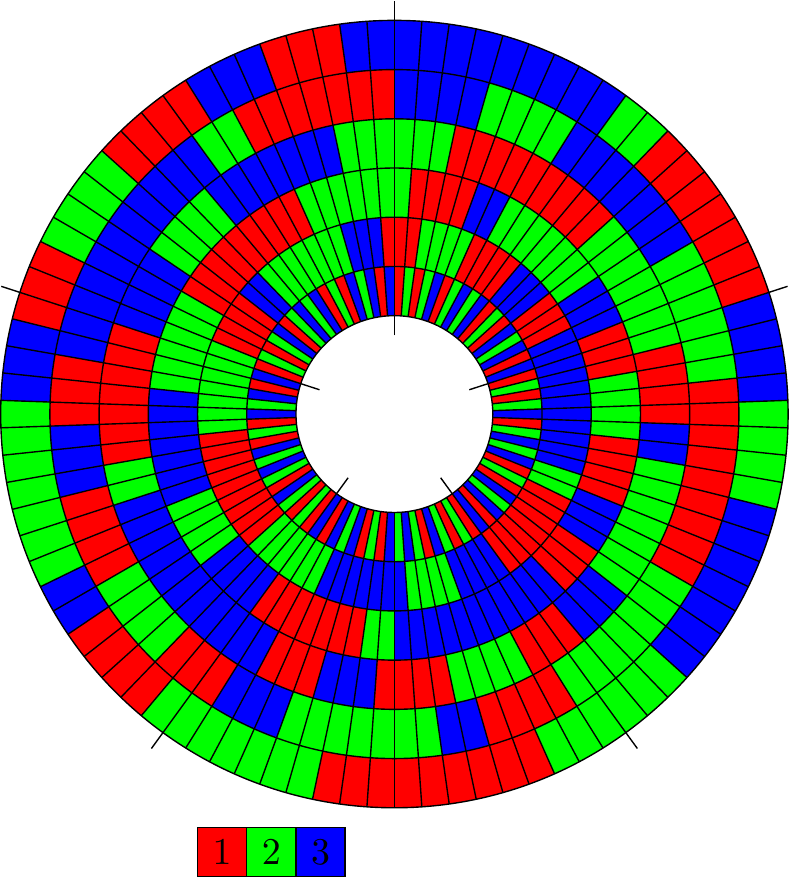}
\caption{Star transposition Gray code for $(2,2,2)$-multiset permutations with 5-fold cyclic symmetry.}
\label{fig:222_symm}
\end{figure}

\section*{Acknowledgements}

We thank Ji\v{r}\'{\i} Fink for inspiring brain-storming sessions in the early phases of this research project, and we thank Aaron Williams for interesting discussions about star transposition Gray codes.

\bibliographystyle{alpha}
\bibliography{../refs}

\appendix

\section{Proofs of Lemmas~\ref{lem:H2'}--\ref{lem:H1''}}

\begin{proof}[Proof of Lemma~\ref{lem:H2'}]
Let $x,y$ be two distinct vertices in~$G(\ba)$.
As $x\neq y$ there is a position $\ihat>1$ such that $x_\ihat\neq y_\ihat$.
As the symbols~1 and~2 both appear with the same frequency~$\alpha$, we may assume w.l.o.g.\ that either $(x_\ihat,y_\ihat)=(1,2)$ or $(x_\ihat,y_\ihat)=(1,3)$.

We first consider the case $(x_\ihat,y_\ihat)=(1,2)$.
As $\alpha\geq 3$, there are two indices $i_1,i_2\in[n]\setminus\{1,\ihat\}$ such that $x_{i_1}\neq 2$ and $x_{i_2}=2$, in particular $x_{i_1}\neq x_{i_2}$.
Similarly, there are two indices $i_3,i_4\in[n]\setminus\{1,\ihat\}$ such that $y_{i_3}\neq 1$ and $y_{i_4}=1$, in particular $y_{i_3}\neq y_{i_4}$.
We choose multiset permutations $u^j,v^j\in \Pi(\ba)$, $j=1,2,3$, satisfying the following conditions; see Figure~\ref{fig:H2'}~(a):
\begin{itemize}[leftmargin=8mm, noitemsep]
\item[(i)] $u^1=x$, $v^1_{i_1}=x_{i_2}=2$, and $v^3=y$, $u^3_{i_3}=y_{i_4}=1$;
\item[(ii)] $u^2_1=v^1_\ihat=1$, $u^2_\ihat=v^1_1=3$, $u^3_1=v^2_\ihat=3$, $u^3_\ihat=v^2_1=2$, and $u^j_i=v^{j-1}_i$ for $j=2,3$ and all $i\in[n]\setminus\{1,\ihat\}$;
\item[(iii)] $(u^2_{i_1},v^2_{i_1})=(2,1)$.
\end{itemize}
Note that $u^1\neq v^1$ and $u^3\neq v^3$ holds by~(i) and the choice of $i_1,i_2$ and $i_3,i_4$, respectively, and $u^2_1\neq v^2_1$ by~(ii).
Consequently, we have $u^j\neq v^j$ for $j=1,2,3$.

We consider a Hamilton path~$P_1$ in~$G^{\ihat,1}(\ba)$ from~$u^1$ to~$v^1$, and a Hamilton path~$P_3$ in~$G^{\ihat,2}(\ba)$ from~$u^3$ to~$v^3$, which exist by the assumption~$G(\alpha,\alpha-1,2)\in\cH$.
We also consider a Hamilton path~$P_2$ in~$G^{\ihat,3}(\ba)$ from~$u^2$ to~$v^2$, which exists by the assumption~$G(\alpha,\alpha,1)\in\cL_{12}$, using that $(u^2_1,v^2_1)=(1,2)$ and $(u^2_{i_1},v^2_{i_1})=(2,1)=(p_3(1,2),q_3(1,2))$ by~(ii)+(iii) and~\eqref{eq:pq}.
The concatenation $P_1P_2P_3$ is a Hamilton path in~$G(\ba)$ from~$x$ to~$y$.

\begin{figure}
\includegraphics[page=5]{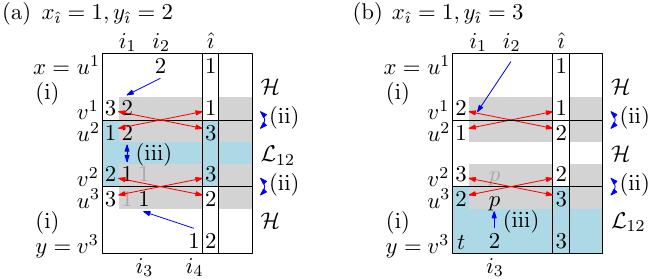}
\caption{Illustration of the proof of Lemma~\ref{lem:H2'}.}
\label{fig:H2'}
\end{figure}

It remains to consider the case~$(x_\ihat,y_\ihat)=(1,3)$.
We define $t:=y_1\in\{1,2,3\}$ and $p:=p_3(2,t)\in\{1,3\}$.
As $\alpha\geq 3$, there are two indices $i_1,i_2\in[n]\setminus\{1,\ihat\}$ such that $x_{i_1}\neq x_{i_2}$.
Similarly, there is an index $i_3\in[n]\setminus\{1,\ihat\}$ such that $y_{i_3}=2$.
We choose multiset permutations $u^j,v^j\in \Pi(\ba)$, $j=1,2,3$, satisfying the following conditions; see Figure~\ref{fig:H2'}~(b):
\begin{itemize}[leftmargin=8mm, noitemsep]
\item[(i)] $u^1=x$, $v^1_{i_1}=x_{i_2}$, and $v^3=y$;
\item[(ii)] $u^2_1=v^1_\ihat=1$, $u^2_\ihat=v^1_1=2$, $u^3_1=v^2_\ihat=2$, $u^3_\ihat=v^2_1=3$, and $u^j_i=v^{j-1}_i$ for $j=2,3$ and all $i\in[n]\setminus\{1,\ihat\}$;
\item[(iii)] $u^3_{i_3}=p$.
\end{itemize}
Note that $u^1\neq v^1$ by~(i) and the choice of $i_1,i_2$, $u^2_1\neq v^2_1$ by~(ii), and $u^3\neq v^3$ by~(iii), the choice of~$i_3$ and the fact that $p\neq 2$.
Consequently, we have $u^j\neq v^j$ for $j=1,2,3$.

We consider a Hamilton path~$P_1$ in~$G^{\ihat,1}(a)$ from~$u^1$ to~$v^1$, and a Hamilton path~$P_2$ in~$G^{\ihat,2}(\ba)$ from~$u^2$ to~$v^2$, which exist by the assumption~$G(\alpha,\alpha-1,2)\in\cH$.
We also consider a Hamilton path~$P_3$ in~$G^{\ihat,3}(\ba)$ from~$u^3$ to~$v^3$, which exists by the assumption~$G(\alpha,\alpha,1)\in\cL_{12}$, using that $(u^3_1,v^3_1)=(2,t)$ and $(u^3_{i_3},v^3_{i_3})=(p,2)=(p_3(2,t),q_3(2,t))$ by~(i)--(iii) and~\eqref{eq:pq}.
The concatenation $P_1P_2P_3$ is a Hamilton path in~$G(\ba)$ from~$x$ to~$y$.
\end{proof}

\begin{proof}[Proof of Lemma~\ref{lem:H2''}]
Let $x,y$ be two distinct vertices in~$G(\ba)$.
As $x\neq y$ there is a position $\ihat>1$ such that $x_\ihat\neq y_\ihat$.
As the symbols~1 and~2 both appear with the same frequency~$\alpha$, and the symbols~3 and~4 both appear with the same frequency~1, we may assume w.l.o.g.\ that either $(x_\ihat,y_\ihat)=(1,2)$, $(x_\ihat,y_\ihat)=(1,3)$, or $(x_\ihat,y_\ihat)=(3,4)$.

We first consider the case $(x_\ihat,y_\ihat)=(1,2)$.
Let $i_2$ be the unique index such that $x_{i_2}=4$, and let $i_4$ be the unique index such that $y_{i_4}=3$.
As $\alpha\geq 3$ and hence $n\geq 2\alpha+2\geq 8$, there is an index~$i_1\in[n]\setminus\{1,\ihat,i_2\}$ and an index~$i_3\in[n]\setminus\{1,\ihat,i_4,i_1\}$, and they will satisfy $x_{i_1}\neq 4$ and $y_{i_3}\neq 3$.
We choose multiset permutations $u^j,v^j\in \Pi(\ba)$, $j=1,2,3,4$, satisfying the following conditions; see Figure~\ref{fig:H2''}~(a):
\begin{itemize}[leftmargin=8mm, noitemsep]
\item[(i)] $u^1=x$, $v^1_{i_1}=x_{i_2}=4$, and $v^4=y$, $u^4_{i_3}=y_{i_4}=3$;
\item[(ii)] $u^2_1=v^1_\ihat=1$, $u^2_\ihat=v^1_1=3$, $u^3_1=v^2_\ihat=3$, $u^3_\ihat=v^2_1=4$, $u^4_1=v^3_\ihat=4$, $u^4_\ihat=v^3_1=2$, and $u^j_i=v^{j-1}_i$ for $j=2,3,4$ and all $i\in[n]\setminus\{1,\ihat\}$;
\item[(iii)] $(u^2_{i_1},v^2_{i_1})=(4,1)$, $(u^3_{i_3},v^3_{i_3})=(2,3)$.
\end{itemize}
Note that $u^1\neq v^1$ and $u^4\neq v^4$ by~(i) and the choice of $i_1,i_2$ and $i_3,i_4$, respectively, and $u^j_1\neq v^j_1$ for $j=2,3$ by~(ii).
Consequently, we have $u^j\neq v^j$ for $j=1,2,3,4$.

We consider a Hamilton path~$P_1$ in~$G^{\ihat,1}(\ba)$ from~$u^1$ to~$v^1$, and a Hamilton path~$P_4$ in~$G^{\ihat,2}(\ba)$ from~$u^4$ to~$v^4$, which exist by the assumption~$G(\alpha,\alpha-1,1,1)\in\cH$.
We also consider a Hamilton path~$P_2$ in~$G^{\ihat,3}(\ba)$ from~$u^2$ to~$v^2$, which exists by the assumption~$G(\alpha,\alpha,1)\in\cL_{12}$, using that $(u^2_1,v^2_1)=(1,4)$ and $(u^2_{i_1},v^2_{i_1})=(4,1)=(p_4(1,4),q_4(1,4))$ by~(ii)+(iii) and~\eqref{eq:pq}.
Note here that in~$G^{\ihat,3}(\ba)$, the symbol~4 takes the role of~3 in~$G(\alpha,\alpha,1)$.
Finally, we consider a Hamilton path~$P_3$ in~$G^{\ihat,4}(\ba)$ from~$u^3$ to~$v^3$, which exists by the assumption~$G(\alpha,\alpha,1)\in\cL_{12}$, using that $(u^3_1,v^3_1)=(3,2)$ and $(u^3_{i_3},v^3_{i_3})=(2,3)=(q_3(2,3),p_3(2,3))$ by~(ii)+(iii) and~\eqref{eq:pq}.
The concatenation $P_1P_2P_3P_4$ is a Hamilton path in~$G(\ba)$ from~$x$ to~$y$.

\begin{figure}
\includegraphics[page=6]{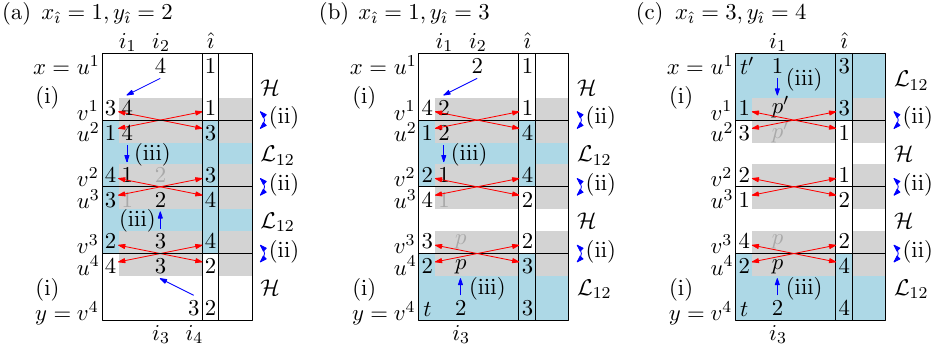}
\caption{Illustration of the proof of Lemma~\ref{lem:H2''}.}
\label{fig:H2''}
\end{figure}

We now consider the case~$(x_\ihat,y_\ihat)=(1,3)$.
We define $t:=y_1\in\{1,2,4\}$ and $p:=p_4(2,t)\in\{1,4\}$.
Let $i_1,i_2\in[n]\setminus\{1,\ihat\}$ be such that $x_{i_1}\neq 2$ and $x_{i_2}=2$.
Similarly, let $i_3\in[n]\setminus\{1,\ihat\}$ be such that $y_{i_3}=2$.
We choose multiset permutations $u^j,v^j\in \Pi(\ba)$, $j=1,2,3,4$, satisfying the following conditions; see Figure~\ref{fig:H2''}~(b):
\begin{itemize}[leftmargin=8mm, noitemsep]
\item[(i)] $u^1=x$, $v^1_{i_1}=x_{i_2}=2$, and $v^4=y$;
\item[(ii)] $u^2_1=v^1_\ihat=1$, $u^2_\ihat=v^1_1=4$, $u^3_1=v^2_\ihat=4$, $u^3_\ihat=v^2_1=2$, $u^4_1=v^3_\ihat=2$, $u^4_\ihat=v^3_1=3$ and $u^j_i=v^{j-1}_i$ for $j=2,3,4$ and all $i\in[n]\setminus\{1,\ihat\}$;
\item[(iii)] $(u^2_{i_1},v^2_{i_1})=(2,1)$ and $u^4_{i_3}=p$.
\end{itemize}
Note that $u^1\neq v^1$ by~(i) and the choice of~$i_1,i_2$, $u^j_1\neq v^j_1$ for $j=2,3$ by~(ii), and $u^4\neq v^4$ by~(iii), the choice of~$i_3$ and the fact that $p\neq 2$.
Consequently, we have $u^j\neq v^j$ for $j=1,2,3,4$.

We consider a Hamilton path~$P_1$ in~$G^{\ihat,1}(\ba)$ from~$u^1$ to~$v^1$, and a Hamilton path~$P_3$ in~$G^{\ihat,2}(\ba)$ from~$u^3$ to~$v^3$, which exist by the assumption~$G(\alpha,\alpha-1,1,1)\in\cH$.
We also consider a Hamilton path~$P_2$ in~$G^{\ihat,4}(\ba)$ from~$u^2$ to~$v^2$, which exists by the assumption~$G(\alpha,\alpha,1)\in\cL_{12}$, using that $(u^2_1,v^2_1)=(1,2)$ and $(u^2_{i_1},v^2_{i_1})=(2,1)=(p_3(1,2),q_3(1,2))$ by~(ii)+(iii) and~\eqref{eq:pq}.
Finally, we consider a Hamilton path~$P_3$ in~$G^{\ihat,3}(\ba)$ from~$u^4$ to~$v^4$, which exists by the assumption~$G(\alpha,\alpha,1)\in\cL_{12}$, using that $(u^4_1,v^4_1)=(2,t)$ and $(u^4_{i_3},v^4_{i_3})=(p,2)=(p_4(2,t),q_4(2,t))$ by~(i)--(iii) and~\eqref{eq:pq}.
The concatenation $P_1P_2P_3P_4$ is a Hamilton path in~$G(\ba)$ from~$x$ to~$y$.

It remains to consider the case~$(x_\ihat,y_\ihat)=(3,4)$.
We define $t:=y_1\in\{1,2,3\}$ and $p:=p_3(2,t)\in\{1,3\}$.
Moreover, we define $t':=x_1\in\{1,2,4\}$ and $p':=p_4(1,t')\in\{2,4\}$.
Let $i_1\in[n]\setminus\{1,\ihat\}$ be such that $x_{i_1}=1$, and let $i_3\in[n]\setminus\{1,\ihat\}$ be such that $y_{i_3}=2$.
We choose multiset permutations $u^j,v^j\in \Pi(\ba)$, $j=1,2,3,4$, satisfying the following conditions; see Figure~\ref{fig:H2''}~(c):
\begin{itemize}[leftmargin=8mm, noitemsep]
\item[(i)] $u^1=x$ and $v^4=y$;
\item[(ii)] $u^2_1=v^1_\ihat=3$, $u^2_\ihat=v^1_1=1$, $u^3_1=v^2_\ihat=1$, $u^3_\ihat=v^2_1=2$, $u^4_1=v^3_\ihat=2$, $u^4_\ihat=v^3_1=4$ and $u^j_i=v^{j-1}_i$ for $j=2,3,4$ and all $i\in[n]\setminus\{1,\ihat\}$;
\item[(iii)] $v^1_{i_1}=p'$ and $u^4_{i_3}=p$.
\end{itemize}
Note that $u^j_1\neq v^j_1$ for $j=2,3$ by~(ii), and $u^1\neq v^1$ and $u^4\neq v^4$ by~(iii), the choice of~$i_1,i_3$ and the fact that $p'\neq 1$ and $p\neq 2$, respectively.
Consequently, we have $u^j\neq v^j$ for $j=1,2,3,4$.

We consider a Hamilton path~$P_2$ in~$G^{\ihat,1}(\ba)$ from~$u^2$ to~$v^2$, and a Hamilton path~$P_3$ in~$G^{\ihat,2}(\ba)$ from~$u^3$ to~$v^3$, which exist by the assumption~$G(\alpha,\alpha-1,1,1)\in\cH$.
We also consider a Hamilton path~$P_1$ in~$G^{\ihat,3}(\ba)$ from~$u^1$ to~$v^1$, which exists by the assumption~$G(\alpha,\alpha,1)\in\cL_{12}$, using that $(u^1_1,v^1_1)=(t',1)$ and $(u^1_{i_1},v^1_{i_1})=(1,p')=(q_4(1,t'),p_4(1,t'))$ by~(i)--(iii) and~\eqref{eq:pq}.
Finally, we consider a Hamilton path~$P_4$ in~$G^{\ihat,4}(\ba)$ from~$u^4$ to~$v^4$, which exists by the assumption~$G(\alpha,\alpha,1)\in\cL_{12}$, using that $(u^4_1,v^4_1)=(2,t)$ and $(u^4_{i_3},v^4_{i_3})=(p,2)=(p_3(2,t),q_3(2,t))$ by~(i)--(iii) and~\eqref{eq:pq}.
The concatenation $P_1P_2P_3P_4$ is a Hamilton path in~$G(\ba)$ from~$x$ to~$y$.
\end{proof}

\begin{proof}[Proof of Lemma~\ref{lem:H1'}]
We proceed exactly as in the proof of Lemma~\ref{lem:H1} for $k=3$, choosing indices $\icheck,\ihat$ and distinguishing cases~(oa), (ob1) and~(ob2) depending on the values of~$y_\icheck$ and~$y_1$, as before.

\textbf{Case~(oa):} $y_\icheck=1$.
By the assumption that $x_\ihat,y_\ihat>1$ we can now assume w.l.o.g.\ that $x_\ihat=2$ and $y_\ihat=3$.
We define the permutations~$\pi:=(1,2,3)$ and~$\rho:=(2,1,3)$.
We set $t:=y_1\in\{1,2,3\}$ and $p:=p_3(1,t)\in\{2,3\}$, and we choose $u^j,v^j\in\Pi(\ba)$, $j=1,2,3$, and $\uhat^1,\vhat^1,\ucheck^1,\vcheck^1\in\Pi(\ba)$, as in case~(oa) in the proof of Lemma~\ref{lem:H1} for $k=3$, subject to the conditions~(ii)--(vi) as before and the following condition~(i') instead of~(i); see Figure~\ref{fig:H1'}~(oa):
\begin{itemize}[leftmargin=8mm, noitemsep]
\item[(i')] $u^1=x$, $v^1_{i_1}=x_{i_2}$, and $v^3=y$, $u^3_{i_3}=p$, where $i_3$ is chosen such that $y_{i_3}=1$.
\end{itemize}

\begin{figure}[b!]
\centerline{
\includegraphics[page=7]{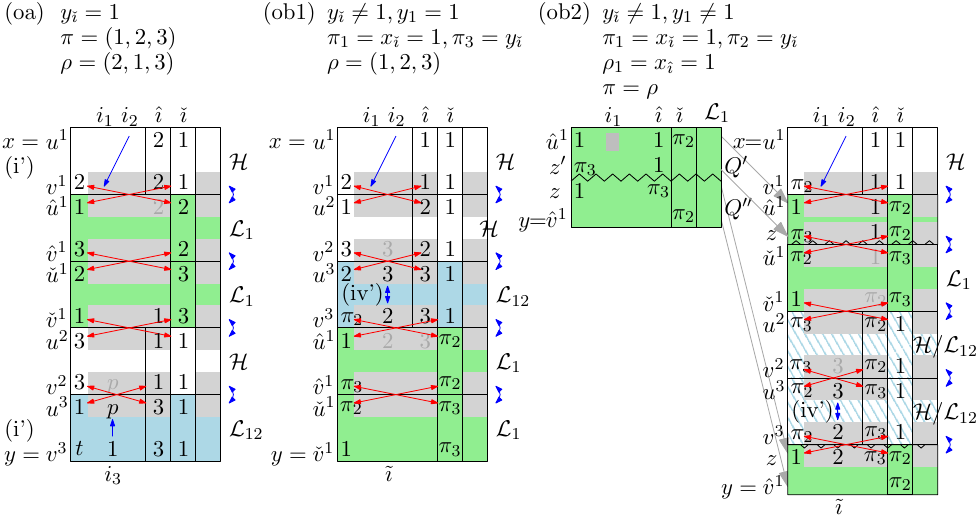}
}
\caption{Illustration of the proof of Lemma~\ref{lem:H1'}.}
\label{fig:H1'}
\end{figure}

For $j=1,2,3$ we consider a Hamilton path~$P_j$ in~$G^{(\icheck,\ihat),(1,\rho_j)}(\ba)$ from~$u^j$ to~$v^j$.
For $j=1,2$ such a path exists by the assumption $G(\varphi(\alpha-1,\alpha-2,2))\in\cH$.
For $j=3$ such a path exists by the assumption $G(\alpha-1,\alpha-1,1)\in\cL_{12}$, using that $(u^3_1,v^3_1)=(1,t)$ and $(u^3_{i_3},v^3_{i_3})=(p,1)=(p_3(1,t),q_3(1,t))$ by~(i') and~\eqref{eq:pq}.

We also consider a Hamilton path~$\Phat_1$ in~$G^{\icheck,\pi_2}(\ba)=G^{\icheck,2}(\ba)$ from $\uhat^1$ to~$\vhat^1$, which exists by the assumption $G(\varphi(\alpha,\alpha-2,2))\in\cL_1$, and a Hamilton path~$\Pcheck_1$ in~$G^{\icheck,\pi_3}(\ba)=G^{\icheck,3}(\ba)$ from $\ucheck^1$ to~$\vcheck^1$, which exists by the assumption $G(\alpha,\alpha-1,1)\in\cL_1$.

The concatenation $P_1\Phat_1\Pcheck_1 P_2P_3$ is a Hamilton path in~$G(\ba)$ from~$x$ to~$y$.

\textbf{Case~(ob1):} $y_\icheck\neq 1$ and $y_1=1$.
We let $\pi=(\pi_1,\pi_2,\pi_3)$ be the permutation defined by $\pi_1=x_\icheck=1$ and $\pi_3=y_\icheck$, and we define $\rho:=(1,2,3)$.

We choose $u^j,v^j\in\Pi(\ba)$, $j=1,2,3$, and $\uhat^1,\vhat^1,\ucheck^1,\vcheck^1\in\Pi(\ba)$, as in case~(ob1) in the proof of Lemma~\ref{lem:H1} for $k=3$, subject to the conditions~(i)--(iii) as before and the following condition~(iv') instead of~(iv); see Figure~\ref{fig:H1'}~(ob1):
\begin{itemize}[leftmargin=8mm, noitemsep]
\item[(iv')] $u^3_\itilde=3$ and $v^3_\itilde=2$ for some $\itilde\in[n]\setminus\{1,\icheck,\ihat\}$.
\end{itemize}

We consider Hamilton paths~$P_1,P_2,P_3$ and $\Phat_1$, $\Pcheck_1$ as before.
The path~$P_3$ from~$u^3$ to~$v^3$ exists by the assumption $G(\alpha-1,\alpha-1,1)\in\cL_{12}$, using that $(u^3_1,v^3_1)=(2,\pi_2)$ and $(u^3_\itilde,v^3_\itilde)=(3,2)=(p_3(2,\pi_2),q_3(2,\pi_2))$ by~(iv') and~\eqref{eq:pq}.

The concatenation $P_1P_2P_3\Phat_1\Pcheck_1$ is a Hamilton path in~$G(\ba)$ from~$x$ to~$y$.

\textbf{Case~(ob2):} $y_\icheck\neq 1$ and $y_1\neq 1$.
We let $\pi=(\pi_1,\pi_2,\pi_3)$ be the permutation defined by $\pi_1=x_\icheck=1$ and $\pi_2=y_\icheck$, and we define $\rho:=\pi$.
We choose $\uhat^1,\vhat^1\in\Pi(\ba)$ as in case~(ob2) in the proof of Lemma~\ref{lem:H1} for $k=3$, and we construct the Hamilton path~$Q$ and its subpaths~$Q'$ and~$Q''$ exactly as before.
Also, we choose $u^j,v^j\in\Pi(\ba)$, $j=1,2,3$, and $\ucheck^1,\vcheck^1\in\Pi(\ba)$, as in the previous proof, subject to conditions~(i)--(iii) as before and the following condition~(iv') instead of~(iv); see Figure~\ref{fig:H1'}~(ob2):
\begin{itemize}[leftmargin=8mm, noitemsep]
\item[(iv')] If $\pi_2=3$, we require that $u^2_\itilde=3$ and $v^2_\itilde=2$ for some $\itilde\in[n]\setminus\{1,\icheck,\ihat\}$ and condition~$(\psi_3)$.
If $\pi_3=3$, on the other hand, we require that $u^3_\itilde=3$ and $v^3_\itilde=2$ for some $\itilde\in[n]\setminus\{1,\icheck,\ihat\}$ with $z_\itilde=2$ and condition~$(\psi_2)$.
\end{itemize}

We consider Hamilton paths~$P_1,P_2,P_3$ and $\Pcheck_1$ as before.
If $j\in\{2,3\}$ is such that~$\pi_j=3$, then the path~$P_j$ from~$u^j$ to~$v^j$ exists by the assumption $G(\alpha-1,\alpha-1,1)\in\cL_{12}$, using that $(u^j_1,v^j_1)=(2,2)$ (clearly, $\pi_3=2$ if $\pi_2=3$ and $\pi_2=2$ if $\pi_3=3$) and $(u^j_\itilde,v^j_\itilde)=(3,2)=(p_3(2,2),q_3(2,2))$ by~(iv') and~\eqref{eq:pq}.

The concatenation $P_1Q'\Pcheck_1 P_2P_3Q''$ is a Hamilton path in~$G(\ba)$ from~$x$ to~$y$.
\end{proof}

\begin{proof}[Proof of Lemma~\ref{lem:H1''}]
We proceed exactly as in the proof of Lemma~\ref{lem:H1} for $k=4$, choosing indices $\icheck,\ihat$ and distinguishing cases~(eb2), (eb1) and~(ea) depending on the values of~$y_\icheck$ and~$y_1$, as before.

\textbf{Case~(eb2):} $y_\icheck\neq 1$ and $y_1\neq 1$.
If $y_\icheck=2$ we define the permutations $\pi:=(1,3,4,y_\icheck)=(1,3,4,2)$ and~$\rho:=(1,2,3,4)$, whereas if $y_\icheck\in\{3,4\}$ we define $\pi:=(1,7-y_\icheck,2,y_\icheck)$ and $\rho:=(1,3,4,2)$.
We choose $u^j,v^j\in\Pi(\ba)$, $j=1,2,3,4$, and $\uhat^j,\vhat^j\in\Pi(\ba)$, $j=1,2$, and $\ucheck^1,\vcheck^1\in\Pi(\ba)$, as in case~(eb2) in the proof of Lemma~\ref{lem:H1} for $k=4$, subject to the conditions~(i)--(iii) as before and the following condition~(iv') instead of~(iv); see Figure~\ref{fig:H1''}~(eb2):
\begin{itemize}[leftmargin=8mm, noitemsep]
\item[(iv')] For $j\in\{2,3\}$ such that $\rho_j=3$ we require that $u^j_\itilde=4$, $v^j_\itilde=2$, $u^{j+1}_\itilde=2$, and $v^{j+1}_\itilde=3$ for some $\itilde\in[n]\setminus\{1,\icheck,\ihat\}$, and condition~$(\psi_4)$ if $j=2$.
\end{itemize}

\begin{figure}
\centerline{
\includegraphics[page=9]{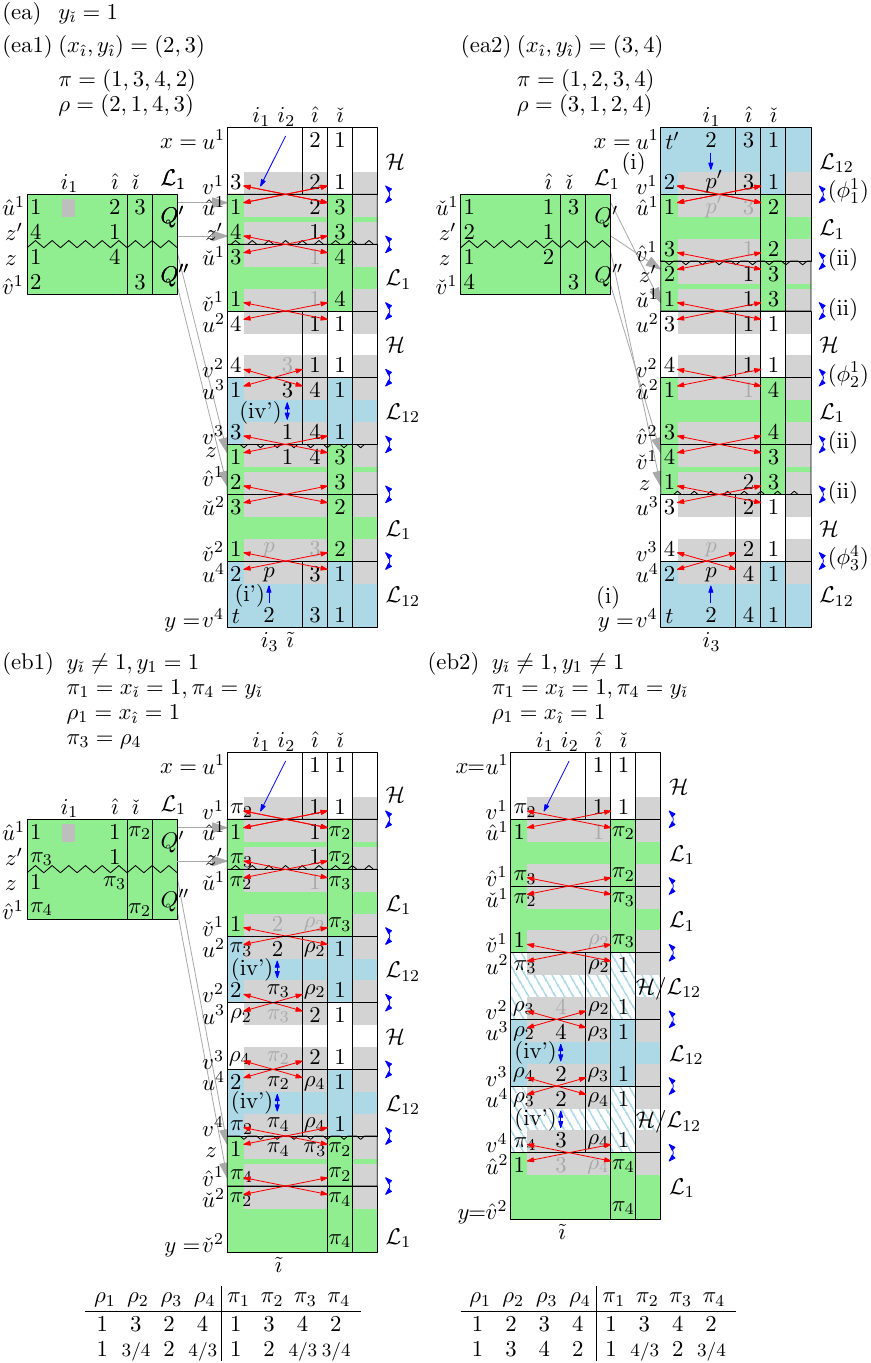}
}
\caption{Illustration of the proof of Lemma~\ref{lem:H1''}.}
\label{fig:H1''}
\end{figure}

For $j=1,2,3,4$ we consider a Hamilton path~$P_j$ in~$G^{(\icheck,\ihat),(1,\rho_j)}(\ba)$ from~$u^j$ to~$v^j$.
For the two values of~$j$ with $\rho_j\in\{1,2\}$ such a path exists by the assumption $G(\alpha-1,\alpha-2,1,1)\in\cH$.
For the two values of~$j$ with $\rho_j\in\{3,4\}$ such a path exists by the assumption $G(\alpha-1,\alpha-1,1)\in\cL_{12}$, using that $(u^j_1,v^j_1)=(2,4)$ and $(u^j_\itilde,v^j_\itilde)=(4,2)=(p_4(2,4),q_4(2,4))$ if $\rho_j=3$ and $(u^j_1,v^j_1)=(3,2)$ and $(u^j_\itilde,v^j_\itilde)=(2,3)=(q_3(2,3),p_3(2,3))$ if $\rho_j=4$, by~(iv') and~\eqref{eq:pq}.

For $j=1,2$ we consider a Hamilton path~$\Phat_j$ in~$G^{\icheck,\pi_{2j}}(\ba)$ from $\uhat^j$ to~$\vhat^j$ and a Hamilton path~$\Pcheck_1$ in~$G^{\icheck,\pi_3}(\ba)$ from $\ucheck^1$ to~$\vcheck^1$, which exist by the assumptions $G(\alpha,\alpha-2,1,1)\in\cL_1$ and $G(\alpha,\alpha-1,1)\in\cL_1$.

The concatenation $P_1\Phat_1\Pcheck_1 P_2P_3P_4\Phat_2$ is a Hamilton path in~$G(\ba)$ from~$x$ to~$y$.

\textbf{Case~(eb1):} $y_\icheck\neq 1$ and $y_1=1$.
If $y_\icheck=2$ we define the permutations $\pi:=(1,3,4,y_\icheck)=(1,3,4,2)$ and~$\rho:=(1,3,2,4)$, whereas if $y_\icheck\in\{3,4\}$ we define $\pi:=(1,2,7-y_\icheck,y_\icheck)$ and $\rho:=(1,y_\icheck,2,7-y_\icheck)$.
We choose $\uhat^1,\vhat^1\in\Pi(\ba)$ as in case~(eb1) in the proof of Lemma~\ref{lem:H1} for $k=4$, and we construct the Hamilton path~$Q$ and its subpaths~$Q'$ and~$Q''$ exactly as before.
Also, we choose $u^j,v^j\in\Pi(\ba)$, $j=1,2,3,4$, and $\ucheck^j,\vcheck^j\in\Pi(\ba)$, $j=1,2$, as in the previous proof, subject to the conditions~(i)--(iii) as before and the following condition~(iv') instead of~(iv); see Figure~\ref{fig:H1''}~(eb1):
\begin{itemize}[leftmargin=8mm, noitemsep]
\item[(iv')] We require that $u^2_\itilde=2$, $v^2_\itilde=\pi_3$, $u^4_\itilde=\pi_2$, and $v^4_\itilde=\pi_4$ for some $\itilde\in[n]\setminus\{1,\icheck,\ihat\}$ with $z_\itilde=\pi_4$.
\end{itemize}

We consider Hamilton paths~$P_1,P_2,P_3,P_4$ and $\Pcheck_1,\Pcheck_2$ as before.
The paths~$P_j$, $j\in\{2,4\}$, from~$u^j$ to~$v^j$ exist by the assumption $G(\alpha-1,\alpha-1,1)\in\cL_{12}$, using that $(u^2_1,v^2_1)=(\pi_3,2)$ and $(u^2_\itilde,v^2_\itilde)=(2,\pi_3)=(q_{\pi_3}(2,\pi_3),p_{\pi_3}(2,\pi_3))$, and that $(u^4_1,v^4_1)=(2,\pi_2)$ and $(u^4_\itilde,v^4_\itilde)=(\pi_2,\pi_4)$, which equals $(p_3(2,\pi_2),q_3(2,\pi_2))$ if~$\pi_2=3$ and $(q_{\pi_4}(\pi_2,2),p_{\pi_4}(\pi_2,2))$ if~$\pi_2=2$ by~(iv') and~\eqref{eq:pq}.

The concatenation $P_1Q'\Pcheck_1 P_2P_3P_4Q''\Pcheck_2$ is a Hamilton path in~$G(\ba)$ from~$x$ to~$y$.

\textbf{Case~(ea):} $y_\icheck=1$.
By the assumption that $x_\ihat,y_\ihat>1$ and the fact that the symbols~3 and~4 both occur exactly once, we can now assume w.l.o.g.\ that $(x_\ihat,y_\ihat)\in\{(2,3),(3,4)\}$, and we treat these two cases separately.

\textbf{Case~(ea1):} $(x_\ihat,y_\ihat)=(2,3)$.
We define the permutations~$\pi:=(1,3,4,2)$ and~$\rho:=(2,1,4,3)$, and we also define $t:=y_1\in\{1,2,4\}$ and $p:=p_4(2,t)\in\{1,4\}$.
We fix indices $i_1,i_2,i_3\in[n]\setminus\{1,\icheck,\ihat\}$ with $x_{i_1}\neq x_{i_2}$ and $y_{i_3}=2$.

We choose $\uhat^1,\vhat^1\in\Pi(\ba)$ as in case~(ea) in the proof of Lemma~\ref{lem:H1} for $k=4$, and we construct the Hamilton path~$Q$ and its subpaths~$Q'$ and~$Q''$ exactly as before.
Also, we choose $u^j,v^j\in\Pi(\ba)$, $j=1,2,3,4$, and $\ucheck^j,\vcheck^j\in\Pi(\ba)$, $j=1,2$, as in the previous proof, subject to the conditions~(ii)--(iii) as before and the following condition~(i') instead of~(i) and (iv') instead of~(iv); see Figure~\ref{fig:H1''}~(ea1):
\begin{itemize}[leftmargin=8mm, noitemsep]
\item[(i')] $u^1=x$, $v^4=y$, $u^4_{i_3}=p$, and conditions~$(\chi^1)$, $(\chi^2_3)$, and $(\chi^3)$;
\item[(iv')] We require that $u^3_\itilde=3$ and $v^3_\itilde=1$ for some $\itilde\in[n]\setminus\{1,\icheck,\ihat\}$ with $z_\itilde=1$.
\end{itemize}
We consider Hamilton paths~$P_1,P_2,P_3,P_4$ and $\Pcheck_1,\Pcheck_2$ as before.
The paths~$P_j$, $j\in\{3,4\}$, from~$u^j$ to~$v^j$ exist by the assumption $G(\alpha-1,\alpha-1,1)\in\cL_{12}$, using that $(u^3_1,v^3_1)=(1,3)$ and $(u^3_\itilde,v^3_\itilde)=(3,1)=(p_3(1,3),q_3(1,3))$ and $(u^4_1,v^4_1)=(2,t)$ and $(u^4_{i_3},v^4_{i_3})=(p,2)=(p_4(2,t),q_4(2,t))$ by~(i')+(iv') and~\eqref{eq:pq}.

\textbf{Case~(ea2):} $(x_\ihat,y_\ihat)=(3,4)$.
This case is treated very differently than case~(ea2) in the proof of Lemma~\ref{lem:H1}.
We define the permutations~$\pi:=(1,2,3,4)$ and~$\rho:=(3,1,2,4)$, and we also define $t:=y_1\in\{1,2,3\}$ and $p:=p_3(2,t)\in\{1,3\}$, and $t':=x_1\in\{1,2,4\}$ and $p':=p_4(2,t')\in\{1,4\}$.
We fix indices $i_1,i_3\in[n]\setminus\{1,\icheck,\ihat\}$ with $x_{i_1}=2$ and $y_{i_3}=2$.

We choose two multiset permutations $\ucheck^1,\vcheck^1\in\Pi(\ba)$ such that $\ucheck^1_1=1$, $\ucheck^1_\ihat=\rho_2=1$, $\ucheck^1_\icheck=\pi_3=3$, $\vcheck^1_1=\pi_4=4$, and $\vcheck^1_\icheck=\pi_3=3$.
We consider a Hamilton path~$Q$ in the graph~$G^{\icheck,\pi_3}(\ba)=G^{\icheck,3}(\ba)$ from~$\ucheck^1$ to~$\vcheck^1$, which exists by the assumption~$G(\alpha,\alpha-1,1)\in\cL_1$, using that~$\ucheck^1\in\Pi(\ba)^{1,1}$ and~$\vcheck^1\notin\Pi(\ba)^{1,1}$ (recall that $\ucheck^1_1=1$ and $\vcheck^1_1=4\neq 1$).
Let $z$ be the first vertex along this path from~$\ucheck^1$ to~$\vcheck^1$ for which $z_{\ihat}=\pi_2=2$, and let $z'$ be the predecessor of~$z$ on the path.
By the definition of~$z$ we have $z'_\ihat\neq\pi_2=2$, and consequently~$z'_1=\pi_2=2$.
By Lemma~\ref{lem:D0}, we thus obtain $z'_\ihat=z_1=1$.
Let $Q'$ be the \emph{backward} subpath of~$Q$ from~$z'$ to~$\ucheck^1$, and let $Q''$ be the \emph{backward} subpath of~$Q$ from~$\vcheck^1$ to~$z$.

We choose multiset permutations $u^j,v^j\in \Pi(\ba)$, $j=1,2,3,4$, and $\uhat^j,\vhat^j\in\Pi(\ba)$, $j=1,2$, satisfying the following conditions; see Figure~\ref{fig:H1''}~(ea2):
\begin{itemize}[leftmargin=8mm, noitemsep]
\item[(i)] $u^1=x$, $v^1_{i_1}=p'$, and $v^4=y$, $u^4_{i_3}=p$;
\item[(ii)] $\vhat^1_1=z'_\icheck=\pi_3=3$, $\vhat^1_\icheck=z'_1=\pi_2=2$, $\vhat^1_i=z'_i$,
$u^2_1=\ucheck^1_\icheck=\pi_3=3$, $u^2_\icheck=\ucheck^1_1=1$, $u^2_i=\ucheck^1_i$,
$\vhat^2_1=\vcheck^1_\icheck=\pi_3=3$, $\vhat^2_\icheck=\vcheck^1_1=\pi_4=4$, $\vhat^2_i=\vcheck^1_i$, and
$u^3_1=z_\icheck=\pi_3=3$, $u^3_\icheck=z_1=1$, $u^3_i=z_i$ for all $i\in[n]\setminus\{1,\icheck\}$;
\item[(iii)] conditions $(\phi^1_1)$, $(\phi^1_2)$ and~$(\phi^4_3)$.
\end{itemize}
The argument that all pairs $(u^j,v^j)$, $j=1,2,3,4$, and $(\uhat^j,\vhat^j)$, $j=1,2$, are distinct, is straightforward.

For $j=1,2,3,4$ we define $\bb:=(a_1-1,a_2,a_3,a_4)=(\alpha-1,\alpha-1,1,1)$, $\bc':=(a_1-1,\ldots,a_{\rho_j}-1,\ldots,a_4)$, and $\bc:=\varphi(\bc')\in\{(\alpha-1,\alpha-1,1)\,(\alpha-1,\alpha-2,1,1)\}$, which satisfy $\bc\precdot\bb\precdot\ba$ and $\Delta(\bb)=2$ and $\Delta(\bc)=1$, and we consider a Hamilton path~$P_j$ in~$G^{(\icheck,\ihat),(1,\rho_j)}(\ba)\simeq G(\bc)$ from~$u^j$ to~$v^j$.
For $j\in\{2,3\}$ these paths exist by the assumption~$G(\alpha-1,\alpha-2,1,1)\in\cH$.
For $j\in\{1,4\}$ these paths exist by the assumption~$G(\alpha-1,\alpha-1,1)\in\cL_{12}$, using that $(u^1_1,v^1_1)=(t',2)$ and $(u^1_{i_1},v^1_{i_1})=(2,p')=(q_4(2,t'),p_4(2,t'))$ and $(u^4_1,v^4_1)=(2,t)$ and $(u^4_{i_3},v^4_{i_3})=(p,2)=(p_3(2,t),q_3(2,t))$ by~(i) and~\eqref{eq:pq}.

For $j=1,2$ we define $\ba':=(a_1,\ldots,a_{\pi_{2j}}-1,\ldots,a_4)$ and $\bb:=\varphi(\ba')\in\{(\alpha,\alpha-2,1,1),(\alpha,\alpha-1,1)\}$, which satisfies $\bb\precdot\ba$ and $\Delta(\bb)=0$, and we consider a Hamilton path~$\Phat_j$ in the graph~$G^{\icheck,\pi_{2j}}(\ba)\simeq G(\bb)$ from~$\uhat^j$ to~$\vhat^j$, which exists by the assumption~$G(\bb)\in\cL_1$, using that~$\uhat^j\in\Pi(\ba)^{1,1}$ and~$\vhat^j\notin\Pi(a)^{1,1}$ (recall that $\uhat^j_1=1$ and $\vhat^j_1=3\neq 1$ by~(ii)+(iii)).

The concatenation $P_1\Phat_1 Q' P_2\Phat_2 Q'' P_3P_4$ is a Hamilton path in~$G(\ba)$ from~$x$ to~$y$.
\end{proof}

\end{document}